\documentclass[a4paper]{amsart}

\usepackage{amssymb,pstricks,amscd}
\usepackage{setspace}

\usepackage[all]{xy}
\usepackage[totalwidth=17cm,totalheight=24cm]{geometry}
\usepackage{graphicx}
\input psfig

\begin{document}

\newtheorem{cor}{Corollary}[section]
\newtheorem{theorem}[cor]{Theorem}
\newtheorem{prop}[cor]{Proposition}
\newtheorem{lemma}[cor]{Lemma}
\theoremstyle{definition}
\newtheorem{defi}[cor]{Definition}
\theoremstyle{remark}
\newtheorem{remark}[cor]{Remark}
\newtheorem{example}[cor]{Example}

\newcommand{\cD}{{\mathcal D}}
\newcommand{\FF}{{\mathcal F}}
\newcommand{\cH}{{\mathcal H}}
\newcommand{\cL}{{\mathcal L}}
\newcommand{\cM}{{\mathcal M}}
\newcommand{\cT}{{\mathcal T}}
\newcommand{\cML}{{\mathcal M\mathcal L}}
\newcommand{\cGH}{{\mathcal G\mathcal H}}
\newcommand{\C}{{\mathbb C}}
\newcommand{\D}{{\mathbb D}}
\newcommand{\N}{{\mathbb N}}
\newcommand{\R}{{\mathbb R}}
\newcommand{\Z}{{\mathbb Z}}
\newcommand{\Kt}{\tilde{K}}
\newcommand{\Mt}{\tilde{M}}
\newcommand{\dr}{{\partial}}
\newcommand{\tr}{\mbox{tr}}
\newcommand{\hol}{\mbox{Hol}}
\newcommand{\isom}{\mbox{Isom}}
\newcommand{\isomz}{\mbox{Isom}_{0,+}}
\newcommand{\vect}{\mbox{Vect}}
\newcommand{\kappab}{\overline{\kappa}}
\newcommand{\pib}{\overline{\pi}}
\newcommand{\Sigmab}{\overline{\Sigma}}
\newcommand{\gd}{\dot{g}}
\newcommand{\diff}{\mbox{Diff}}
\newcommand{\dev}{\mbox{dev}}
\newcommand{\devb}{\overline{\mbox{dev}}}
\newcommand{\devt}{\tilde{\mbox{dev}}}
\newcommand{\vol}{\mbox{Vol}}
\newcommand{\hess}{\mbox{Hess}}
\newcommand{\db}{\overline{\partial}}
\newcommand{\gammab}{\overline{\gamma}}
\newcommand{\Sigmat}{\tilde{\Sigma}}
\newcommand{\mut}{\tilde{\mu}}
\newcommand{\phit}{\tilde{\phi}}

\newcommand{\cunc}{{\mathcal C}^\infty_c}
\newcommand{\cun}{{\mathcal C}^\infty}
\newcommand{\dd}{d_D}
\newcommand{\dmin}{d_{\mathrm{min}}}
\newcommand{\dmax}{d_{\mathrm{max}}}
\newcommand{\Dom}{\mathrm{Dom}}
\newcommand{\dn}{d_\nabla}
\newcommand{\ded}{\delta_D}
\newcommand{\delmin}{\delta_{\mathrm{min}}}
\newcommand{\delmax}{\delta_{\mathrm{max}}}
\newcommand{\hmin}{H_{\mathrm{min}}}
\newcommand{\maxi}{\mathrm{max}}
\newcommand{\oL}{\overline{L}}
\newcommand{\oP}{{\overline{P}}}
\newcommand{\Ran}{\mathrm{Ran}}
\newcommand{\tgamma}{\tilde{\gamma}}
\newcommand{\cotan}{\mbox{cotan}}
\newcommand{\lambdat}{\tilde\lambda}
\newcommand{\St}{\tilde S}

\newcommand{\II}{I\hspace{-0.1cm}I}
\newcommand{\III}{I\hspace{-0.1cm}I\hspace{-0.1cm}I}
\newcommand{\HSt}{\tilde{\operatorname{HS}}}
\newcommand{\note}[1]{\marginpar{\tiny #1}}

\newcommand{\op}{\operatorname}

\newcommand{\AdS}{\operatorname{AdS}}
\newcommand{\uAdS}{\widetilde{\operatorname{AdS}}}
\newcommand{\dS}{\operatorname{dS}}
\newcommand{\HH}{\mathbb H}
\newcommand{\PP}{\mathbb P}
\newcommand{\RR}{\mathbb R}
\newcommand{\uRP}{\widetilde{\R\PP}^1}
\newcommand{\rp}{\R\PP}
\newcommand{\HS}{\operatorname{HS}}
\newcommand{\SO}{\operatorname{SO}}
\newcommand{\cF}{\operatorname{\mathcal{F}}}
\newcommand{\kD}{\mathfrak{D}}

\title[Collisions of particles]{Collisions of particles in locally AdS spacetimes I\\
Local description and global examples}

\author[]{Thierry Barbot}
\address{Laboratoire d'analyse non lin\'eaire et g\'eom\'etrie\\
Universit\'e d'Avignon et des pays de Vaucluse\\
33, rue Louis Pasteur\\
F-84 018 AVIGNON
}
\email{thierry.barbot@univ-avignon.fr}
\author[]{Francesco Bonsante}
\address{Dipartimento di Matematica dell'Universit\`a degli Studi di Pavia,
via Ferrata 1, 27100 Pavia (ITALY)}
\email{francesco.bonsante@unipv.it}
\author[]{Jean-Marc Schlenker}
\address{Institut de Math\'ematiques de Toulouse, 
UMR CNRS 5219 \\
Universit\'e Paul Sabatier\\
31062 Toulouse Cedex 9}
\email{schlenker@math.univ-toulouse.fr}
\thanks{T. B. and F. B. were partially supported by CNRS, ANR GEODYCOS. J.-M. S. was
partially supported by the A.N.R. programs RepSurf,
ANR-06-BLAN-0311, GeomEinstein, 06-BLAN-0154, and ETTT, 2009-2013}

\date{\today}

\begin{abstract}
We investigate 3-dimensional globally hyperbolic AdS manifolds containing ``particles'',
i.e., cone singularities along a graph $\Gamma$. We impose physically relevant conditions on
the cone singularities, e.g. positivity of mass (angle less than $2\pi$ on time-like
singular segments). We construct examples of such manifolds, describe the cone singularities
that can arise and the way they can interact (the local geometry near the vertices of $\Gamma$).
We then adapt to this setting some notions like global hyperbolicity which are natural
for Lorentz manifolds, and construct some examples of globally hyperbolic AdS manifolds
with interacting particles.
\end{abstract}

\maketitle

\tableofcontents

\section{Introduction}

\subsection{Three-dimensional cone-manifolds}

The 3-dimensional hyperbolic space can be defined as a quadric in the 4-dimensional
Minkowski space:
$$ H^3 = \{ x\in \R^{3,1}~|~\langle x,x\rangle = -1 \& x_0>0 \}~. $$
Hyperbolic manifolds, which are manifolds with a Riemannian metric locally isometric
to the metric on $H^3$, have been a major focus of attention for modern geometry. 

More recently attention has turned to hyperbolic cone-manifolds, which are the 
types of singular hyperbolic manifolds that one can obtain by gluing isometrically
the faces of hyperbolic polyhedra. Three-dimensional hyperbolic cone-manifolds 
are singular along lines, and at ``vertices'' where three or more singular
segments intersect. The local geometry at a singular vertex is determined by
its {\it link}, which is a spherical surface with cone singularities.
Among key recent results on hyperbolic cone-manifolds are rigidity results 
\cite{HK,mazzeo-montcouquiol,weiss:09} as well as many applications to three-dimensional
geometry (see e.g. \cite{bromberg1,brock-bromberg-evans-souto}).

\subsection{AdS manifolds}

The three-dimensional anti-de Sitter (AdS) space can be defined, similarly
as $H^3$, as a quadric in the 4-dimensional flat space of signature $(2,2)$:
$$ AdS_3 = \{ x\in \R^{2,2}~|~\langle x,x\rangle = -1\}~. $$
It is a complete Lorentz space of constant curvature $-1$, with fundamental
group $\Z$.

AdS geometry provides in certain ways a Lorentz analog of hyperbolic geometry,
a fact mostly discovered by Mess (see \cite{mess,mess-notes}). In particular,
the so-called {\it globally hyperbolic} AdS 3-manifolds are in key ways 
analogs of quasifuchsian hyperbolic 3-manifolds. Among the striking singularities
one can note an analog of the Bers double uniformization theorem for globally
hyperbolic AdS manifolds, or a similar description of the convex core and of
its boundary. Three-dimensional AdS geometry, like 3-dimensional hyperbolic
geometry, has some deep relationships with Teichm\"uller theory (see e.g.
\cite{mess,mess-notes,cone,bks,earthquakes,maximal}).

Lorentz manifolds have often been studied for reasons related to physics and
in particular gravitation. In three dimensions, Einstein metrics are the same
as constant curvature metrics, so constant curvature 3-dimensional Lorentz
manifolds -- and in particular AdS manifolds -- are 3-dimensional models of 
gravity in four dimensions. From this point of view, cone singularities 
have been extensively used to model point particles, see e.g. \cite{thooft1,thooft2}.

The goal pursued here is to start a geometric study of 3-dimensional AdS manifolds
with cone singularities. We will in particular
\begin{itemize}
\item describe the possible ``particles'', or 
cone singularities along a singular line,
\item describe the singular vertices -- the way those 
``particles'' can ``interact'',
\item show that classical notions like global hyperbolicity can be extended
to AdS cone-manifolds,
\item give examples of globally hyperbolic AdS particles with ``interesting'' 
particles and particle interactions.
\end{itemize}
We outline in more details those main contributions below.

\subsection{A classification of cone singularities along lines}

We start in Section \ref{ssc:singular} an analysis of the possible local geometry near a
singular point. For hyperbolic cone-manifold this local geometry is 
described by the {\it link} of the point, which is a spherical surface
with cone singularities. In the AdS setting there is an analog notion of
link, which is now what we call a singular {\it $HS$-surface}, that is, a
surface with a geometric structure locally modelled on the space of 
rays starting from a point in $\R^{2,1}$ (see Section \ref{sub:HSsingular}).

We then describe the possible geometry in the neighborhood of a point on
a singular segment (Proposition~\ref{pro:classising}).
For hyperbolic cone-manifolds, this local description is quite simple: 
there is only one possible local model, depending on only one parameter, the angle.
For AdS cone-manifolds -- or more generally cone manifolds with a constant 
curvature Lorentz metric -- the situation is more complicated, and cone 
singularities along segments can be of different types. For instance it is
clear that the fact that the singular segment is space-like, time-like or
light-like should play a role.

There are two physically natural restrictions which appear in 
this section. The first is the {\it degree} of a cone singularity along
a segment $c$: the 
number of connected components of time-like vectors in the normal bundle
of $c$ (Section~\ref{sub:HSsingularity}). In the ``usual'' situation where each point has a past and a 
future, this degree is equal to $2$. We restrict our study to the case
where the degree is at most equal to $2$. There are many interesting
situations where this degree can be strictly less than $2$, see below.

The second condition (see Section~\ref{sub.futpast}) is that each point should have a neighborhood containing
no closed causal curve -- also physically relevant since closed causal curves
induce causality violations. AdS manifolds with cone singularities satisfying
those two conditions are called {\it causal} here.
We classify and describe all cone singularities along segments
in causal AdS manifolds with cone singularities, and provide a short description
of each kind. They are called here: massive particles, tachyons, Misner
singularities, BTZ-like singularities, and light-like and extreme BTZ-like
singularities.

We also define a notion of {\it positivity} for those cone singularities 
along lines. Heuristically, positivity means that those geodesics tend 
to ``converge'' along those cone singularitites; for instance, for a 
``massive particle'' -- a cone singularity along a time-like singularity --
positivity means that the angle should be less than $2\pi$, and it 
corresponds physically to the positivity of mass.

\subsection{Interactions and convex polyhedra}

In Section \ref{sc:polyhedra} we turn our attention to the vertices of the singular locus of AdS
manifolds with cone singularities, in other terms the ``interaction points''
where several ``particles'' -- cone singularities along lines -- meet and
``interact''. The construction of the link as an $HS$-surface, in Section 
\ref{ssc:singular}, 
means that we need to understand the geometry of singular $HS$-surfaces. 
The singular lines arriving at an interaction point $p$ correspond to the
singular points of the link of $p$. An important point is that the positivity of the
singular lines arriving at $p$, and the absence of closed causal curves near $p$,
can be read directly on the link, this leads to a natural notion of {\it causal}
singular $HS$-surface, those causal singular $HS$-surfaces are precisely those
occuring as links of interaction points in causal singular AdS manifolds. 

The first point of Section \ref{sc:polyhedra} is the construction of many examples of positive causal
singular $HS$-surfaces from convex polyhedra in $\HS^3$, the natural analog of
$\HS^2$ in one dimension higher. Given a convex polyhedron in $\HS^3$ one can
consider the induced geometric structure on its boundary, and it is often an
$HS$-structure and without closed causal curve. Moreover the positivity 
condition is always satisfied. This makes it easy to visualize many examples
of causal $HS$-structures, and should therefore help in following the arguments
used in Section \ref{sec.classificationHS} to classify causal $HS$-surfaces.

However the relation between causal $HS$-surfaces and convex polyhedra is
perhaps deeper than just a convenient way to construct examples. This is
indicated in Theorem \ref{tm:converse}, which shows that all $HS$-surfaces
having some topological properties (those which are ``causally regular'') 
are actually obtained as induced on a unique convex polyhedron in $\HS^3$.

\subsection{A classification of $HS$-structures}

Section \ref{sec.classificationHS} contains a classification of causal $HS$-structures, or, in other
terms, of interaction points in causal singular AdS manifolds. The main 
result is Theorem \ref{tm:thierry}, which describes what types of interactions
can, or cannot, occur. The strinking point is that there are geometric 
restrictions on what kind of singularities along segments can interact at
one point. 

\subsection{Global hyperbolicity}

In Section \ref{sc:hyperbolicity} we consider singular AdS manifolds globally. We first
extend to this setting the notion of global hyperbolicity which plays
an important role in Lorentz geometry.

A key result for non-singular AdS manifolds is the existence, for any
globally hyperbolic manifold $M$, of a unique maximal globally hyperbolic
extension. We prove a similar result in the singular context (see
Proposition \ref{pr:extension1} and Proposition \ref{pr:extension2}). However this maximal
extension is unique only under the condition that the extension does not contain more
interactions than $M$.

\subsection{Construction of global examples}

Finally Section \ref{sc:examples} is intented to convince the reader that the general considerations
on globally hyperbolic AdS manifolds with interacting particles are not empty: 
it contains several examples, constructed using basically two methods. 

The first
relies again on 3-dimensional polyhedra, but not used in the same way as in Section 
\ref{sc:polyhedra}:
here we glue their faces isometrically so as to obtain cone singularities along the
edges, and interactions points at the vertices. The second method is based on 
surgery: we show that, in many situations, it is possible to excise a tube in
an AdS manifolds with non-interacting particles (like those arising in \cite{cone})
and replace it by a more interesting tube containing an interaction point.

\section{Preliminaries}

\subsection{$(G,X)$-structures}

Let $G$ be a Lie group, and $X$ an analytic space on which $G$ acts analytically and faithfully.
In this paper, we are essentially concerned with the case where $X=\AdS_3$ and $G$ its isometry group, but we will also
consider other pairs $(G,X)$.

A \textit{$(G,X)$-structure} on a manifold $M$ is a covering of $M$ by open sets with
homeomorphisms into $X$, such that the transition maps on the overlap of any
two sets are in $G$. A \textit{$(G,X)$-manifold} is a manifold equipped with a $(G,X)$-structure.
Observe that if $\tilde{X}$ denotes the universal covering of $X$, and $\tilde{G}$ the universal covering
of $G$, any $(G,X)$-structure defines a unique $(\tilde{G}, \tilde{X})$-structure, and, conversely,
any $(\tilde{G}, \tilde{X})$-structure defines a unique $(G,X)$-structure.

A $(G,X)$-manifold is characterized by its \textit{developing map} $\cD: \widetilde{M} \to X$ (where $\widetilde{M}$
denotes the universal covering of $M$) and the holonomy representation $\rho: \pi_1(M) \to G$.
This representation determines through the action of $G$ on $X$ a representation, still
denoted by $\rho$, of $\pi_1(M)$ on $X$.
The map $\cD$ is a local isometry, \textit{\textit{i.e.}} if expressed in local coordinates
provided by the lifted $(G,X)$-structure,
it is the restriction of an element of $G$. Moreover, it is a local homeomorphism, and also
$\pi_1(M)$-equivariant (where the action of $\pi_1(M)$ on $\widetilde{M}$ is the action
by deck transformations).

For more details, we refer to the recent expository paper \cite{goldmanexp}, 
or to the book \cite{carlip} oriented towards a physics audience.

\subsection{Background on the AdS space}
\label{sub:backgroundads}

Let $\R^{2,2}$ denotes the vector space $\R^4$ equipped with a quadratic form $q_{2,2}$ of signature
$(2,2)$.
The Anti-de Sitter $AdS_3$ space is defined as the locus in $\R^{2,2}$ of points
where the quadratic form takes the value $-1$,
endowed with the Lorentz metric induced by $q_{2,2}$.

On the Lie algebra ${\mathfrak gl}(2,\R)$ of $2\times2$ matrices with real coefficients,
the determinant defines a quadratic form of signature $(2,2)$. Hence we can
 consider the anti-de Sitter space $\AdS_3$ as the group
$\op{SL}(2,\R)$ equipped with its Killing metric, which is bi-invariant. There is therefore
an isometric action of $\op{SL}(2,\R)\times \op{SL}(2,\R)$ on $\AdS_3$, where the
two factors act by left and right multiplication, respectively. It is well known
(see \cite{mess}) that this yields an isomorphism between the identity component
$\isom_0(\AdS_3)$ of the isometry group of $\AdS_3$ and $\op{SL}(2,\R)\times \op{SL}(2,\R)/(-I,-I)$.
It follows directly that the identity component of the isometry group of $\AdS_{3,+}$
(the quotient of $\AdS_3$ by the antipodal map) is $\op{PSL}(2,\R)\times \op{PSL}(2,\R)$.
In all the paper, we denote by $\isomz$ the identity component of the
isometry group of $\AdS_{3,+}$, so that $\isomz$ is isomorphic to
$\op{PSL}(2,\R)\times \op{PSL}(2,\R)$.

Another way to identify the identity component of the isometry group of $\AdS_3$ is
by considering the projective model of $\AdS_{3,+}$, as the interior of a quadric
$Q\subset \R P^3$. This quadric is ruled by two families of lines, which we
call the ``left'' and ``right'' families and denote by $\cL_l,\cL_r$. Those two
families of lines have a natural projective structure (given for instance by
the intersection of the lines of $\cL_l$ with a fixed line of $\cL_r$). Given
an isometry $u\in \isomz$, it acts projectively on both $\cL_l$ and
$\cL_r$, defining two elements $\rho_l,\rho_r$ of $\op{PSL}(2,\R)$. This provides
an identification of $\isomz$ with
$\op{PSL}(2,\R)\times \op{PSL}(2,\R)$.

The projective space $\R P^3$ referred above is of course the projectivization of $\R^{2,2}$, and the elements of
the quadric $Q$ are the projections of $q_{2,2}$-isotropic vectors.
The geodesics of  $\AdS_{3,+}$ are the intersections between projective lines of $\R P^3$ and the interior of $Q$.
Such a projective line is the projection of a 2-plane $P$ in $\R^{2,2}$. If $P$ is space-like, the geodesic is
space-like. If $P$ is light-like (\textit{\textit{i.e.}}
the restriction of $q_{2,2}$ on $P$ is degenerate) then the geodesic is light-like. Finally,
if $P$ is time-like (\textit{\textit{i.e.}} contains a time-like direction) then the geodesic is time-like.

Similarly, totally geodesic planes are projections of 3-planes in $\R^{2,2}$. They can be space-like, light-like
or time-like. Observe that space-like planes in $\AdS_{3,+}$, with the induced metric, are isometric to the hyperbolic disk.
Actually, their images in the projective model of $\AdS_{3,+}$ are
Klein models of the hyperbolic disk.
Time-like planes in $\AdS_{3,+}$ are isometric to the anti-de Sitter space of dimension two.

Consider an affine chart of $\R P^3$, complement of the projection of a space-like hyperplane of $\R^{2,2}$.
The quadric in such an affine chart is a one-sheeted hyperboloid.
The interior of this hyperboloid is an \textit{affine chart} of $\AdS_3$.
The intersection of a geodesic of $\AdS_{3,+}$ with an affine chart
is a component of the intersection of the affine chart  with an affine line $\ell$.
The geodesic is space-like if $\ell$ intersects\footnote{of course, such an intersection may happen at the projective plane at infinity.} twice the hyperboloid, light-like
if $\ell$ is tangent to the hyperboloid, and time-like if $\ell$ avoids the hyperboloid.

For any $p$ in $\AdS_{3,+}$, the $q_{2,2}$-orthogonal $p^\perp$ is a space-like hyperplane. Its complement is therefore an
affine chart, that we denote by ${\mathcal A}(p)$. It is the \textit{affine chart centered at $p$.} Observe that ${\mathcal A}(p)$ contains $p$,
any non-time-like geodesic containing $p$ is contained in ${\mathcal A}(p)$.

Unfortunately, affine charts always miss some region of $\AdS_{3,+}$,
and we will consider regions of $\AdS_{3,+}$ which do not
fit entirely in such an affine chart. In this situation, one can consider
the conformal model: there is a conformal map from $\AdS_3$ to
$\D^2 \times {\mathbb S}^1$, equipped with the metric
$ds^2_0 - dt^2$ where $ds_0^2$ is the spherical metric on the disk $\D^2$, \textit{\textit{i.e.}} where $(\D^2, ds^2_0)$
is a hemisphere.

One needs also to consider the universal covering $\uAdS_3$.
It is conformally isometric to $\D^2 \times \R$ equipped
with the metric $ds^2_0 - dt^2$. But it is also advisable to consider it
as the union of an infinite sequence $({\mathcal A}_n)_{(n \in \Z)}$ of affine charts.
This sequence is totally ordered, every term lying in the future of the previous one
and in the past of the next one. These copies of affine charts are separated one from the other by a space-like plane,
\textit{\textit{i.e.}} a totally geodesic plane isometric to the hyperbolic disk.
Observe that each space-like or light-like geodesic of $\uAdS_3$ is contained in such an affine chart;
whereas each time-like geodesic
intersects every copy ${\mathcal A}_n$ of the affine chart.

If two time-like geodesics meet at some point $p$, then they meet infinitely many times. More precisely, there
is a point $q$ in $\uAdS_3$ such that if a time-like geodesic contains $p$, then it  contains $q$ also.
Such a point is said to be
\textit{conjugate to $p$.} The existence of conjugate points
corresponds to the fact that for
any $p$ in $\AdS_3 \subset \R^{2,2}$, every 2-plane containing $p$ contains also $-p$.
If we consider $\uAdS_3$ as the union of infinitely many copies ${\mathcal A}_n \quad (n \in \Z)$ of the affine chart ${\mathcal A}(p)$
centered at $p$, with ${\mathcal A}_0 = {\mathcal A}(p)$, then the points conjugate to $p$ are precisely the centers of the ${\mathcal A}_n$, all representing the same element
in the interior of the hyperboloid.

The center of ${\mathcal A}_1$ is  \textit{the first conjugate point $p^+$ of $p$ in the future.} It has the property that
any other point in the future of $p$ and conjugate to $p$ lies in the future of $p^+$. Inverting the time, one defines similarly
the \textit{first conjugate point $p^-$ of $p$ in the past} as the center of ${\mathcal A}_{-1}$.

Finally, the future in ${\mathcal A}_0$ of $p$ is the interior of a convex cone based at $p$ (more precisely, the interior of
the convex hull in $\R P^3$ of the union of $p$ with the space-like 2-plane between ${\mathcal A}_0$ and ${\mathcal A}_1$). The future
of $p$ in $\uAdS_3$ is the union of this cone with all the ${\mathcal A}_n$ with $n>0$.

In particular, one can give the following description of the domain $E(p)$,
intersection between the future of $p^-$ and the past of $p^+$:
its the union of ${\mathcal A}_0$, the past of $p^+$ in ${\mathcal A}_1$ and the future of $p^-$ in ${\mathcal A}_{-1}$.

We will need a similar description of 2-planes in $\uAdS_3$
(\textit{i.e.} of totally geodesic hypersurfaces) containing a given space-like geodesic.
Let $c$ be such a space-like geodesic, consider an affine chart ${\mathcal A}_0$
centered at a point in $c$ (therefore, $c$ is the segment
joining two points in the hyperboloid).
The set comprising first conjugate points in the future of points in $c$ is a space-like geodesic $c_+$,
contained in the chart ${\mathcal A}_1$. Every time-like 2-plane containing $c$ contains
also $c_+$, and \textit{vice versa.} The intersection between the future of $c$ and the past of $c_+$ is the union of:
\begin{itemize}
\item a wedge between two light-like half-planes both containing $c$ in their boundary,
\item a wedge between two light-like half-planes both containing $c_+$ in their boundary,
\item the space-like 2-plane between ${\mathcal A}_0$ and ${\mathcal A}_1$.
\end{itemize}

\section{Singularities in singular $\AdS$-spacetimes}
\label{ssc:singular}

In this paper, we require spacetimes to be oriented and time oriented. Therefore, by $\AdS$-spacetime we mean an
$(\op{Isom}_0(\AdS_3), \AdS_3)$-manifold.
In this section, we classify singular lines and singular points in singular $\AdS$-spacetimes.

In order to understand the notion of singularities, let us consider first the similar situation in the classical case of Riemannian geometric structures,
for example, of (singular) euclidean manifolds (see p. 523-524 of \cite{thurstonshape}). Locally, a singular point $p$
in a singular euclidean space is the intersection of various singular rays, the complement of these rays being
locally isometric to $\R^3$. The singular rays look as if they were geodesic rays. Since the singular space is assumed to have a manifold topology,
the space of rays, singular or not, starting from
$p$ is a topological $2$-sphere $L(p)$: 
the \textit{link} of $p$. Outside the singular rays, $L(p)$ is  locally modeled on the space of rays starting from a
point in the regular model, i.e. the $2$-sphere
${\mathbb S}^2$ equipped with its usual round metric. But this metric degenerates on the singular points of $L(p)$, i.e. the singular rays.
The way it may degenerate is described similarly: let $r$ be a singular point in $L(p)$ (a singular ray), and let $\ell(p)$
be the space of rays in $L(p)$ starting from $r$. It is a topological circle, locally modeled on the space $\ell_0$ of geodesic rays at a point
in the metric sphere ${\mathbb S}^2$. The space $\ell_0$ is naturally identified with the $1$-sphere ${\mathbb S}^1$ of perimeter $2\pi$,
and locally ${\mathbb S}^1$-structures on topological circles $\ell(p)$ are easily classified: they are determined by a positive real number,
the \textit{cone angle}, and $\ell(p)$ is isomorphic to $\ell_0$ if and only if this cone angle is $2\pi$.
Therefore, the link $L(p)$ is naturally equipped with a spherical metric with cone-angle singularities, and one easily
recover the geometry around $p$ by a fairly intuitive construction, the \textit{suspension} of $L(p)$.
We refer to \cite{thurstonshape} for further details.

Our approach in the $\AdS$ case is similar. The neighborhood of a singular point $p$ is the suspension of its link $L(p)$, this link
being a topological $2$-sphere equipped with a structure whose regular part is locally modeled on the link $\HS^2$ of a regular point
in $\AdS_3$, and whose singularities are suspensions of their links $\ell(r)$, which are circles locally modeled on the link
of a point in $\HS^2$.

However, the situation in the $\AdS$ case is much more intricate than in the euclidean case, since there is a bigger variety of singularity types  in $L(p)$:
a singularity in $L(p)$, i.e. a singular ray through $p$ can be time-like, space-like or light-like. Moreover, non-time-like lines may differ through the causal behavior
near them (for the definition of the future and past of a singular line, see section~\ref{sub.futpast}).

\begin{prop}
\label{pro:classising}
The various types of singular lines in $\AdS$ spacetimes are:
\begin{itemize}
\item \textbf{Time-like lines:} they correspond to massive particles (see section~\ref{sub:massiveparticle}).
\item \textbf{Light-like lines of degree $2$:} they correspond to photons (see Remark~\ref{rk:photon}).
\item \textbf{Space-like lines of degree $2$:} they correspond to tachyons (see section~\ref{sub.tachyon}).
\item \textbf{Future BTZ-like singular lines:} These singularities are characterized by the property that it is space-like, but has no future.
\item \textbf{Past BTZ-like singular lines:} These singularities are characterized by the property that it is space-like, but has no past.
\item \textbf{(Past or future) extreme BTZ-like singular lines:} they look like past/future BTZ-like singular lines, except that they are light-like.
\item \textbf{Misner lines:} they are space-like, but have no future and no past. Moreover, any neighborhood of the singular lines contains
closed time-like curves.
\item \textbf{Light-like or space-like lines of degree $k \geq 4$:} they can be described has $k/2$-branched cover over light-like or space-like
lines of degree $2$. They have the ``unphysical'' property of admitting a non-connected future.
\end{itemize}
\end{prop}

The several types of singular lines, as a not-so-big surprise, reproduce the several types of particles considered in physics. Some of these singularities
appear in the physics litterature, but, as far as we know, not all of them (for example, the terminology
\textit{tachyons,} that we feel adapted, does not seem to appear anywhere).


In section~\ref{sub:HS} we briefly present the space $\HS^2$ of rays through a point in $\AdS_3$. In section
\ref{sub:HSsurface}, we give the precise definition of regular HS-surfaces and their suspensions. In section~\ref{sub:HSsingularity}
we classify the circles locally modeled on links of points in $\HS^2$, i.e. of singularities of singular HS-surfaces which can then be defined
in the following section~\ref{sub:HSsingular}. In this section~\ref{sub:HSsingular}, we can state the definition of singular $\AdS$ spacetimes.

In section~\ref{sub:classiline}, we classify singular lines. In section~\ref{sub.futpast} we define and study the causality notion in singular AdS spacetimes. In particular we define
the notion of \textbf{causal HS-surface}, i.e. singular points admitting a neighborhood containing no closed causal curve. It is in this
section that we establish the description of the causality relation near the singular lines as stated in
Proposition~\ref{pro:classising}.

Finally, in section~\ref{sub:desgeom}, we provide a geometric description of each singular lines; in particular, we justify the ``massive particle'',
``photon'' and ``tachyon'' terminology.

\subsection{HS geometry}
\label{sub:HS}
%
Given a point $p$ in $\uAdS_3$, let $L(p)$ be the link of $p$, \textit{i.e.} the set of (non-parametrized)
oriented geodesic rays based at $p$. Since these rays are determined by their tangent vector at $p$
up to rescaling, $L(p)$ is naturally identified with the set of rays in $T_p\uAdS_3$.
Geometrically, $T_p\uAdS_3$ is a copy of Minkowski space $\R^{1,2}$. Denote by ${\operatorname{HS}}^2$
the set of geodesic rays into $\R^{1,2}$. It admits a natural decomposition in five subsets:
\begin{itemize}
\item the domains $\HH^2_+$ and $\HH^2_-$ comprising respectively future oriented and past oriented
time-like rays,
\item the domain $\dS^2$ comprising space-like rays,
\item the two circles $\partial\HH^2_+$ and $\partial\HH^2_-$, boundaries of $\HH^2_\pm$ in $\HS^2$.
\end{itemize}

The domains $\HH^2_\pm$ are the Klein models of the hyperbolic plane,
and $\dS^2$ is the Klein model of de Sitter space of dimension $2$. The group $\SO_0(1,2)$,
\textit{i.e.} the group of of time-orientation preserving and orientation preserving isometries of
$\R^{1,2}$, acts naturally (and projectively) on $\HS^2$, preserving this decomposition.

The classification of elements of $\SO_0(1,2) \approx \op{PSL}(2,\R)$ is presumably well-known
by most of the readers, but we stress out here that it is related to the $\HS^2$-geometry:
let $g$ be a non trivial element of $\SO_0(1,2)$.

\begin{itemize}
\item $g$ is \textit{elliptic} if and only if it admits exactly two fixed points, one in $\HH^2_+$, and the other (the opposite) in $\HH^2_-$,
\item $g$ is \textit{parabolic} if and only if it admits exactly two fixed points, one in $\partial\HH^2_+$, and the other (the opposite) in $\partial\HH^2_-$,
\item $g$ is \textit{hyperbolic} if and only it admits exactly 6 fixed points: a pair of opposite points in $\partial\HH^2_\pm$, and 2 pairs of opposite points
in $\dS^2$.
\end{itemize}

In particular, $g$ is elliptic (respectively hyperbolic) if and only if it admits a fixed in $\HH^2_\pm$ (respectively in $\dS^2$).

\subsection{Suspension of regular HS-surfaces}
\label{sub:HSsurface}
\begin{defi}
A HS-surface is a topological surface endowed with a $(\SO_0(1,2), \HS^2)$-structure.
\end{defi}

The $\SO_0(1,2)$-invariant orientation on $\HS^2$ induces an orientation on
every HS-surface. Similarly, the $\dS^2$ regions admit a canonical time orientation.
Hence any HS-surface is oriented, and its de Sitter regions are time oriented.

Given a HS-surface $\Sigma$, and once fixed a point $p$ in $\uAdS_3$,
we can construct a locally $\AdS$ manifold $e(\Sigma)$, called the suspension
of $\Sigma$, defined as follows:
\begin{itemize}
\item for any $v$ in $\HS^2 \approx L(p)$, let $r(v)$ be the geodesic ray issued from $p$
tangent to $v$. If $v$ lies in the closure of $\dS^2$, defines $e(v) := r(v)$; if
$v$ lies in $\HH^2_\pm$, let $e(v)$ be the portion of $r(v)$ between $p$ and the first
conjugate point $p^\pm$.
\item for any open subset $U$ in $\HS^2$, let $e(U)$ be the union of all $e(v)$ for $v$ in $U$.
\end{itemize}

Observe that $e(U) \setminus \{ p \}$ is an open domain in $\uAdS_3$,
and that $e(\HS^2)$ is the intersection $E(p)$ between the future of the first conjugate point in the past
and the past of the first conjugate point in the future (cf. the end of section~\ref{sub:backgroundads}).

The HS-surface $\Sigma$ can be understood as the disjoint union of open domains
$U_i$ in $\HS^2$, glued one to the other by coordinate change maps $g_{ij}$ given
by restrictions of elements of $\SO_0(1,2)$:

\[ g_{ij}: U_{ij} \subset U_j \to U_{ji} \subset U_i \]

But $\SO_0(1,2)$ can be considered
as the group of isometries of $\AdS_3$ fixing $p$. Hence every $g_{ij}$ induces
an identification between $e(U_{ij})$ and $e(U_{ji})$. Define $e(\Sigma)$
as the disjoint union of the $e(U_i)$, quotiented by the relation
identifying $q$ in $e(U_{ij})$ with $g_{ij}(q)$ in $e(U_{ji})$.
This quotient space contains a special point $\bar{p}$, represented in every $e(U_i)$
by $p$, and called the \textit{vertex} (we will sometimes abusively denote $\bar{p}$ by $p$).
The fact that $\Sigma$ is a surface implies that $e(\Sigma) \setminus \bar{p}$ is
a three-dimensional manifold, homeomorphic to $\Sigma \times \R$. The topological space
$e(\Sigma)$ itself is homeomorphic to the cone over $\Sigma.$  Therefore $e(\Sigma)$ is a manifold only
when $\Sigma$ is homeomorphic to the 2-sphere.
But it is easy to see that every HS-structure on the 2-sphere is isomorphic
to $\HS^2$ itself; and the suspension $e(\HS^2)$ is simply the regular AdS-manifold $E(p)$.

Hence in order to obtain singular $\AdS$-\textit{manifolds} that are not merely \textit{regular} $\AdS$-manifolds,
we need to consider (and define!) singular HS-surfaces.

\subsection{Singularities in singular HS-surfaces}
\label{sub:HSsingularity}
The classification of singularities in singular HS-surfaces essentially reduces (but not totally) to the
classification of  $\R\PP^1$-structures on the circle.

\subsubsection{Real projective structures on the circle}
\label{sub.RP1circle}
Let $\R\PP^1$ be the real projective line, and let $\widetilde{\R\PP}^1$ be its universal
covering. We fix a homeomorphism between $\widetilde{\R\PP}^1$ and the real line:
this defines an orientation and an order $<$ on $\widetilde{\R\PP}^1$.
Let $G$ be the group $\operatorname{PSL}(2, \R)$ of projective transformations
of $\R\PP^1$, and let $\tilde{G}$ be its universal covering: it is the group of
projective transformations of $\widetilde{\R\PP^1}$. We have an exact sequence:

\[ 0 \rightarrow \Z \rightarrow \tilde{G} \rightarrow G \rightarrow 0\]

Let $\delta$ be the generator of the center $\Z$ such that for every $x$ in $\widetilde{\R\PP}^1$ the
inequality $\delta x > x$ holds. The quotient of $\widetilde{\R\PP}^1$ by $\Z$ is projectively
isomorphic to $\R\PP^1$.

The elliptic-parabolic-hyperbolic classification of elements of $G$
induces a similar classification for elements in $\tilde{G}$, according
to the nature of their projection in $G$. Observe that non-trivial elliptic elements acts on
$\widetilde{\R\PP}^1$ as translations, \textit{i.e.} freely and properly discontinuously.
Hence the quotient space of their action is naturally a real projective structure on the circle.
We call these quotient spaces \textit{elliptic circles.} Observe that it includes
the usual real projective structure on $\R\PP^1$.

Parabolic and hyperbolic elements can all be decomposed as a product $\tilde{g}=\delta^kg$
where $g$ has the same nature (parabolic or hyperbolic) than $\tilde{g}$, but admits fixed
points in $\widetilde{\R\PP}^1$. The integer $k \in \Z$ is uniquely defined. Observe
that if $k \neq 0$, the action of $\tilde{g}$ on $\widetilde{\R\PP}^1$ is free and properly
discontinuous. Hence the associated quotient space, which is naturally equipped with a real
projective structure, is homeomorphic to the circle. We call it \textit{parabolic} or
\textit{hyperbolic circle,} according to the nature of $g$, \textit{of degree $k$.} Inverting
$\tilde{g}$ if necessary, we can always assume, up to a real projective isomorphism, that $k \geq 1$.

Finally, let $g$ be a parabolic or hyperbolic element of $\tilde{G}$ fixing a point $x_0$
in $\widetilde{\R\PP}^1$. Let $x_1$ be the unique fixed point of $g$ such that $x_1 > x_0$
and such that $g$ admits no fixed point between $x_0$ and $x_1$: if $g$ is parabolic, $x_1 = \delta x_0$;
and if $g$ is hyperbolic, $x_1$ is the unique $g$-fixed point in $] x_0, \delta x_0 [$.
Then the action of $g$ on $]x_0, x_1[$ is free and properly discontinuous, the quotient space
is a \textit{parabolic} or \textit{hyperbolic circle of degree $0$.}

These examples exhaust the list of real projective structures on the circle up to
real projective isomorphism. We briefly recall the proof: the developing map
$d: \R \to \widetilde{\R\PP}^1$ of a real projective structure on $\R/\Z$ is a local
homeomorphism from the real line into the real line, hence a homeomorphism onto its image $I$.
Let $\rho: \Z \to \tilde{G}$ be the holonomy morphism: being a homeomorphism, $d$ induces
a real projective isomorphism between the initial projective circle and $I/\rho(\Z)$.
In particular, $\rho(1)$ is non-trivial, preserves $I$, and acts freely and properly discontinuously on
$I$. An easy case-by-case study leads to a proof of our claim.

It follows that every cyclic subgroup of $\tilde{G}$ is the holonomy group of a real projective
circle, and that two such real projective circles are projectively isomorphic if and only if
their holonomy groups are conjugate one to the other. But some subtlety appears if one takes into consideration the
orientations:
usually, by real projective structure we mean a $(\operatorname{PGL}(2,\R), \R\PP^1)$-structure,
ie coordinate changes might reverse the orientation. In particular, two such structures are isomorphic
if there is a real projective transformation conjugating the holonomy groups, even if this transformation
reverses the orientation. But here, by $\R\PP^1$-circle we mean
a $(G, \R\PP^1)$-structure on the circle, with $G=\op{PSL}(2,\R)$. In particular, it admits a canonical orientation,
preserved by the holonomy group: the one whose lifting
to $\R$ is such that the developing map is orientation preserving. To be a $\R\PP^1$-isomorphism, a real projective conjugacy needs to preserve this orientation.

Let $L$ be a $\R\PP^1$-circle. Let $\gamma_0$ be the generator of $\pi_1(L)$ such that, for the canonical orientation defined above, and
for every $x$ in the image of the developing map:
\begin{equation}\label{eq>}
 \rho(\gamma_0)x > x
\end{equation}

Let $\rho(\gamma_0) =  \delta^kg$ be the decomposition such that $g$ admits fixed points in $\widetilde{\R\PP}^1$.
According to the inequality \eqref{eq>}, the degree $k$ is non-negative. Moreover:

\textbf{The elliptic case:}
Elliptic $\R\PP^1$-circles (\textit{i.e.} with elliptic holonomy) are uniquely parametrized
by a positive real number (the angle).

\textbf{The case $k \geq 1$:} Non-elliptic $\R\PP^1$-circles of degree $k \geq 1$ are uniquely parametrized by
the pair $(k, [g])$, where $[g]$ is a conjugacy class in $G$. Hyperbolic conjugacy classes are uniquely parametrized by
a positive real number: the absolute value of their trace. There are exactly two parabolic conjugacy classes: the \textit{positive parabolic class,} comprising the parabolic elements $g$ such that $gx \geq x$ for every $x$ in $\widetilde{\R\PP}^1$, and the
\textit{negative parabolic class,} comprising the parabolic elements $g$ such that $gx \leq x$ for every $x$ in $\widetilde{\R\PP}^1$
(this terminology is justified in {section}~\ref{sub.graviton}, and Remark~\ref{rk:positive/negative}).

\textbf{The case $k = 0$:} In this case, $L$ is isomorphic to the quotient by $g$ of a segment $]x_0, x_1[$
admitting as extremities two successive fixed points of $g$. Since we must have $gx > x$ for every $x$ in this segment,
$g$ cannot belong to the negative parabolic class: \textit{Every parabolic $\R\PP^1$-circle of degree $0$ is
positive.} Concerning the hyperbolic $\R\PP^1$-circles, the conclusion is the same as in the case $k \geq 1$:
they are uniquely parametrized by a positive real number. Indeed, given a hyperbolic element $g$ in
$\tilde{G}$, any $\R\PP^1$-circle of degree $0$ with holonomy $g$ is a quotient of a segment $]x_0, x_1[$ where
the left extremity $x_0$ is a repelling fixed point of $g$, and the right extremity an attractive fixed point.

\subsubsection{ HS-singularities}
For every $p$ in $\HS^2$, let $\ell(p)$ the link of $p$, \textit{i.e.} the space of rays in $T_p\HS^2$. Such a ray $v$ defines an
oriented projective line $c_v$ starting from $p$. Let $\Gamma_p$ be the \textit{pointwise} stabilizer in $\SO_0(1,2)\approx \op{PSL}(2, \R)$ of $p$,
\textit{i.e.} the subgroup which preserves every point of the ray $p$ in $\R^{1,2}$.

\begin{defi}
A $(\Gamma_p, \ell(p))$-circle is the data of a point $p$ in $\HS^2$ and a $(\Gamma_p, \ell(p))$-structure on the circle.
\end{defi}

Since $\HS^2$ is oriented, $\ell(p)$ admits a natural
$\R\PP^1$-structure, and thus every $(\Gamma_p, \ell(p))$-circle admits a natural underlying $\R\PP^1$-structure.

Given a $(\Gamma_p, \ell(p))$-circle $L$,
we construct a singular HS-surface $\mathfrak{e}(L)$: for every
element $v$ in the link of $p$, define $\mathfrak{e}(v)$ as the closed segment $[-p, p]$ contained
in the projective ray defined by $v$, where $-p$ is the antipodal point of $p$ in $\HS^2$,
and then operate as we did for defining the AdS space $e(\Sigma)$ associated to
a HS-surface. The resulting space $\mathfrak{e}(L)$ is topologically a sphere, locally modeled on
$\HS^2$ in the complement of two singular points corresponding to $p$ and $-p$.
These singular points will be typical singularities in HS-surfaces. Here, the singularity corresponding to $p$
as a prefered status, as representation a $(\Gamma_p, \ell(p))$-singularity.

There are several types of singularity, mutually non isomorphic:
\begin{itemize}

\item \textit{time-like singularities:} they correspond to the case where $p$ lies
in $\HH_\pm^2$. Then, $\Gamma_p$ is a 1-parameter elliptic subgroup of $G$, and
$L$ is an elliptic $\R\PP^1$-circle. When $p$ lies in $\HH^2_+$ (respectively $\HH^2_-$),
then the singularity is a {future (respectively past) time-like singularity.}

\item \textit{space-like singularities:} when $p$ lies in $\dS^2$, $\Gamma_p$
is generated by hyperbolic element of $\SO_0(1,2)$, and $L$ is a hyperbolic $\R\PP^1$-circle.

\item \textit{light-like singularities:} it is the case where $p$ lies in $\partial\HH^2_\pm$.
The stabilizer $\Gamma_p$ is generated by a parabolic element of $\SO_0(1,2)$,
and the link $L$ is a parabolic $\R\PP^1$-circle. We still have to
distinguish between past and future light-like singularities.

\end{itemize}

It is easy to classify time-like singularities up to (local) HS-isomorphisms: they are locally
characterized by their underlying structure of elliptic $\R\PP^1$-circle.
In other words, time-like singularities are nothing but the usual cone singularities of hyperbolic surfaces, since they
admit neighborhoods locally modeled on the Klein model of the hyperbolic disk.

But there are several types of space-like singularities, according to the causal structure around them.
More precisely: recall that every element $v$ of $\ell(p)$ is a ray in $T_p\HS^2$. 
Taking into account that $\dS^2$ is the Klein model of the 2-dimensional de Sitter space, it follows
that $v$, as a direction in a Lorentzian spacetime,
can be a time-like, light-like or space-like direction.
Moreover, in the two first cases, it can be future oriented or past oriented.

\begin{defi}
If $p$ lies in $\dS^2$, we denote by $i^+(\ell(p))$ (respectively $i^-(\ell(p))$) the set of future oriented (resp. past oriented)
directions.
\end{defi}

Observe that $i^+(\ell(p))$ and $i^-(\ell(p))$ are connected, and that their complement in $\ell(p)$ has two connected components.

This notion can be extended to time-like singularities:

\begin{defi}
If $p$ lies in $\partial\HH^2_+$, the domain $i^+(\ell(p))$ (respectively $i^-(\ell(p))$) is the set of directions $v$ such that $c_v(s)$ lies
in $\HH^2_+$ (respectively $\dS^2$) for $s$ sufficiently small.
Similarly, if $p$ lies in $\partial\HH^2_-$, the domain $i^-(\ell(p))$ (respectively $i^+(\ell(p))$) is the set of directions $v$ such that $c_v(s)$ lies
in $\HH^2_-$ (respectively $\dS^2$) for $s$ sufficiently small.
\end{defi}

In this situation, $i^+(\ell(p))$ and $i^-(\ell(p))$ are the connected components of the complement of the two points in $\ell(p)$
which are directions tangent to $\partial\HH^2_\pm$.

For time-like singularities, we simply define  $i^+(\ell(p)) = i^-(\ell(p)) = \emptyset$.

Finally, observe that the extremities of the arcs $i^\pm(\ell(p))$ are precisely the fixed points of $\Gamma_p$.

\begin{defi}
Let $L$ be a $(\Gamma_p, \ell(p))$-circle. Let $d: \tilde{L} \to \ell(p)$ the developing map.
The preimages $d^{-1}(i^+(\ell(p)))$ and $d^{-1}(i^-(\ell(p)))$ are open domain in $\tilde{L}$, preserved by the deck transformations.
Their projections in $L$ are denoted respectively by $i^+(L)$ and $i^-(L)$.
 \end{defi}

We invite the reader to convince himself that
the $\R\PP^1$-structure and the additional data of $i^\pm(L)$
determine the $(\Gamma_p, \ell(p))$-structure on the link,
hence the HS-singular point up to HS-isomorphism.

In the sequel, we present all the possible types of singularities, according to the position
in $\HS^2$ of the reference point $p$, and according to the degree of the underlying $\R\PP^1$-circle.
Some of them are called BTZ-like or Misner singularities; the reason of this terminology will be explained later
in section~\ref{sub:btz}, \ref{sub:misner}, respectively.

\begin{enumerate}

\item\textit{time-like singularities:}
We have already observed that they are easily classified: they can be considered as $\HH^2$-singularities. They
are characterized by their cone angle, and by their future/past quality.

\item\textit{space-like singularities of degree $0$:}
Let $L$ be a space-like singularity of degree $0$, \textit{i.e.} a $(\Gamma_p, \ell(p))$-circle
such that the underlying hyperbolic $\R\PP^1$-circle has degree $0$. Then  the holonomy of
$L$ is generated by a hyperbolic element $g$, and $L$ is isomorphic to the quotient of
an interval $I$ of $\ell(p)$ by the group $\langle g \rangle$ generated by $g$.
The extremities of $I$ are fixed points of $g$, therefore we have three possibilities:
\begin{itemize}
\item If $I=i^+(\ell(p))$,
then $L=i^+(L)$ and $i^-(L)=\emptyset$. The singularity is then called a \textit{BTZ-like past singularity}.
\item If $I=i^-(\ell(p))$, then $L=i^-(L)$ and $i^+(L)=\emptyset$.
The singularity is then called a \textit{BTZ-like future singularity}.
\item If $I$ is a component of $\ell(p) \setminus (i^+(\ell(p)) \cup i^-(\ell(p)))$,
then $L=i^+(L) = i^-(L)=\emptyset$. The singularity is a \textit{Misner singularity}.
\end{itemize}
\item\textit{light-like singularities of degree $0$:} When $p$ lies in $\partial\HH^2_+$,
and when the underlying parabolic $\R\PP^1$-circle
has degree $0$, then $L$ is the quotient of $i^+(\ell(p))$ or $i^-(\ell(p))$ by a parabolic element.
\begin{itemize}
\item If $I=i^+(\ell(p))$, then $L=i^+(L)$ and $i^-(L)=\emptyset$. The singularity is then called a \textit{future cuspidal singularity}. Indeed,
in that case, a neighborhood of the singular point in $\mathfrak{e}(L)$ with the singular point removed is an annulus locally modelled on the
quotient of $\HH^2_+$ by a parabolic isometry, \textit{i.e.}, a hyperbolic cusp.
\item If $I=i^-(\ell(p))$, then $L=i^-(L)$ and $i^+(L)=\emptyset$.
The singularity is then called a \textit{extreme BTZ-like future singularity}.
The case where $p$ lies in $\partial\HH^2_-$ and $L$ of degree $0$ is similar, we get the notion of \textit{past cuspidal singularity}
and \textit{extreme BTZ-like past singularity}.
\end{itemize}

\item\textit{space-like singularities of degree $k \geq 1$:} when the singularity is space-like of
degree $k \geq 1$, \textit{i.e.} when $L$ is a hyperbolic
$(\Gamma_p, \ell(p))$-circle of degree $\geq 1$, the situation is slightly more complicated.
In that situation, $L$ is the quotient of the universal covering $\tilde{L}_p \approx \uRP$ by a group generated by
an element of the form $\delta^kg$ where $\delta$ is in the center of $\tilde{G}$ and $g$ admits fixed points in $\tilde{L}_p$.
Let $I^\pm$ be the preimage in $\tilde{L}_p$ of $i^\pm(\ell(p))$ by the developing map.
Let $x_{0}$ be a fixed point of $g$ in $\tilde{L}_p$ which is a left extremity of a component of $I^+$ (recall
that we have prescribed an orientation, \textit{i.e.} an order, on the universal covering of any $\R\PP^1$-circle:
the one for which the developing map is increasing).
Then, this component is an interval $]x_{0}, x_{1}[$
where $x_{1}$ is another $g$-fixed point. All the other $g$-fixed points are the iterates $x_{2i}=\delta^{i}x_{0}$
and $x_{2i+1}=\delta^{i}x_{1}$. The components of $I^+$ are the intervals $\delta^{2i}]x_{0}, x_{1}[$
and the components of $I^-$ are $\delta^{2i+1}]x_{0}, x_{1}[$. It follows that the degree $k$ is an even integer.
We have a dichotomy:
\begin{itemize}
\item If, for every integer $i$, the point $x_{2i}$ (\textit{i.e.}
the left extremities of the components of $I^+$) is a repelling fixed point of $g$,
then the singularity is a \textit{positive space-like singularity of degree $k$.}
\item In the other case, \textit{i.e.} if the left extremities of the components of
$I^+$ are attracting fixed point of $g$,
then the singularity is a \textit{negative space-like singularity of degree $k$.}
\end{itemize}
In other words, the singularity is positive if and only if for every $x$ in $I^+$ we have $gx \geq x$.

\item\textit{light-like singularities of degree $k \geq 1$:}
Similarly, parabolic $(\Gamma_p, \ell(p))$-circles have even degree, and the dichotomy past/future among parabolic
$(\Gamma_p, \ell(p))$-circles of degree $\geq 2$ splits into two subcases: the positive case for which the parabolic
element $g$ satisfies $gx \geq x$
on $\tilde{L}_p$, and the negative case satisfying the reverse inequality (this positive/negative dichotomy is
inherent of the structure of $\uRP$-circle data, cf. the end of {section}~\ref{sub.RP1circle}).

\end{enumerate}

\begin{remark}
In the previous {section} we observed that there is only one $\R\PP^1$ hyperbolic circle of holonomy $\langle g \rangle$
up to $\R\PP^1$-isomorphism, but this remark does not extend to hyperbolic $(\Gamma_p, \ell(p))$-circles since
a real projective conjugacy between $g$ and $g^{-1}$, if preserving the orientation, must permute time-like and space-like
components. Hence positive hyperbolic $(\Gamma_p, \ell(p))$-circles and negative hyperbolic $(\Gamma_p, \ell(p))$-circles
are not isomorphic.
\end{remark}

\begin{remark}
\label{rk:p-p}
Let $L$ be a $(\Gamma_p, \ell(p))$-circle. The suspension $\mathfrak{e}(L)$ admits two singular points $\bar{p}$, $-\bar{p}$, corresponding to
$p$ and $-p$. Observe that when $p$ is space-like, $\bar{p}$ and $-\bar{p}$, as HS-singularities, are always isomorphic.
When $p$ is time-like, one of the singularities is future time-like and the other is past time-like. If $\bar{p}$ is a future light-like singularity of degree $k \geq 1$,
then $-\bar{p}$ is a \textit{past} light-like singularity of degree $k$, and \textit{vice versa}.

Finally, if $\bar{p}$ is a future cuspidal singularity, then $-\bar{p}$ is a past extreme BTZ-like singularity.
\end{remark}

\subsection{Singular HS-surfaces}
\label{sub:HSsingular}
Once we know all possible HS-singularities, we can define singular HS-surfaces:

\begin{defi}
A singular HS-surface $\Sigma$ is an oriented surface containing a discrete subset $\mathcal{S}$
such that $\Sigma \setminus \mathcal{S}$ is a regular HS-surface, and such that
every $p$ in $\mathcal{S}$ admits a neighborhood HS-isomorphic to an open subset of
the suspension $\mathfrak{e}(L)$ of a $(\Gamma_p, \ell(p))$-circle $L$.
\end{defi}

The construction of AdS-manifolds $e(\Sigma)$ extends to singular HS-surfaces:

\begin{defi}
A singular AdS spacetime is a 3-manifold $M$ containing a closed subset $\cL$ (the singular set) such that
$M \setminus \cL$ is a regular AdS-spacetime, and such that every $x$ in $\cL$ admits a neighborhood AdS-isomorphic
to the suspension $e(\Sigma)$ of a singular HS-surface.
\end{defi}

Since we require $M$ to be a manifold, each cone $e(\Sigma)$ must be a 3-ball, \textit{i.e.} each surface $\Sigma$ must be actually
homeomorphic to the 2-sphere.

There are two types of points in the singular set of a singular AdS spacetime:

\begin{defi}
\label{def:singline}
Let $M$ be a singular AdS spacetime. A singular line in $M$ is a connected subset of the singular
set comprising points $x$ such that every neighborhood of $x$ is AdS-isomorphic to the suspension
$e(\Sigma_x)$ where $\Sigma_x$ is a singular HS-surface ${\mathfrak e}(L_x)$ where $L_x$ is a  $(\Gamma_p, \ell(p))$-circle.
An interaction (or collision) in $M$ is a point $x$ in the singular set which is not on a singular line.
\end{defi}

Consider point $x$ in a singular line. Then, by definition, a neighborhood $U$ of $x$ is isomorphic to the suspension $e(\Sigma_x)$
where the HS-sphere $\Sigma_x$ is the suspension of a $(\Gamma_p, \ell(p))$-circle $L$. The suspension ${\mathfrak e}(L)$
contains precisely two opposite points $\bar{p}$ and $-\bar{p}$. Each of them defines a ray in $U$, and every point $x'$
in these rays are singular points, whose links are also described by the same singular HS-sphere ${\mathfrak e}(L)$.

Therefore, we can define the type of the singular line: it is the type of the $(\Gamma_p, \ell(p))$-circle describing the singularity
type of each of its elements. Therefore, a singular line is time-like, space-like or light-like, and it has a degree.

On the other hand, when $x$ is an interaction, then the HS-sphere $\Sigma_x$
is not the suspension of a $(\Gamma_p, \ell(p))$-circle.
Let $\bar{p}$ be a singularity of $\Sigma_x$. It defines in $e(\Sigma)$ a ray, and for every $y$ in this ray, the link of $y$ is isomorphic
to the suspension $e(L)$ of the $(\Gamma_p, \ell(p))$-circle defining the singular point $\bar{p}$.

It follows that the interactions form a discrete closed subset. In the neighborhood of
an interaction, with the interaction removed, the singular set is an union of singular lines, along which the singularity-type is constant
(however see Remark~\ref{rk:p-p}).

\subsection{Classification of singular lines}
\label{sub:classiline}
The classification of singular lines, \textit{i.e.} of $(\Gamma_p, \ell(p))$-circles, follows from the classification of singularities of
singular HS-surfaces:

\begin{itemize}
\item \textit{time-like lines,}
\item \textit{space-like or light-like line of degree $2$,}
\item \textit{BTZ-like singular lines,} extreme or not, past or future,
\item \textit{Misner lines,}
\item \textit{space-like or light-like line of degree $k\geq4$.} Recall that the degree is necessarily even.
\end{itemize}

Indeed, according to Remark~\ref{rk:p-p}, what could have been called a cuspidal singular line, is actually included
an extreme BTZ-like singular line.

\subsection{Local future and past of singular points}
\label{sub.futpast}
In the previous section, we almost completed the proof of Proposition~\ref{pro:classising}, except that we still have to
describe, as stated in this proposition, what is the future and the past of the singular line (in particular, that the future and the
past of non-time-like lines
of degree $k \geq 4$ has $k/2$ connected components), and to see that Misner lines are surrounded by closed causal curves.

Let $M$ be a singular AdS-manifold $M$. Outside the singular set, $M$ is isometric to an AdS manifold.
Therefore one can define as usually the notion of time-like or causal curve, at least outside singular points.

If $x$ is a singular point, then a neighborhood $U$ of $x$ is isomorphic to the suspension of a singular HS-surface $\Sigma_x$.
Every point in $\Sigma_x$, singular or not, is the direction of a line $\ell$ in $U$ starting from $x$. When $x$ is singular,
$\ell$ is a singular line, in the meaning of definition~\ref{def:singline}; if not, $\ell$, with $x$ removed, is a geodesic segment.
Hence, we can extend the notion of causal curves, allowing them to cross an interaction or a space-like singular line, or
to go for a while along a time-like or a light-like singular line.

Once introduced this notion
one can define the future $I^+(x)$ of a point $x$ as the set of final extremities of future oriented time-like
curves starting from $x$. Similarly, one defines the past $I^-(x)$, and the causal past/future $J^\pm(x)$.

Let $\HH^+_{x}$ (resp. $\HH^-_{x}$) be the set of future (resp. past) time-like elements of the HS-surface $\Sigma_{x}$.
It is easy to see that the local future of $x$ in $e(\Sigma_{x})$, which is locally isometric to $M$, is the open domain
$e(\HH^+_{x}) \subset e(\Sigma_{x})$. Similarly, the past of $x$ in $e(\Sigma_{x})$ is $e(\HH^-_{x})$.

It follows that the causality relation in the neighborhood of a point in a time-like singular line has
the same feature than the causality relation near a regular point: the local past and the local future are non-empty
connected open subsets, bounded by light-like geodesics.
The same is true for a light-like or space-like singular line of degree exactly 2.

On the other hand, points in a future BTZ-like singularity, extreme or not, have no future, and only one past
component. This past component is moreover isometric to the quotient of the past of a point in
$\uAdS_3$ by a hyperbolic (parabolic in the extreme case) isometry fixing the point. Hence, it is
homeomorphic to the product of an annulus by the real line.

If $L$ has degree $k \geq 4$, then the local future of a singular point in $e(\mathfrak{e}(L))$ admits $k/2$ components,
hence at least $2$, and the local past as well.
This situation is quite unusual, and in our further study we exclude it: \textbf{from now, we always assume that
light-like or space-like singular lines have degree $0$ or $2$.}

Points in Misner singularities have no future, and no past. Besides, any neighborhood of such a point
contains closed time-like curves (CTC in short). Indeed, in that case, $\mathfrak{e}(L)$ is obtained by glueing the
two space-like sides of a bigon entirely contained in the de Sitter region $\dS^2$ by some isometry $g$, and for every point $x$
in the past side of this bigon, the image $gx$ lies in the future of $x$: any time-like curve joining $x$ to $gx$ induces a CTC in $\mathfrak{e}(L)$.
But:

\begin{lemma}
\label{le.CTC}
Let $\Sigma$ be a singular HS-surface. Then the singular AdS-manifold $e(\Sigma)$ contains closed causal curves (CCC in short) if and only if
the de Sitter region of $\Sigma$ contains CCC. Moreover, if it is the case, every neighborhood of the vertex of $e(\Sigma)$
contains a CCC of arbitrarly small length.
\end{lemma}

\begin{proof}
Let $\bar{p}$ be the vertex of $e(\Sigma)$. Let $\HH^\pm_{\bar{p}}$ denote the future and past hyperbolic part
of $\Sigma$, and let $\dS_{\bar{p}}$ be the de Sitter region in $\Sigma$.
As we have already observed, the future of $\bar{p}$ is the suspension $e(\HH^+_{\bar{p}})$.
Its boundary is ruled by future oriented lighlike lines, singular or not. It follows, as in the regular case, that any
future oriented time-like line entering in the future of $\bar{p}$ remains trapped therein and cannot escape anymore:
such a curve cannot be part of a CCC. Furthermore, the future $e(\HH^+_{\bar{p}})$ is isometric to the product
$\HH^+_{\bar{p}} \times ]0, +\infty[$ equipped with the singular Lorentz metric
$-dt^2 + \cos^2(t) g_{hyp}$, where
$g_{hyp}$ is the singular hyperbolic metric with cone singularities on $\HH^+_{\bar{p}}$ induced by the HS-structure.
The coordinate $t$ induces a time function, strictly increasing along causal curves. Therefore, $e(\HH^+_{\bar{p}})$
contains no CCC.

It follows that CCC in $e(\Sigma)$ avoid the future of $\bar{p}$. Similarly, they avoid the past of $\bar{p}$: every CCC
are entirely contained in the suspension $e(\dS^2_{\bar{p}})$ of the de Sitter region of $\Sigma$.

For any real number $\epsilon$, let $f_\epsilon: \dS^2_{\bar{p}} \to e(\dS^2_{\bar{p}})$ be the map associating to $v$ in
the de Sitter region the point at distance $\epsilon$ to $\bar{p}$ on the space-like geodesic $r(v)$. Then the image of $f_\epsilon$ is a singular Lorentzian submanifold
locally isometric to the de Sitter space rescaled by a factor $\lambda(\epsilon)$. Moreover, $f_\epsilon$ is a conformal isometry: its
differential multiply by $\lambda(\epsilon)$ the norms of tangent vectors.
Since $\lambda(\epsilon)$ tends to $0$ with $\epsilon$, it follows that if $\Sigma$
has a CCC, then $e(\Sigma)$ has a CCC of arbitrarily short length.

Conversely, if $e(\Sigma)$ has a CCC, it can be projected along the radial directions
on a surface corresponding to a fixed value of $t$, keeping it causal, as can be seen
from the explicit form of the metric on $e(\Sigma)$ above. It follows that, when
$e(\Sigma)$ has a CCC, $\Sigma$ also has one. This finishes the proof of the lemma.
\end{proof}

The proof of Proposition~\ref{pro:classising} is now complete.

From now, we
will restrict our attention to HS-surfaces without CCC and  corresponding to singular points where the fure and the past,
if non-empty, are connected:

\begin{defi}
\label{def:Hcausal}
A singular HS-surface is causal if it admits no singularity of degree $\geq 4$ and no CCC.
A singular line is causal if the suspension $\mathfrak{e}(L)$ of the associated $(\Gamma_p, \ell(p))$-circle $L$ is causal.
\end{defi}

In other words, a singular HS-surface is causal if the following singularity types are excluded:

\begin{itemize}
\item space-like or light-like singularities of degree $\geq 4$,
\item Misner singularities.
\end{itemize}


\subsection{Geometric description  of HS-singularities and AdS singular lines}
\label{sub:desgeom}
The approach of singular lines we have given so far has the advantage to be systematic, but is quite abstract. In this section,
we give cut-and-paste constructions of singular AdS-spacetimes which provide a better insight on the geometry of  AdS singularities.

\subsubsection{Massive particles}
\label{sub:massiveparticle}
Let $D$ be a domain in
$\uAdS_3$ bounded by two time-like totally geodesic half-planes $P_1$, $P_2$
sharing as common boundary a time-like geodesic $c$. The angle $\theta$ of $D$ is
the angle between the two geodesic rays $H \cap P_1$, $H \cap P_2$
issued from $c \cap H$, where $H$ is a totally geodesic hyperbolic plane
orthogonal to $c$. Glue $P_1$ to $P_2$ by the elliptic isometry of
$\uAdS_3$ fixing $c$ pointwise. The
resulting space, up to isometry,
only depends on $\theta$, and not on the choices of $c$
and of $D$ with angle $\theta$.
The complement of $c$ is locally modeled on $\AdS_3$, while
$c$ corresponds to a cone singularity with some cone angle $\theta$.

We can also consider a domain $D,$ still bounded by two time-like planes,
but not embedded in $\uAdS_3$, wrapping around $c$, maybe several time, by an angle $\theta > 2\pi$.
Glueing as above, we obtain a singular spacetime with angle $\theta > 2\pi$.

In these examples, the singular line is a time-like singular line, and all time-like singular lines
are clearly produced in this way.

\begin{remark}
\label{rk:masspositive}
There is an important litterature in physics involving such singularities, in the AdS background like here or in
the Minkowski space background, where they are called wordlines, or cosmic strings, describing a massive
particle in motion, with mass $m := 1 - \theta/2\pi$. Hence $\theta > 2\pi$
corresponds to particles with negative mass - but they are usually not considered in physics.
See for example \cite[p. 41-42]{carlip}. Let us mention in particular a famous example by R. Gott in
\cite{gott}, followed by several papers (for example, \cite{grant}, \cite{carroll}, \cite{steif})
where it is shown that a (flat) spacetime containing two such singular lines may present some causal pathology at large scale.
\end{remark}





\subsubsection{Tachyons}
\label{sub.tachyon}
Consider a space-like geodesic $c$  in $\uAdS_3$, and two time-like
totally geodesic planes $Q_1$, $Q_2$ containing $c$.
We will also consider the two light-like totally geodesic subspaces $L_1$ and $L_2$ of $\uAdS_3$
containing $c$, and, more generally, the space $\mathcal P$ of totally geodesic subspaces containing $c$.
Observe that the future of $c$, near $c$, is bounded by $L_1$ and $L_2$.

We choose an orientation of $c$: the orientation of $\uAdS_3$ then induces a (counterclockwise) orientation on
$\mathcal P$, hence on every loop turning around $c$.
We choose the indexation of the various planes $Q_1$, $Q_2$, $L_1$ and $L_2$ such that every loop turning counterclockwise around $x$,
enters in the future of $c$ through $L_1$, then crosses successively $Q_1$, $Q_2$, and finally exits
from the future of $c$ through $L_2$. Observe that if we had considered the past of $c$ instead of the
future, we would have obtained the same indexation. Furthermore, this indexation does not depend
on the orientation on $c$ selected initially.

The planes $Q_1$ and $Q_2$ intersect each other
along infinitely many space-like geodesics, always under the same angle. In each of these planes,
there is an open domain $P_i$ bounded by $c$ and another component $c_+$ of
$Q_1 \cap Q_2$ in the future of $c$ and which does not intersect another component
of $Q_1 \cap Q_2$. The component $c_+$ is a space-like geodesic, which can also be defined
as the set of first conjugate points in the future of points in $c$ (cf. the end of section~\ref{sub:backgroundads}).

The union $c \cup c_+ \cup P_1 \cup P_2$ disconnects
$\uAdS_3$. One of these components, denoted $W$, is contained in the future of $c$
and the past of $c_+$. Let $D$ be the other component, containing the future of
$c_+$ and the past of $c$. Consider the closure of $D$, and
glue $P_1$ to $P_2$ by a hyperbolic isometry of $\uAdS_3$ fixing every point
in $c$ and $c_+$. The resulting spacetime 
contains two space-like singular
lines, still denoted by $c$, $c_+$, and is locally modeled on $\AdS_3$ on the complement of these lines.

Clearly, these singular lines are space-like singularities, isometric to the singularities associated to
a space-like $(\Gamma_p, \ell(p))$-circle $L$ of degree two. We claim futhermore that $c$ is positive.
Indeed, the $(\Gamma_p, \ell(p))$-circle $L$ is naturally identified with
$\mathcal P$. Our choice of indexation implies that the left extremity of $i^+(L)$ is  $L_1$. Since
the holonomy send $Q_1$ onto $Q_2$, the left extremity $L_1$ is a repelling fixed point of the holonomy. Therefore, the singular line
corresponding to $c$ is positive according to our terminology.

On the other hand, a similar reasoning shows that the space-like singular line $c_+$ is \textit{negative}. Indeed,
the totally geodesic plane $L_1$ does not correspond anymore to the left extremities of the time-like components in
the $(\Gamma_p, \ell(p))$-circle associated to $c_+$, but to the right extremities.

\begin{remark}
\label{rk:positive/negative}
Consider a time-like geodesic $\ell$ in $\uAdS_3$, hitting the boundary of the future of $c$ at a point in $P_1$.
This geodesic corresponds to a time-like geodesic $\ell'$ in the singular spacetime defined by our cut-and-paste surgery
which coincides with $\ell$ before crossing $P_1$, and, after the crossing, with the image $\ell'$ of $\ell$ by the holonomy.
The direction of $\ell'$ is closer to $L_2$ than was $\ell$.

In other words, the situation is as if the singular line $c$ were attracting the lightrays,
\textit{i.e.} had positive mass.
This is the reason why we call $c$ a \textit{positive} singular line (section~\ref{sub:positive}).
\end{remark}


There is an alternative description of these singularities: start again from a space-like geodesic $c$ in $\uAdS_3$,
but now consider two space-like half-planes $S_{1}$, $S_{2}$ with common boundary $c$, such that
$S_{2}$ lies above $S_{1}$, \textit{i.e.} in the future of $S_{1}$. Then remove the intersection $V$ between
the past of $S_{2}$ and the future
of $S_{1}$, and glue $S_{1}$ to $S_{2}$ by a hyperbolic isometry fixing every point in $c$.
The resulting singular
spacetime contains a singular space-like line.
It should be clear to the reader that this singular line is space-like of degree $2$ and negative.
If instead of removing a wedge $V$ we insert it in the spacetime obtained by cutting
$\uAdS_3$ along a space-like half-plane $S$,
we obtain a spacetime with a positive space-like singularity of degree $2$.

Last but not least, there is another way to construct space-like singularities of degree $2$.
Given the space-like geodesic $c$, let $L_1^+$ be the future component of $L_1 \setminus c$.
Cut along $L_1^+$, and glue back by a hyperbolic
isometry $\gamma$ fixing every point in $c$. More precisely, we consider the singular spacetime
such that for every future oriented time-like curve in $\uAdS_3 \setminus L_1^+$ terminating at $L_1^+$ at a point $x$
can be continued in the singular spacetime by a future oriented time-like curve starting from $\gamma x$.
Once more, we obtain a singular AdS-spacetime containing a space-like singular line of degree $2$.
We leave to the reader the proof of the following fact:
\textit{the singular line is positive mass if and only if for every $x$ in $L_1^+$ the light-like
segment $[x, \gamma x]$ is past-oriented,\/} \textit{i.e.} $\gamma$ sends every point in $L_1^+$ in
its own causal past.

\begin{remark}
\label{rk.cuttachyon2}
As a corollary we get the following description space-like HS-singularities of degree $2$:
consider a small disk $U$ in $\dS^2$ and a point $x$ in $U$. Let $r$ be one light-like geodesic ray
contained in $U$ issued from $x$, cut along it and glue back by a hyperbolic $\dS^2$-isometry $\gamma$
like described in figure~\ref{fig.tachyon} (be careful that in this figure, the isometry, glueing the future copy of
$r$ in the boundary of $U \setminus r$ into the past copy of $r$; hence $\gamma$ is the inverse of the holonomy).
Observe that one cannot match one side on the other, but the resulting space
is still homeomorphic to the disk. The resulting HS-singularity is space-like, of degree $2$. If $r$ is future
oriented, the singularity is positive if and only if for every $y$ in $r$ the image $\gamma y$ lies
in the future of $y$. If $r$ is past oriented, the singularity is positive
if and only if $\gamma y$ lies in the past of $y$ for every $y$ in $r$.

\begin{figure}[ht]
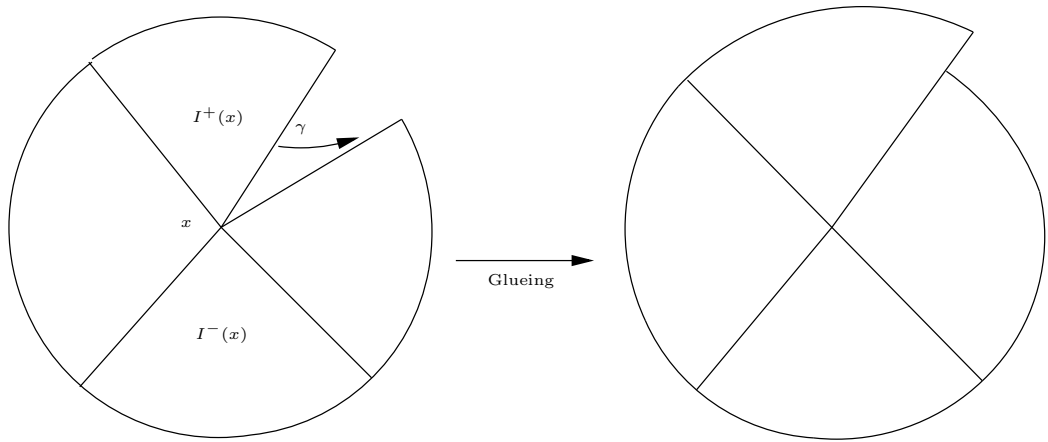

\begin{center}
\input Chirtachyon.pstex_t
\end{center}
\caption{Construction of a positive space-like singular line of degree $2$.}
\label{fig.tachyon}
\end{figure}
\end{remark}

\begin{remark}
As far as we know, this kind of singular lines is not considered in
physics litterature. However, it is a very natural extension of the notion of massive particles.
It sounds to us natural to call these singularities, representing particles faster than light, \textit{tachyons,}
which can be positive or negative, depending on their influence on lightrays.
\end{remark}

\begin{remark}
\label{rk:degk}
space-like singularity of any (even) degree $2k$ can be constructed as $k$-branched cover of a space-like singularity of
degree $2$. In other words, they are obtained by identifying $P_1$ and $P_2$, but now seen as
the boundaries of a wedge turning $k$ times around $c$.
\end{remark}


\subsubsection{Misner singularities}
\label{sub:misner}
Let $S_1$, $S_2$ be two space-like half-planes with common boundary as appearing in the second version
of definition of tachyons in the previous section. Now, instead of removing the intersection $V$ between the future
of $S_1$ and the past of $S_2$, keep it and remove the other part (the main part!) of $\uAdS_3$. Glue its two boundary components
$S_1$, $S_2$ by an AdS-isometry fixing $c$ pointwise. The reader will easily convince himself that the resulting spacetime
contains a space-like line of degree $0$, \textit{i.e.} what we have called a Misner singular line.

The reason of this terminology is that this kind of singularity is oftently considered, or mentioned\footnote{Essentially
because of their main feature pointed out in section~\ref{sub.futpast}: they are surrounded by CTC.},
in papers dedicated to
gravity in dimension $2+1$, maybe most of the time in the Minkowski background, but also in the AdS background.
They are attributed to Misner who considered the $3+1$-dimensional analog of this spacetime
(for example, the glueing is called ``Misner identification''
in \cite{deser}; see also \cite{gottli}).

\subsubsection{BTZ-like singularities}
\label{sub:btz}
Consider the same data $(c, c_{+}, P_{1}, P_{2})$ used for the description of tachyons, \textit{i.e.} space-like singularities,
but now remove $D$,
and glue the boundaries $P_{1}$, $P_{2}$ of $W$ by a hyperbolic element $\gamma_{0}$ fixing every point in $c$.
The resulting space is a manifold $\mathcal{B}$ containing two singular lines, that we abusively
still denote $c$ and $c_+$, and is locally $\AdS_3$ outside $c$, $c_+$.
Observe that every point of $\mathcal{B}$ lies in the past of the singular line corresponding
to $c_+$ and in the future of the singular line corresponding to $c$. It follows easily that $c$ is
a BTZ-like past singularity, and that $c_+$ is a BTZ-like future singularity.

\begin{remark}
Let $E$ be the open domain in $\uAdS_3$, intersection between the future of $c$ and the past of $c_{+}$. Observe that
$\overline{W} \setminus P_1$ is a fundamental domain for the action on $E$ of the group $\langle\gamma_{0}\rangle$ generated
by $\gamma_0$. In other words, the regular part of $\mathcal{B}$ is isometric
to the quotient $E/\langle\gamma_{0}\rangle$. This quotient is precisely a \textit{static BTZ black-hole}
as first introduced by Ban{\~a}dos, Teitelboim and Zanelli in \cite{BTZ} (see also \cite{barbtz1, barbtz2}).
It is homeomorphic to the product of the annulus by the real
line. The singular spacetime $\mathcal{B}$ is obtained by adjoining to this BTZ black-hole
two singular lines: this follows that $\mathcal{B}$ is homeomorphic to the product of a 2-sphere with the real line
in which $c_+$ and $c$ can be naturally considered respectively as the future singularity and the past singularity.
This is the explanation of the ``BTZ-like'' terminology.
More details will be given in section~\ref{sub.bhfrancesco}.
\end{remark}

\begin{remark}
This kind of singularity appears in several papers in the physics litterature.
We point out among them the excellent paper \cite{matschull} where the Gott's construction
quoted above is
adapted to the AdS case, and where is provided a complete and very subtle description of singular AdS-spacetimes
interpreted as the creation of a BTZ black-hole by a pair of light-like particles, or by a pair of massive particles.
In our terminology, these spacetimes contains three singularities: a pair of light-like or time-like positive singular lines,
and a BTZ-like future singularity. These examples shows that even if all the singular lines are causal, in the sense
of Definition~\ref{def:Hcausal}, a singular spacetime may exhibit big CCC due to a more global phenomenon.
\end{remark}

\subsubsection{light-like and extreme BTZ-like singularities}
\label{sub.graviton}
The definition of a light-like singularity is similar to that of space-like singularities of degree $2$ (tachyons),
but starts with the choice of a \textit{light-like} geodesic $c$ in $\uAdS_3$. Given such
a geodesic, we consider another light-like geodesic $c_+$ in the future of $c$, and two disjoint time-like
totally geodesic annuli $P_1$, $P_2$ with boundary $c \cup c_+$.


More precisely, consider pairs of space-like geodesics $(c^n, c^n_+)$ as those appearing
in the description of tachyons, contained in time-like planes $Q^n_1$, $Q^n_2$, so that
$c^n$ converge to the light-like geodesic $c$. Then, $c^n_+$ converge to a light-like geodesic $c_+$, whose past extremity
in the boundary of $\uAdS_3$ coincide with the future extremity of $c$. The time-like planes
$Q^n_1$, $Q^n_2$ converge to time-like planes $Q_1$, $Q_2$ containing $c$ and $c_+$. Then
$P_i$ is the annulus bounded in $Q_i$ by $c$ and $c_+$.
Glue the boundaries $P_1$ and $P_2$ of the component $D$ of $\uAdS_3 \setminus (P_1 \cup P_2)$
containing the future of $c$ by an isometry of $\uAdS_3$ fixing every point in $c$ (and in $c_+$): the
resulting space is a singular AdS-spacetime, containing two singular lines, abusely denoted by
$c$, $c_+$. As in the case of tachyons, we can see that these singular lines have degree $2$,
but they are light-like instead of space-like. The line $c$ is called \textit{positive,} and $c_+$ is \textit{negative.}



Similarly to what happens for tachyons, there is an alternative way to construct light-like singularities.
Let $L^+$ be the unique light-like half-plane bounded by $c$ and contained in the causal future of $c$.
Cut $\uAdS_3$ along $L^+$, and glue back by an isometry $\gamma$ fixing pointwise $c$:
the result is a singular spacetime containing a light-like singularity of degree $2$. Furthermore,  \textit{this light-like singularity is positive if and only
if $\gamma$ sends every point of $L^+$ in its own causal past.}

Finally, extreme BTZ-like singularities can be described in a way similar to what we have done for
(non extreme) BTZ-like singularities. As a matter of fact, when we glue the wedge $W$ between $P_1$ and
$P_2$ we obtain a (static) extreme BTZ black-holes as described in \cite{BTZ} (see also
\cite[section~3.2, section~10.3]{barbtz2}).
Further comments and details are left to the reader

\begin{remark}
\label{rk:photon}
light-like singularities of degree $2$ appear very frequently in physics, where they are called
wordlines, or cosmic strings, of massless particles, or even sometimes ``photons'' (\cite{deser}).
\end{remark}

\begin{remark}
As in the case of tachyons (see Remark~\ref{rk:degk}) one can construct light-like singularities
of any degree $2k$ by considering a wedge turning $k$ times around $c$ before glueing its
boundaries.
\end{remark}

\begin{remark}
\label{rk:positive/negative2}
A study similar to what has been done in Remark~\ref{rk:positive/negative}, one observes that positive photons attract
lightrays, whereas negative photons have a repelling behavior.
\end{remark}

\begin{remark}
However, there is no positive/negative dichotomy for BTZ-like singularities, extreme or not.
\end{remark}

\begin{remark}
From now on, we allow ourselves to qualify HS-singularities according to the nature of the associated
AdS-singular lines: an elliptic HS-singularity is a (massive) particle, a
space-like singularity is a tachyon, positive or negative, etc...
\end{remark}

\begin{remark}
\label{rk.cheminsing}
Let $[p_1, p_2]$ be an oriented arc in $\partial\HH^2_+$, and for every $x$ in $\HH^2_+$ consider the elliptic singularity
(with positive mass) obtained by removing the wedge comprising geodesic rays issued from $x$ and with extremity in $[p_1, p_2]$ and
glueing back by an elliptic isometry. Move $x$ until it reaches a point $x_\infty$ in $\partial\HH^2 \setminus [p_1, p_2]$.
It provides a continuous deformation of an elliptic singularity to a light-like singularity, which can be continued further into $\dS^2$
by a continuous sequence of space-like singularities.
Observe that the light-like (resp. space-like) singularities appearing in this continuous family are positive (resp. have positive mass).
\end{remark}

\subsection{Positive HS-surfaces}
\label{sub:positive}
Among singular lines, i.e. ``particles'', we can distinguish the ones having an attracting behavior on lightrays (see Remark~\ref{rk:masspositive},
\ref{rk:positive/negative}, \ref{rk:positive/negative2}):

\begin{defi}
A HS-surface, an interaction or a singular line is
\textit{positive} if all space-like and light-like singularities of degree $\geq 2$ therein are positive, and if all time-like singularities
have a cone angle less that $2\pi$.
\end{defi}



\section{Particle interactions and convex polyhedra}
\label{sc:polyhedra}

This short section briefly describes a relationship between
interactions of particles in $3$-dimensional AdS manifolds, 
HS-structure on the sphere, and convex polyhedra in $\HS^3$, the natural
extension of the hyperbolic $3$-dimensional by the de Sitter space. 

Convex polyhedra in $\HS^3$ provide a convenient way to visualize 
a large variety of particle interactions in AdS manifolds (or 
more generally in Lorentzian 3-manifolds). This section should provide
the reader with a wealth of examples of particle interactions -- 
obtained from convex polyhedra in $\HS^3$ -- exhibiting various 
interesting behavior. It should then be easier to follow the
classification of positive causal $HS$-surfaces in the next section.

The relationship between convex polyhedra and particle interactions
might however be deeper than just a convenient way to construct
examples. It appears that many, and possibly all,
particle interactions in an AdS manifold satisfying some natural 
conditions correspond to a unique
convex polyhedron in $\HS^3$. 
This deeper aspect of the relationship
between particle interactions and convex polyhedra is described 
in Section \ref{ssc:converse}
only in a special case: interactions between only massive
particles and tachyons. It appears likely that it extends to a 
more general context, however it appears preferable to restrict
those considerations here to a special case which, although 
already exhibiting interesting phenomena, avoids the technical
complications of the general case. 

\subsection{The space $\HS^3$}

The definition used above for $\HS^2$ can be extended as it is to
higher dimensions. So $\HS^3$ is the space of geodesic rays starting from $0$ in 
the four-dimensional Minkowskip space $\R^{3,1}$. It admits a natural action of
$SO_0(1,3)$, and has a decomposition in 5 connected components:
\begin{itemize}
\item The ``upper'' and ``lower'' hyperbolic components, denoted by $H^3_+$ and $H^3_-$,
corresponding to the future-oriented and past-oriented time-like rays. 
On those two components the angle between geodesic rays corresponds to the 
hyperbolic metric on $H^3$.
\item The domain $dS_3$ composed of space-like geodesic rays.
\item The two spheres $\dr H^3_+$ and $\dr H^3_-$ which are the boundaries of 
$H^3_+$ and $H^3_-$, respectively. We call $Q$ their union.
\end{itemize}

There is a natural projective model of $\HS^3$ in the double cover of $\rp^3$ -- 
we have to use the double cover because $\HS^3$ is defined as a space of geodesic
rays, rather than as a space of geodesics containing $0$. This model has the
key properties that projective lines correspond to the hyperbolic and to the 
de Sitter geodesics in the $H^3_\pm$ and the $dS_3$ components of $\HS^3$.

Note that there is a danger of confusion with the notations used in \cite{shu}, since
the space which we call $\HS^3$ here is denoted by $\HSt^3$ there, while the space 
$\HS^3$ in \cite{shu} is the quotient of the space $\HS^3 $ considered here by the
antipodal action of $\Z/2\Z$.

\subsection{Convex polyhedra in $\HS^3$}

In all this section we consider convex polyhedra in $\HS^3$ but will always
suppose that they do not have any vertex on $Q$. We now consider such a 
polyhedron, calling it $P$.

The geometry induced on the boundary of $P$
depends on its position relative to the two hyperbolic components of $\HS^3$,
and we can distinguish three types of polyhedra.

\begin{figure}[ht]
\begin{center}
\psfig{figure=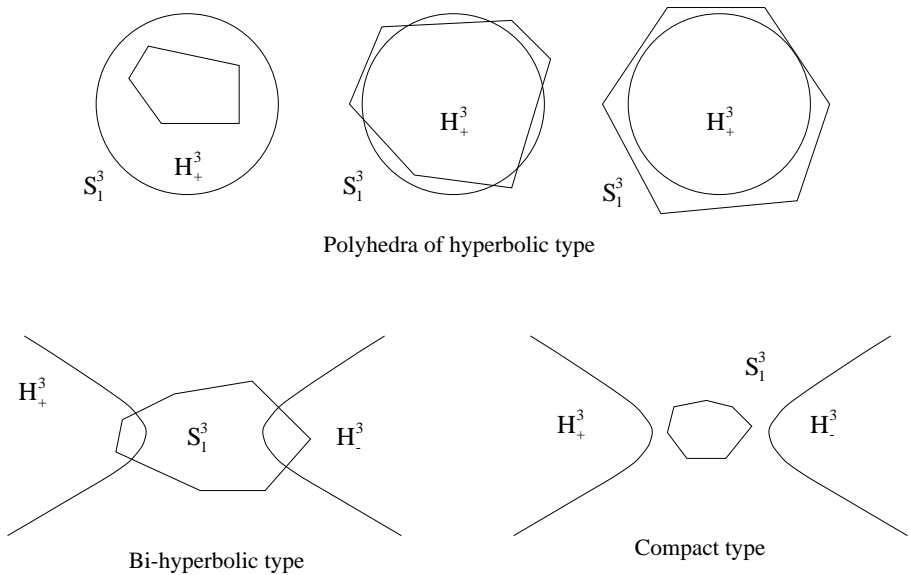,width=12cm}
\end{center}
\caption{Three types of polyhedra in $\HS^3$.}
\label{fig:3poly}
\end{figure}

\begin{itemize}
\item polyhedra of {\it hyperbolic} type intersect one of the hyperbolic
components of $\HS^3$, but not the other. We find for instance in this group:
\begin{itemize}
\item the usual, compact hyperbolic polyhedra, entirely contained in one
of the hyperbolic components of $\HS^3$,
\item the ideal or hyperideal hyperbolic polyhedra,
\item the duals of compact hyperbolic polyhedra, which contain one of the
hyperbolic components of $\HS^3$ in their interior.
\end{itemize}
\item polyhedra of {\it bi-hyperbolic} type intersect both hyperbolic components
of $\HS^3$,
\item polyhedra of {\it compact} type are contained in the de Sitter component of 
$\HS^3$.
\end{itemize}
The terminology used here is taken from \cite{cpt}.

We will see below that polyhedra of bi-hyperbolic type
play the simplest role in relation to particle interactions: 
they are always related to the simpler interactions involving only massive
particles and tachyons. Those of hyperbolic type are (sometimes) related to particle
interactions involving a black hole or a white hole. Polyhedra of compact type are 
the most exotic when considered in relation to particle interactions and will
not be considered much here, for reasons which should appear clearly below.

\subsection{Induced $HS$-structures on the boundary of a polyhedron}

We now consider the geometric structure induced on the boundary of a 
convex polyhedron in $\HS^3$. Those geometric structures have been 
studied in \cite{shu,cpt}, and we will partly rely on those references,
while trying to make the current section as self-contained as possible. 
Note however that the notion of $HS$ {\it metric} used in \cite{shu,cpt}
is more general than the notion of $HS$-structure considered here.

In fact the geometric structure induced on the boundary of a convex
polyhedron $P\subset \HS^3$ is an $HS$-structure in some, but not
all, cases, and the different types of polyhedra behave differently in
this respect.

\subsubsection{Polyhedra of bi-hyperbolic type}

This is the simplest situation: the induced geometric structure is 
{\it always} a causal positive singular $HS$-structure.

The geometry of the induced geometric structure on those polyhedra
is described in \cite{cpt}, under the condition that there there
is no vertex on the boundary at infinity of the two hyperbolic components
of $\HS^3$. The boundary of $P$ can be decomposed in three components:
\begin{itemize}
\item A ``future'' hyperbolic disk $D_+:=\partial P \cap H^3_+$, on 
which the induced metric is hyperbolic (with cone singularities at 
the vertices) and complete.
\item A ``past'' hyperbolic disk $D_-=\partial P\cap H^3_-$, similarly
with a complete hyperbolic metric.
\item A de Sitter annulus, also with cone singularities at the vertices
of $P$.
\end{itemize}

In other terms, $\partial P$ is endowed with an $HS$-structure. Moreover
all vertices in the de Sitter part of the $HS$-structure have degree $2$.

A key point is that the convexity of $P$ implies directly that this
$HS$-structure is positive: the cone angles are less than $2\pi$
at the hyperbolic vertices of $P$, while the positivity condition
is also satisfied at the de Sitter vertices. This can be checked
by elementary geometric arguments or can be found in 
\cite[Definition 3.1 and Theorem 1.3]{cpt}.

\subsubsection{Polyhedra of hyperbolic type}

In this case the induced geometric structure
is {\it sometimes} a causal positive $HS$-structure.
The geometric structure on those polyhedra is described in \cite{shu},
again when $P$ has no vertex on $\dr H^3_+\cup \dr H^3_-$.

\begin{figure}[ht]
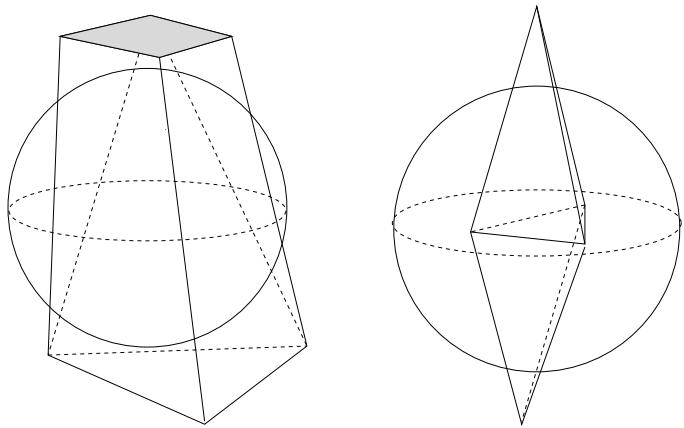

\input hyperbolic1.pstex_t 
\hspace{1cm}
\input hyperbolic2.pstex_t
\caption{Two polyhedra of hyperbolic type.}
\label{fg:hyperbolic}
\end{figure}

Figure \ref{fg:hyperbolic} shows on the left an example of 
polyhedron of hyperbolic type for which the induced geometric structure
is not an $HS$-structure, since the upper face (in gray) is a space-like
face in the de Sitter part of $\HS^3$, so that it is not modelled on 
$\HS^2$. 

The induced geometric structure on the boundary of the polyhedron
shown on the right, however, is a positive causal $HS$-structure. 
At the upper and lower vertices, this $HS$-structure has degree $0$.
The three ``middle'' vertices are contained in the hyperbolic part of 
the $HS$-structure, and the positivity of the $HS$-structure at those
vertices follows from the convexity of the polyhedron. 

\subsubsection{Polyhedra of compact type}

In this case too, the induced geometric structure is also 
{\it sometimes} a causal $HS$-structure.

\begin{figure}[ht]
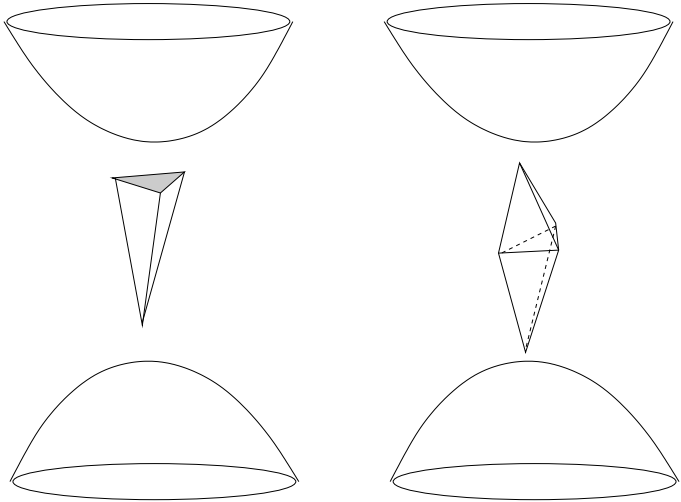

\input compact1.pstex_t 
\hspace{1cm}
\input compact2.pstex_t
\caption{Two polyhedra of compact type.}
\label{fg:compact}
\end{figure}

On the left side of figure \ref{fg:compact} we find an example of a 
polyhedron of compact type on which the induced geometric structure
is not an $HS$-structure -- the upper face, in gray, is a space-like
face in the de Sitter component of $\HS^3$. On the right side, the
geometric structure on the boundary of the polyhedron is a positive causal 
$HS$-structure. All faces are time-like faces, so that they are modelled
on $\HS^2$. The upper and lower vertices have degree $0$, while the
three ``middle'' vertices have degree $2$, and the positivity of the
$HS$-structure at those points follows from the convexity of the 
polyhedron (see \cite{cpt}).

\subsection{From a convex polyhedron to a particle interaction}

When a convex polyhedron has on its boundary an induced positive
causal $HS$-structure, it is possible to consider the interaction
corresponding to this $HS$-structure. 

This interaction can be 
constructed from the $HS$-structure by a warped product metric construction.
It can also be obtained as in Section 2, by noting that each open subset of the 
regular part of the $HS$-structure corresponds to a cone in $AdS_3$, and
that those cones can be glued in a way corresponding to the gluing of the
corresponding domains in the $HS$-structure.

The different types of polyhedra -- in particular the examples in 
Figure \ref{fg:compact} and Figure \ref{fg:hyperbolic} -- correspond to 
differen types of interactions. 

\subsubsection{Polyhedra of bi-hyperbolic type}

For those polyhedra the hyperbolic vertices in $H^3_+$ (resp. 
$H^3_-$) correspond to massive particles leaving from (resp. 
arriving at) the interaction. The de Sitter vertices, at which
the induced $HS$-structure has degree $2$, correspond to tachyons.

\subsubsection{Polyhedra of hyperbolic type}

In the example on the right of Figure \ref{fg:hyperbolic}, the upper and lower
vertices correspond, through the definitions in Section \ref{ssc:singular}, to two
black holes (or two white holes, depending on the time orientation). 
The three middle vertices correspond to massive particles. The interaction
corresponding to this polyhedron therefore involves two black (resp. white)
holes and three massive particles. 

The interactions corresponding to polyhedra of hyperbolic type can be
more complex, in particular because the topology of the intersection of
the boundary of a convex polyhedron with the de Sitter part of $\HS^3$ 
could be a sphere with an arbitrary number of disks removed. Those
interaction can involve black holes and massive particles, but also
tachyons.

\subsubsection{Polyhedra of compact type}

The interaction corresponding to the polyhedron at the right of 
Figure \ref{fg:compact} is even more exotic. The upper vertex corresponds
to a BTZ-type singularity, the lower to a white hole, and the three middle vertices 
correspond to tachyons. The interaction therefore involves a BTZ-type singularity,
a white hole and three tachyons.

\subsection{From a particle interaction to a convex polyhedron}
\label{ssc:converse}

This section describes, in a restricted setting, a converse to the construction
of an interaction from a convex polyhedron in $\HS^3$. We show below that,
under an additional condition which seems to be physically relevant, an 
interaction can always be obtained from a convex polyhedron in $\HS^3$. 
Using the relation described in Section 2 between interactions and positive
causal $HS$-structures, we will related convex polyhedra to those $HS$-structures
rather than directly to interactions.

This converse relation is described here only for simple interactions
involving massive particles and tachyons.

\subsubsection{A positive mass condition}

The additional condition appearing in the converse relation is natural in view
of the following remark. 

\begin{remark}
Let $M$ be a singular AdS manifold, $c$ be a cone singularity along a time-like
curve, with positive mass (angle less than $2\pi$). Let $x\in c$ and let $L_x$ be
the link of $M$ at $x$, and let $\gamma$ be a simple closed space-like geodesic in 
the de Sitter part of $L_x$. Then the length of $\gamma$ is less than $2\pi$.
\end{remark}

\begin{proof}
An explicit description of $L_x$ follows from the construction of the AdS metric
in the neighborhood of a time-like singularity, as seen in Section 2. The 
de Sitter part of this link contains a unique simple closed geodesic, and its
length is equal to the angle at the singularity. So it is less than $2\pi$.
\end{proof}

In the sequel we consider a singular $HS$-structure $\sigma$ on $S^2$, which is the link
of an interaction involving massive particles and tachyons. This means that $\sigma$
is positive and causal, and moreover:
\begin{itemize}
\item it has two hyperbolic components, $D_-$ and $D_+$, on which $\sigma$ restricts 
to a complete hyperbolic metric with cone singularities,
\item any future-oriented inextendible time-like line in the de Sitter region of $\sigma$
connects the closure of $D_-$ to the closure of $D_+$.
\end{itemize}

\begin{defi}
$\sigma$ has {\bf positive mass} if any simple closed space-like geodesic in the de
Sitter part of $(S^2,\sigma)$ has length less than $2\pi$.
\end{defi}

This notion of positivity of mass for an interaction generalizes the natural notion
of positivity for time-like singularities.

\subsubsection{A convex polyhedron from simpler interactions}

\begin{theorem} \label{tm:converse}
Let $\sigma$ be a positive causal $HS$-structure on $S^2$, such that
\begin{itemize}
\item it has two hyperbolic components, $D_-$ and $D_+$, on which $\sigma$ restricts 
to a complete hyperbolic metric with cone singularities,
\item any future-oriented inextendible time-like line in the de Sitter region of $\sigma$
connects the closure of $D_-$ to the closure of $D_+$.
\end{itemize}
Then $\sigma$ is induced on a convex polyhedron in $\HS^3$ if and only if it has
positive mass. If so, this polyhedron is unique, and it is of bi-hyperbolic type.
\end{theorem}

\begin{proof}
This is a direct translation of \cite[Theorem 1.3]{cpt} (see in particular case
D.2).
\end{proof}

The previous theorem is strongly related to classical statements on the induced
metrics on convex polyhedra in the hyperbolic space, see \cite{alex}.

\subsubsection{More general interactions/polyhedra}

As mentioned above we believe that Theorem \ref{tm:converse} might be extended
to wider situations. This could be based on the statements on the induced
geometric structures on the boundaries of convex polyhedra in $\HS^3$, 
as studied in \cite{shu,cpt}.

\section{Classification of positive causal HS-surfaces}
\label{sec.classificationHS}

In all this {section} $\Sigma$ denotes a closed (compact without boundary) connected positive causal HS-surface. It decomposes
in three regions:

\begin{itemize}
\item \textit{Photons:} a photon is a point corresponding in every
HS-chart to points in $\partial\HH^2_\pm$. Observe that a photon might be singular,
\textit{i.e.} corresponds to a light-like singularity (a massless particle or an extreme BTZ-like singularity).
The set of photons, denoted
$\mathcal{P}(\Sigma)$, or simply $\mathcal{P}$ in the non-ambiguous situations,
is the disjoint union of a finite number of isolated points (extreme BTZ-like singularities
or cuspidal singularities)  and of a compact embedded
one dimensional manifold, \textit{i.e.} a finite union of circles.

\item \textit{Hyperbolic regions:} They are the connected components
of  the open subset $\HH^2(\Sigma)$ of $\Sigma$
corresponding to the time-like regions $\HH^2_\pm$ of $\HS^2$.
They are naturally hyperbolic surfaces with cone singularities. There
are two types of hyperbolic regions: the future and the past ones.
The boundary of every hyperbolic region is a finite union of circles of photons
and of cuspidal (parabolic) singularities.

\item \textit{De Sitter regions:} They are the connected components
of  the open subset $\dS^2(\Sigma)$ of $\Sigma$
corresponding to the time-like regions $\dS^2$ of $\HS^2$.
Alternatively, they are the connected components of $\Sigma \setminus \mathcal{P}$
that are not hyperbolic regions. Every de Sitter region is a singular dS surface,
whose closure is compact and with boundary made of circles of photons and of
a finite number of extreme parabolic singularities.

\end{itemize}

\subsection{Photons}
Let $C$ be a circle of photons. It admits two natural $\R\PP^1$-structures, which may not coincide if
$C$ contains light-like singularities.

Consider a closed annulus $A$ in $\Sigma$ containing $C$ so that all HS-singularities in
$A$ lie in $C$.
Consider first the hyperbolic side, \textit{i.e.} the component $A_{H}$ of $A \setminus C$ comprising time-like elements.
Reducing $A$ if necessary we can assume that $A_{H}$ is contained in one hyperbolic region. Then every path
starting from a point in $C$ has infinite length in $A_{H}$, and inversely every complete geodesic ray in $A_{H}$
accumulates on an unique point in $C$. In other words, $C$ is the conformal boundary at $\infty$ of $A_{H}$. Since
the conformal boundary of $\HH^2$ is naturally $\R\PP^1$ and that hyperbolic isometries are restrictions of real projective transformations,
$C$ inherits, as conformal boundary of $A_{H}$, a $\R\PP^1$-structure that we call \textit{$\R\PP^1$-structure on $C$ from the hyperbolic side.}

Consider now the component $A_{S}$ in the de Sitter region adjacent to $C$. It is
is foliated by the light-like lines. Actually, there are two such foliations (for more details, see
\ref{sub.dSclass} below).
An adequate selection of this annulus ensures that the leaf space of each of these foliations
is homeomorphic to the circle - actually, there is a natural identification
between this leaf space and $C$: the map associating to a leaf its extremity.
These foliations are transversely projective: hence they induce a $\R\PP^1$-structure
on $C$.

This structure is the same for both foliations, we call it \textit{$\R\PP^1$-structure on $C$ from the de Sitter side.}
In order to sustain this
claim, we refer \cite[{\S} 6]{mess}: first observe that $C$ can be slightly pushed inside $A_{S}$
onto a space-like simple closed curve (take a loop around $C$ following alternatively past oriented light-like
segments in leaves of one of the foliations, and future oriented segments in the other foliation; and smooth it).
Then apply \cite[Proposition 17]{mess}.

If $C$ contains no light-like singularity, the $\R\PP^1$-structures from the hyperbolic and de Sitter sides coincide.
But it is not necessarily true if $C$ contains light-like singularities. Actually, the holonomy from one side is obtained
by composing the holonomy from the other side by parabolic elements, one for each light-like singularity in $C$.
Observe that in general even the degrees may not coincide.

\subsection{Hyperbolic regions}
Every component of the hyperbolic region has a compact closure in $\Sigma$. It follows easily that
every hyperbolic region is a complete hyperbolic surface with cone singularities (corresponding to massive particles)
and cusps (corresponding to extreme BTZ-like singularities) and that is of finite type, \textit{i.e.} homeomorphic
to a compact surface without boundary with a finite set of points removed.

\begin{prop}
\label{pro.hypdegree0}
Let $C$ be a circle of photons in $\Sigma$, and $H$ the hyperbolic region adjacent to $C$.
Let $\bar{H}$ be the open domain in $\Sigma$ comprising $H$ and all cuspidal singularities
contained in the closure of $H$. Assume that $\bar{H}$ is not homeomorphic to the disk.
Then, as a $\R\PP^1$-circle defined by the hyperbolic side, the circle $C$ is hyperbolic of degree 0.
\end{prop}

\begin{proof}
The proposition will be proved if we find
an annulus in $H$ bounded by $C$ and a simple closed geodesic in $H$
containing no singularity. Indeed, the holonomy along $C$ coincide then with the holonomy of
the closed geodesic. It is well-known that closed geodesics in hyperbolic surfaces are hyperbolic. Further details
are left to the reader.

Since we assume that $\bar{H}$ is not a disk, $C$ represents a non-trivial free homotopy class in $H$.
Consider absolutely continuous simple loops in $H$ freely homotopic to $C$ in $H \cup C$.
Let $L$ be the length of one of them. There are two compact subsets $K \subset K' \subset \bar{H}$ such that
every loop of length $\leq 2L$ containing a point in the complement of $K'$ stays outside $K$ and
is homotopically trivial. It follows that every loop freely homotopic to $C$
of length $\leq L$ lies in $K'$:
by Ascoli and semi-continuity of the length, one of them has minimal length $l_0$
(we also use the fact that $C$ is not freely homotopic to a small closed loop around a cusp of $H$,
details are left to the reader).
It is obviously simple, and it contains no singular point, since every path containing a singularity
can be shortened. Hence it is a closed hyperbolic geodesic.

There could be several such closed simple geodesics of minimal length, but they are two-by-two disjoint,
and the annulus bounded by two such minimal closed geodesic must contain at least one singularity since there
is no closed hyperbolic annulus bounded by geodesics. Hence, there is only a finite number of such
minimal geodesics, and for one of them, $c_0$, the annulus $A_0$ bounded by $C$ and $c_0$
contains no other minimal closed geodesic.

If $A_0$ contains no singularity,
the proposition is proved. If not, for every $r > 0$, let $A(r)$ be the
set of points in $A_0$ at distance $< r$ from $c_0$, and let $A'(r)$ be the complement of
$A(r)$ in $A_0$. For small value of $r$, $A(r)$ contains no singularity. Thus, it is isometric
to the similar annulus in the unique hyperbolic annulus containing a geodesic loop of length
$l_0$. This remarks holds as long as $A(r)$ is regular. Denote by $l(r)$ the length
of the boundary $c(r)$ of $A(r)$.

Let $R$ be the supremum of positive real numbers $r_0$ such that for every $r < r_0$
every essential loop in $A'(r)$ has length $\geq l(r)$. Since $A_0$ contains no closed
geodesic of length $\leq l_0$, this supremum is positive. On the other hand,
let $r_1$ be the distance between $c_0$ and the singularity $x_1$ in $A_0$ nearest to $c_0$.

We claim that $r_1 > R$. Indeed: near $x_1$ the surface is isometric
to a hyperbolic disk $D$ centered at $x_1$ with a wedge between two geodesic rays $l_1$, $l_2$ issued from $x_1$
of angle $2\theta$ removed. Let $\Delta$ be the geodesic ray issued from $x_1$ made of points at equal distance
from $l_1$ and from $l_2$.
Assume by contradiction $r_1 \leq R$. Then, $c(r_1)$ is a simple loop, containing $x_1$ and
minimizing the length of loops inside the closure of $A'(r_1)$.
Singularities of cone angle $2\pi-2\theta < \pi$ cannot be approached by length minimizing closed loops,
hence $\theta \leq \pi/2$. Moreover, we can assume without loss of generality that $c(r)$
near $x_1$ is the projection of a $C^1$-curve $\hat{c}$ in $D$ orthogonal to $\Delta$
at $x_1$, and such that the removed wedge between $l_1$, $l_2$, and the part of $D$ projecting into $A(r)$ are on opposite
sides  of this curve. For every $\epsilon > 0$, let $y_1^\epsilon$, $y^\epsilon_2$ be the points at distance
$\epsilon$ from $x_1$ in respectively $l_1$, $l_2$. Consider the geodesic $\Delta^\epsilon_i$
at equal distance from $y_i^\epsilon$ and $x_1$ ($i=1,2$): it is orthogonal to $l_i$, hence not
tangent to $\hat{c}$. It follows that, for $\epsilon$ small enough, $\hat{c}$ contains a point
$p_i$ closer to $y_i^\epsilon$ than to $x_1$. Hence, $c(r_1)$ can be shortened be replacing
the part between $p_1$ and $p_2$ by the union of the projections of the geodesics $[p_i, y_i^\epsilon]$.
This shorter curve is contained in $A'(r_1)$: contradiction.

Hence $R < r_1$. In particular, $R$ is finite. For $\epsilon$ small enough, the annulus $A'(R+\epsilon)$ contains
an essential loop $c_\epsilon$ of minimal length $< l(R+\epsilon)$. Since it lies in $A'(R)$, this loop has
length $\geq l(R)$. On the other hand, there is $\alpha>0$ such that any essential loop in $A'(R+\epsilon)$
contained in the $\alpha$-neighborhood of $c(R+\epsilon)$ has length $\geq l(R+\epsilon) > l(R)$. It follows
that $c_\epsilon$ is disjoint from $c(R+\epsilon)$, and thus, is actually a geodesic loop.

The annulus $A_\epsilon$ bounded by $c_\epsilon$ and $c(R+\epsilon)$ cannot be regular: indeed, if it was,
its union with $A(R+\epsilon)$ would be a regular hyperbolic annulus bounded by two closed geodesics.
Therefore, $A_\epsilon$ contains a singularity. Let $A_1$ be the annulus bounded by $C$ and $c_\epsilon$:
every essential loop inside $A_1$ has length $\geq l(R)$
(since it lies in $A'(R)$). It contains strictly less singularities than $A_0$. If we restart the process
from this annulus, we obtain by induction an annulus bounded by $C$ and a closed geodesic inside $T$
with no singularity.
\end{proof}

\subsection{De Sitter regions}
\label{sub.dSclass}

Let $T$ be a de Sitter region of $\Sigma$. We recall that $\Sigma$ is assumed to be positive, \textit{i.e.} that
all non-time-like singularities of non-vanishing degree have degree $2$ and are positive. This last feature will
be essential in our study (cf. Remark~\ref{rk:pospos}).

Future oriented isotropic directions defines two oriented line fields
on the regular part of $T$, defining two oriented foliations. Since we assume that $\Sigma$ is causal, space-like singularities
have degree $2$, and these foliations extend continuously on sigularities (but not differentially) as regular oriented foliations.
Besides, in the neighborhood of every BTZ-like singularity $x$, the leaves of each of these foliations spiral around
$x$. They thus define two singular oriented foliations $\cF_{1}$, $\cF_{2}$, where the singularities are precisely
the BTZ-like singularitie, \textit{i.e.} hyperbolic time-like ones, and have degree $+1$. By Poincar\'e-Hopf index formula
we immediatly get:

\begin{cor}
Every de Sitter region is homeomorphic to the  annulus, the disk or the sphere.
Moreover, it contains at most two BTZ-like singularities. If it contains two such singularities,
it is homeomorphic to the 2-sphere, and if it contains exactly one BTZ-like singularity, it is homeomorphic to the disk.
\end{cor}

Let $c: \R \to L$ be a parametrization of a leaf $L$ of $\cF_i$, increasing with respect to the time orientation.
Recall that the $\alpha$-limit set (respectively $\omega$-limit set) is the set of points in $T$ which are limits
of a sequence $(c(t_n))_{(n \in \N)}$, where $(t_n)_{(n \in \N)}$ is a decreasing (respectively an increasing) sequence of real numbers.
By assumption, $T$ contains no CCC. Hence, according to Poincar\'e-Bendixson Theorem:

\begin{cor}
\label{cor.Lclosed}
For every leaf $L$ of $\cF_{1}$ or  $\cF_{2}$, oriented by its time orientation, the $\alpha$-limit set
(resp. $\omega$-limit set) of $L$ is either empty or a past (resp. future) BTZ-like singularity.
Moreover, if the $\alpha$-limit set (resp. $\omega$-limit set) is empty, the leaf accumulates in the past (resp. future)
direction to a past (resp. future) boundary component of $T$ that is a point in a circle of photons,
or a extreme BTZ-like singularity.
\end{cor}

\begin{prop}
\label{pro.hyphyp}
Let $\Sigma$ be a positive, causal singular HS-surface.
Let $T$ be a de Sitter component of $\Sigma$ adjacent to a hyperbolic region $H$ along a circle of photons $C$. If
the completion $\bar{H}$ of $H$ is not homeomorphic to the disk, then
either $T$ is a disk containing exactly one BTZ-like singularity, or the boundary of $T$ in $\Sigma$ is the disjoint union
of $C$ and one extreme BTZ-like singularity.
\end{prop}

\begin{proof}
If $T$ is a disk, we are done. Hence we can assume
that $T$ is homeomorphic to the annulus. Inverting the time if necessary we also can assume
that $H$ is a past hyperbolic component. Let $C'$ be the other connected boundary component of $T$, \textit{i.e.} its
future boundary. If $C'$ is an extreme BTZ-like singularity, the proposition is proved. Hence we are
reduced to the case where $C'$ is a circle of photons.

According to Corollary~\ref{cor.Lclosed} every leaf of $\cF_{1}$ or $\cF_{2}$ is a closed
line joining the two boundary components of $T$.
For every singularity $x$ in $T$, or every light-like singularity in $C$, let $L_{x}$ be the future oriented half-leaf of $\cF_{1}$
emerging from $x$. Assume that $L_{x}$ does not contain any other singularity. Cut along
$L_{x}$: we obtain a $\dS^2$-surface $T^\ast$ admitting in its boundary two copies of $L_{x}$. Since $L_{x}$ accumulates to
a point in
$C'$ it develops in $\dS^2$ into a geodesic ray touching $\partial\HH^2$. In
particular, we can glue the two copies of $L_{x}$ in the boundary of $T^\ast$ by an isometry fixing their common
point $x$. For the appropriate choice of this glueing map, we obtain
a new $\dS^2$-spacetime where $x$ has been replaced by a regular point: we call this process,
well defined, \textit{regularization at $x$} (see figure~\ref{fig.regul}).

\begin{figure}[ht]
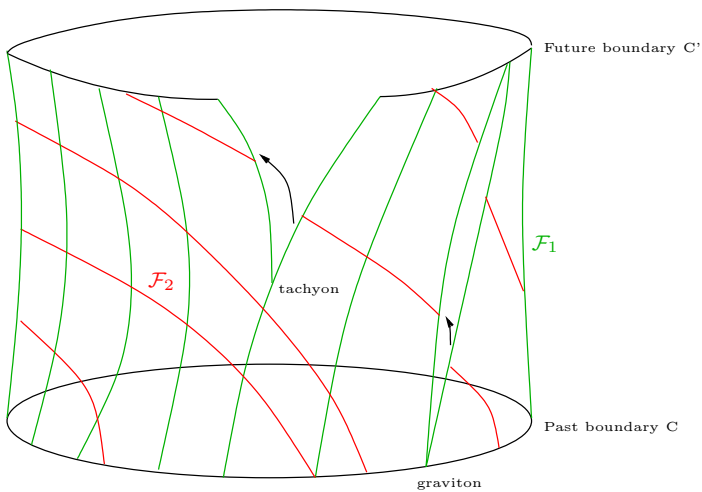

\begin{center}
\input regul.pstex_t
\end{center}
\caption{Regularization of a tachyon and a light-like singularity.}
\label{fig.regul}
\end{figure}

After a finite number of regularizations, we obtain a regular $\dS^2$-spacetime
$T'$. Moreover, all these surgeries can actually be performed on $T \cup C \cup H$: the de Sitter annulus $A'$ can be glued
to $H \cup C$, giving rise to a HS-surface containing the circle of photons $C$ disconnecting the hyperbolic region $H$ from
the regular de Sitter region $T'$ (however, the other boundary component $C'$ have been modified and do not match anymore
the other hyperbolic region adjacent to $T$). Moreover, the circle of photons $C$ now contains no light-like singularity, hence its $\R\PP^1$-structure from
the de Sitter side coincide with the $\R\PP^1$-structure from the hyperbolic side. According to Proposition~\ref{pro.hypdegree0}
this structure is hyperbolic of degree $0$: it is the quotient of an interval $I$ of $\R\PP^1$ by a hyperbolic element $\gamma_{0}$,
with no fixed point inside $I$.

Denote by $\cF'_{1}$, $\cF'_{2}$ the isotropic foliations in $T'$.
Since we performed the surgery along half-leaves of $\cF_{1}$, leaves of $\cF'_{1}$ are still closed in $T'$. Moreover,
each of them accumulates at a unique point in $C$: the space of leaves of $\cF'_{1}$ is identified with $C$. Let
$\widetilde{T}'$ the universal covering of $T'$, and $\widetilde{\cF}'_{1}$ the lifting of $\cF_{1}$.
Recall that $\dS^2$ is naturally identified with $\R\PP^1 \times \R\PP^1 \setminus \kD$, where $\kD$ is the diagonal.
The developing map $\cD: \widetilde{T}' \to \R\PP^1 \times \R\PP^1 \setminus \kD$ maps every leaf of
$\widetilde{\cF}'_{1}$ into a fiber $\{ \ast \} \times \R\PP^1$. Besides, as affine lines, they are complete affine lines,
meaning that they still develop onto the entire geodesic $\{ \ast \} \times (\R\PP^1 \setminus \{ \ast \})$.
It follows that $\cD$ is a homeomorphism between $\widetilde{T}'$ and the open domain $W$ in $\R\PP^1 \times \R\PP^1 \setminus \kD$
comprising points with first component in the interval $I$, \textit{i.e.} the region in $\dS^2$ bounded by two $\gamma_{0}$-invariant
isotropic geodesics. Hence $T'$ is isometric to the quotient of $W$ by $\gamma_{0}$, which is well understood (see Figure~\ref{fig.W};
it has been more convenient to draw the lift $W$ in the region in $\uRP \times \uRP$ between the graph of the identity
map and the translation $\delta$, region which is isomorphic to
the universal cover of $\R\PP^1 \times \R\PP^1 \setminus \kD$).

\begin{figure}[ht]
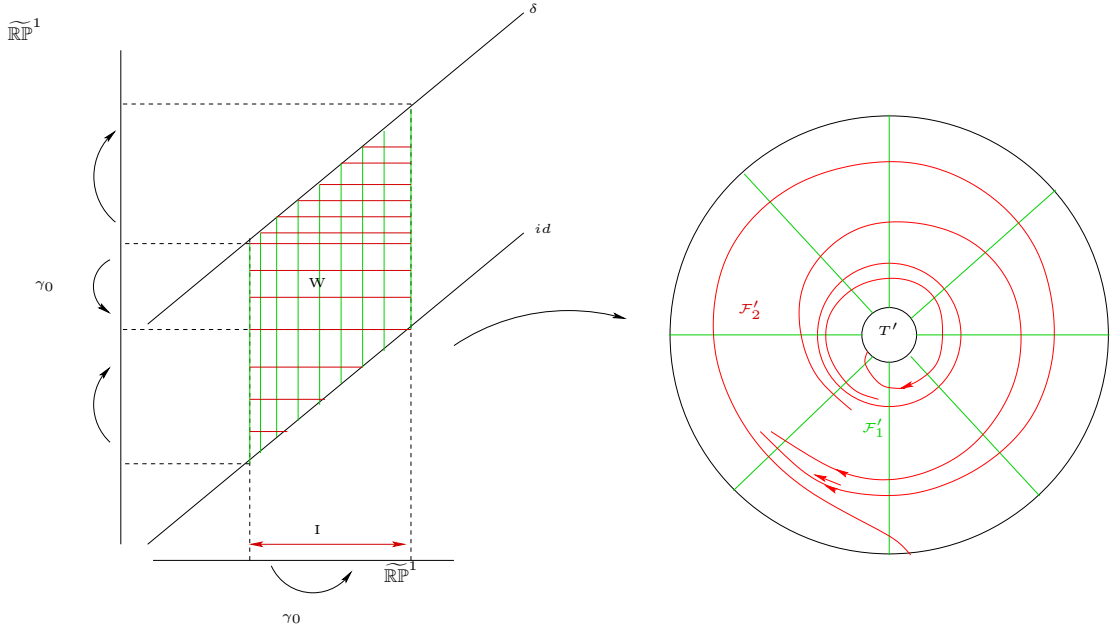

\begin{center}
\input W.pstex_t
\end{center}
\caption{The domain $W$ and its quotient $T'$.}
\label{fig.W}
\end{figure}

Hence the foliation $\cF'_{2}$ admits two compact leaves.
These leaves are CCC, but it is not yet in contradiction with the fact that $\Sigma$ is causal,
since the regularization might create such CCC.

The regularization
procedure is invertible and $T$ is obtained from $T'$ by \textit{positive} surgeries along future oriented
half-leaves of $\cF'_{1}$, \textit{i.e.} obeying the rules described in Remark~\ref{rk.cuttachyon2}.
We need to be more precise: pick a leaf $L'_{1}$ of $\cF'_{1}$. It corresponds to a vertical line in $W$ depicted
in Figure~\ref{fig.W}. We consider the first return $f'$ map from $L'_{1}$ to $L'_{1}$ along future oriented leaves of $\cF'_{2}$:
it is defined on an interval $]-\infty, x_{\infty}[$ of $L'_{1}$, where $-\infty$ corresponds to the end of $L'_{1}$ accumulating on
$C$. It admits two fixed points $x_{1} < x_{2} < x_{\infty}$, corresponding to the two compact leaves
of $\cF'_2$. The former is attracting and the latter is repelling.
Let $L_{1}$ be a leaf of $\cF_{1}$ corresponding, by the reverse surgery, to $L'_{1}$. We can assume
without loss of generality that $L_{1}$ contains no singularity. Let $f$ be the first return map from $L_{1}$ into itself
along future oriented leaves of $\cF_{2}$. There is a natural identification between $L_{1}$ and $L'_{1}$, and since all light-like singularities and tachyons in
$T \cup C$ are positive, \textit{the deviation of $f$ with respect to $f'$ is in the past direction,} \textit{i.e.} for every $x$ in
$L_{1} \approx L'_{1}$ we have $f(x) \leq f'(x)$ (it includes the case where $x$ is not in the domain of definition of $f$,
in which case, by convention, $f(x) = \infty$). In particular, $f(x_{2}) \leq x_{2}$. It follows that the future part of the oriented leaf
of $\cF_{2}$ through $x_2$  is trapped below its portion between $x_{2}$, $f(x_{2})$. Since it is closed, and not compact, it must accumulate on $C$.
But it is impossible since future oriented leaves near $C$ exit from $C$, intersect a space-like loop, and cannot go back because
of orientation considerations. The proposition is proved.
\end{proof}

\begin{figure}[ht]
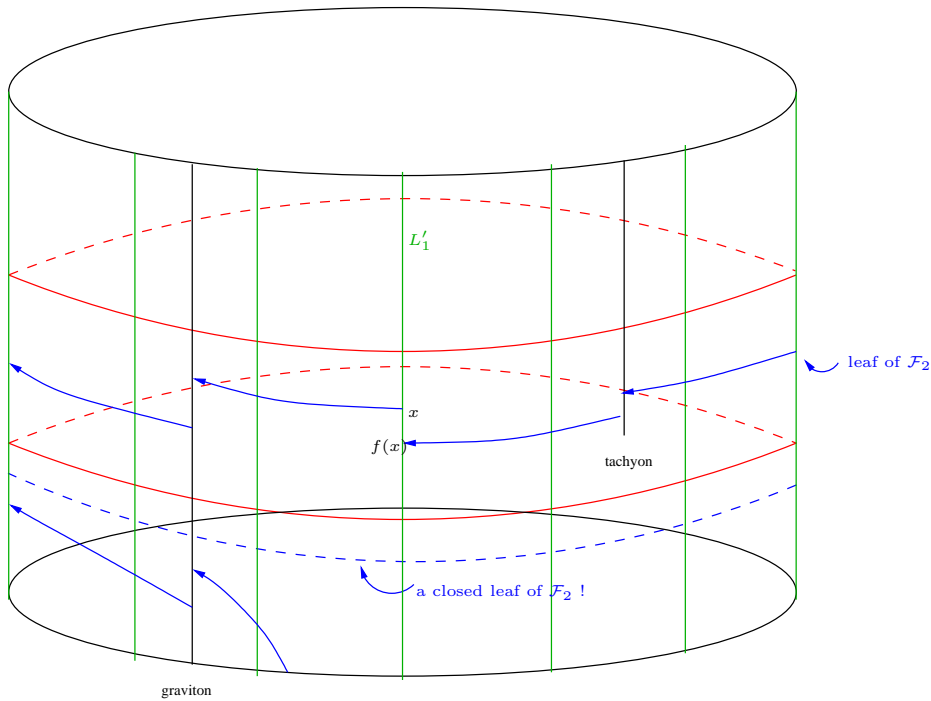

\begin{center}
\input return.pstex_t
\end{center}
\caption{First return maps. The identification maps along lines above time-like and light-like singularities
compose the almost horizontal broken arcs which are contained in leaves of $\mathcal{F}_2$. }
\label{fig.return}
\end{figure}

\begin{remark}
\label{rk:pospos}
In Proposition~\ref{pro.hyphyp} the positivity hypothesis is necessary. Indeed, consider a regular HS-surface
made of one annular past hyperbolic region connected to one annular future hyperbolic region by two de Sitter regions
isometric to the region $T' = W/\langle\gamma_{0}\rangle$ appearing in the proof of Proposition~\ref{pro.hyphyp}.
Pick up a photon $x$ in the past boundary of one of these de Sitter components $T$, and let $L$ be the leaf of $\cF_{1}$
accumulating in the past to $x$. Then $L$ accumulates in the future to a point $y$ in the future boundary component.
Cut along $L$, and glue back by a parabolic isometry fixing $x$ and $y$. The main argument in the proof above is that
if this surgery is performed in the positive way, so that $x$ and $y$ become positive tachyons, then the resulting spacetime
still admits two CCC, leaves of the foliation $\cF_{2}$. But if the surgery is performed in the negative way, with a sufficiently
big parabolic element, the closed leaves of $\cF_{2}$ in $T$ are destroyed, and every leaf of the new foliation $\cF_{2}$
in the new singular surface joins the two boundary components of the de Sitter region, which is therefore causal.
\end{remark}

\begin{theorem} \label{tm:thierry}
Let $\Sigma$ be a singular causal positive HS-surface, homeomorphic to the sphere. Then, it admits at most one past hyperbolic
component, and at most one future hyperbolic component. Moreover, we are in one of the following mutually exclusive situations:

\begin{enumerate}

\item \textit{Causally regular case: } There is a unique de Sitter component, which is an annulus connecting one past hyperbolic region
homeomorphic to the disk to a future hyperbolic region homeomorphic to the disk.

\item  \textit{Interaction of black holes or white holes: } There is no past or no future hyperbolic region, and every de Sitter region is a
either a disk containing a unique BTZ-like singularity, or a disk with an extreme BTZ-like singularity removed.

\item \textit{Big Bang and Big Crunch: } There is no de Sitter region, and only one hyperbolic region, which is a
singular hyperbolic sphere - if the time-like region is a future one, the singularity is called a Big Bang; if the time-like region is a
past one, the singularity is a Big Crunch.

\item \textit{Interaction of a white hole with a black hole: } There is no hyperbolic region. The surface $\Sigma$ contains one past BTZ-like singularity
and one future BTZ-like singularity - these singularities may be extreme or not.

\end{enumerate}
\end{theorem}

\begin{remark}
This Theorem, despite of the terminology inspired
from cosmology, has no serious pretention of relevance for physics. However these
appelations have the advantage to provide a reasonable intuition on the geometry of
the interaction. For example, in what is called a Big Bang, the spacetime is entirely contained in
the future of the singularity, and the singularity lines can be seen as massive particles or ``photons''
emitted by the initial singularity.

Actually, it is one of few examples suggesting that the prescription of the surface $\Sigma$
to be a sphere could be relaxed: whereas it seems hard to imagine that the spacetime could fail
to be a manifold at a singular point describing
a collision of particles, it is nevertheless not so hard, at least for us, to admit
that the topology of the initial singularity may be
more complicated, as it is the case in the regular case (see \cite{mess-notes}).
\end{remark}

\begin{proof}
If the future hyperbolic region and the past hyperbolic region is not empty, there must be a de Sitter annulus
connecting one past hyperbolic component to a future hyperbolic component. By Proposition \ref{pro.hyphyp}
these hyperbolic components are disks: we are in the causally regular case.

If there is no future hyperbolic region, but one past hyperbolic region, and at least one de Sitter region,
then there cannot be any annular de Sitter component
connecting two hyperbolic regions. Hence, the closure of each de Sitter component is a closed disk.
It follows that there is only one past hyperbolic component:
$\Sigma$ is an interaction of black holes. Similarly, if there is a de Sitter region, a future hyperbolic region but no past hyperbolic region,
$\Sigma$ is an interaction of white holes.

The remaining situations are the cases where $\Sigma$ has no de Sitter region, or no hyperbolic region. The former
case corresponds obviously to the description (3) of Big Bang or Big Crunch , and the latter to the description (4) of
an interaction between one black hole and one white hole.
\end{proof}

\begin{remark}
It is easy to construct singular hyperbolic spheres, \textit{i.e.} Big Bang or Big Crunch: take for example the double
of a hyperbolic triangle. The existence of interactions of a white hole with black hole is slightly less obvious.
Consider the HS-surface $\Sigma_{m}$ associated to the BTZ black hole $\mathcal{B}_{m}$. It can be described as follows:
take a point $p$ in $\dS^2$, let $d_{1}$, $d_{2}$ be the two projective circles in $\HS$ containing $p$, its opposite
$-p$, and tangent to $\partial\HH^2_{\pm}$. It decomposes $\HS^2$ in four regions. One of these components,
that we denote by $U$, contains the past hyperbolic region $\HH^2_{-}$. Then, $\Sigma_{m}$ is the quotient
of $U$ by the group generated by a hyperbolic isometry $\gamma_{0}$ fixing $p$, $-p$, $d_{1}$ and $d_{2}$.
Let $x_{1}$, $x_{2}$ be the points where $d_{1}$, $d_{2}$ are tangent to $\partial\HH^2_{-}$, and let $I_{1}$,
$I_{2}$ be the connected components of $\partial\HH^2_{-} \setminus \{ x_{1}, x_{2} \}$. We select the index
so that $I_{1}$ is the boundary of the de Sitter component $T_{1}$ of $U$ containing $p$. Now let $q$ be a point in
$T_{1}$ so that the past of $q$ in $T_{1}$ has a closure in $U$ containing a fundamental domain $J$ for the action of
$\gamma_{0}$ on $I_{1}$. Then there are two time-like geodesic rays starting from $q$ and accumulating at points
in $I_{1}$ which are extremities of a subinterval containing $J$. These rays project in $\Sigma_{m}$ onto two time-like
geodesic rays $l_{1}$ and $l_{2}$ starting from the projection $\bar{q}$ of $q$.
These rays admit a first intersection point $\bar{q}'$ in the past of $\bar{q}$. Let $l'_{1}$, $l'_{2}$
be the subintervalls in respectively $l_{1}$, $l_{2}$ with extremities $\bar{q}$, $\bar{q}'$: their union is a circle disconnecting
the singular point $\bar{p}$ from the boundary of the de Sitter component. Remove the component of $\Sigma \setminus (l'_1 \cup l'_2)$
adjacent to this boundary.
If $\bar{q}'$ is well-chosen, $l'_{1}$ and $l'_{2}$ have the same proper time. Then we can glue one to the other by a hyperbolic isometry. The resulting spacetime is as required an interaction between a BTZ black hole corresponding to $\bar{p}$ with a white hole corresponding
to $\bar{q}'$ - it contains also a tachyon of positive mass corresponding to $\bar{q}$.
\end{remark}

\section{Global hyperbolicity}
\label{sc:hyperbolicity}

In previous sections, we considered local properties of AdS manifolds with particles.
We already observed in section~\ref{sub.futpast}
that the usual notions of causality (causal curves, future, past, time functions...)
in regular Lorentzian manifolds still hold.
In this section, we consider the global character of causal properties of AdS manifolds with particles.
The main point presented here is that, as long as no interaction  appears, global hyperbolicity is
still a meaningfull notion for singular AdS spacetimes. This notion will be necessary in Section
6, as well as in the continuation of this paper \cite{colII} (see also the final part of
\cite{collision}).

In all this section $M$ denotes a singular AdS manifold admitting
as singularities only massive particules and no interaction.
The regular part of $M$ is denoted by $M^\ast$ in this section.
Since we will consider other Lorentzian metrics on
$M$, we need a denomination for the singular AdS metric~: we denote it $g_0$.

\subsection{Local coordinates near a singular line}
\label{sub.localcoord}

Causality notions only depend on the conformal class of the metric, and
AdS is conformally flat. Hence, AdS spacetimes and flat spacetimes share
the same local causal properties. Every regular AdS spacetime admits
an atlas for which local coordinates have the form $(z, t)$,
where $z$ describes the unit disk $D$ in the complex plane, $t$
the interval $]-1, 1[$ and such that the AdS metric is conformal to:
$$-dt^2 + \left|dz\right|^2~. $$

For the singular case considered here, any point $x$ lying on
a singular line $l$ (a massive particule of mass $m$), the same
expression holds, but we have to remove a wedge $ \{ 2\alpha\pi <Arg(z) < 2\pi \}$ where
$\alpha=1-m$ is positive, and to glue the two sides of this wedge.
Consider the map $z \to \zeta = z^{1/\alpha}$: it sends the disk $D$ with a wedge removed
onto the entire disk, and is compatible with the glueing of the sides of the wedge.
Hence, a convenient
local coordinate system near $x$ is $(\zeta, t)$ where $(\zeta,t)$ still lies in $D \times ]-1, 1[$.
The singular AdS metric is then, in these coordinates, conformal to:
$$ (1-m)^2 \frac{\left|d\zeta\right|^2}{\left|\zeta\right|^{2m}} - dt^2~. $$

In these coordinates, future oriented causal curves can be parametrized by
the time coordinate $t$, and satisfies:
$$ \frac{\left|\zeta'(t)\right|}{\left|\zeta\right|^{m}} \leq \frac{1}{1-m}~. $$

Observe that all these local coordinates define a differentiable atlas on the
topological manifold $M$ for which the AdS metric on the regular part is smooth.

\subsection{Achronal surfaces}
Usual definitions in regular Lorentzian manifolds still apply to the singular AdS spacetime $M$:

\begin{defi}
A subset $S$ of $M$ is \textit{achronal} (resp. \textit{acausal}) if there is no non-trivial time-like (resp. causal)
curve joining two points in $S$.
It is only \textit{locally achronal} (resp. \textit{locally acausal}) if every point in $S$ admits a neighborhood $U$ such that
the intersection $U \cap S$ is achronal (resp. acausal) inside $U$.
\end{defi}

Typical examples of locally acausal subsets are space-like surfaces, but the definition above
also includes non-differentiable "space-like" surfaces, with only Lipschitz regularity. Lipschitz
space-like surfaces provide actually the general case if one adds the {\it edgeless} assumption~:

\begin{defi}
A locally achronal subset $S$ is \textit{edgeless} if every point $x$ in $S$ admits a neighborhood
$U$ such that every causal curve in $U$ joining one point of the past of $x$ (inside $U$)
to a point in the future (in $U$) of $x$ intersects $S$.
\end{defi}

In the regular case, closed edgeless locally achronal subsets are embedded locally Lipschitz surfaces.
More precisely, in the coordinates $(z,t)$ defined in section~\ref{sub.localcoord}, they are graphs
of $1$-Lipschitz maps defined on $D$.

This property still holds in $M$, except the locally Lipschitz property which is not valid anymore at
singular points, but only a weaker weighted version holds:
closed edgeless acausal subsets containing $x$ corresponds to H\"{o}lder functions
$f: D \to ]-1, 1 [$ differentiable almost everywhere and
satisfying:
$$\Vert d_\zeta f \Vert < \frac{\left|\zeta\right|^{-m}}{1-m}~. $$
Go back to the coordinate system $(z,t)$.
The acausal subset is then the graph of a $1$-Lipschitz map $\varphi$ over
the disk minus the wedge. Moreover, the values of $\varphi$ on the boundary of
the wedge must coincide since they have to be send one to the other by the rotation
performing the glueing. Hence, for every $r < 1$:
$$\varphi(r)=\varphi(re^{i2\alpha\pi})~. $$
We can extend $\varphi$ over the wedge by defining $\varphi(re^{i\theta}) = \varphi(r)$
for $2\alpha\pi \leq \theta \leq 2\pi$. This extension over the entire $D \setminus \{ 0 \}$ is then clearly
$1$-Lipschitz. It therefore extends to $0$. We have just proved:

\begin{lemma}
The closure of any closed edgeless achronal subset of $M^\ast$
is a closed edgeless achronal subset of $M$.
\end{lemma}

\begin{defi}
A space-like surface $S$ in $M$ is a closed edgeless locally acausal subset whose
intersection with the regular part $M^\ast$ is a smooth embedded space-like surface.
\end{defi}

\subsection{Time functions}
\label{sub.singtime}

As in the regular case, we can define time functions as maps $T: M \to \RR$ which
are strictly increasing along any future oriented causal curve. For non-singular
spacetimes the existence is related to \textit{stable causality~:}

\begin{defi}
Let $g$, $g'$ be two Lorentzian metrics on the same manifold $X$.
Then, $g'$ dominates $g$ if
every causal tangent vector for $g$ is time-like for $g'$. We denote
this relation by $g \prec g'$.
\end{defi}

\begin{defi}
A Lorentzian metric $g$ is stably causal if there is a metric $g'$ such that
$g \prec g'$, and such that $(X, g')$ is chronological, i.e. admits no periodic time-like
curve.
\end{defi}

\begin{theorem}[See \cite{beem}]
\label{thm.stabletime}
A Lorentzian manifold $(M, g)$ admits a time function if and only if it is
\textit{stably causal.} Moreover, when a time function exists, then
there is a smooth time function.
\end{theorem}

\begin{remark}
\label{rk.stablecausal}
In section~\ref{sub.localcoord} we defined some differentiable atlas on
the manifold $M$. For this differentiable structure, the null cones of $g_0$
degenerate along singular lines to half-lines tangent to the "singular" line
(which is perfectly smooth for the selected differentiable atlas). Obviously,
we can extend the definition of domination to the more general case
$g_0 \prec g$ where $g_0$ is our singular metric and $g$ a smooth regular
metric. Therefore, we can define the stable causality of in this context:
$g_0$ is stably causal if there is a smooth Lorentzian metric $g'$ which is
chronological and such that $g_0 \prec g'$.
Theorem~\ref{thm.stabletime} is still valid in this more general context.
Indeed, there is a smooth Lorentzian metric $g$ such that $g_0 \prec g \prec g'$,
which is stably causal since $g$ is dominated by the achronal metric $g'$.
Hence there is a time function
$T$ for the metric $g$, which is still a time function for $g_0$ since $g_0 \prec g$:
causal curves for $g_0$ are causal curves for $g$.
\end{remark}

\begin{lemma}
\label{le.singstable}
The singular metric $g_0$ is stably causal if and only if its restriction
to the regular part $M^\ast$ is stably causal. Therefore, $(M, g_0)$ admits
a smooth time function if and only if $(M^\ast, g_0)$ admits a time function.
\end{lemma}

\begin{proof}
The fact that $(M^\ast, g_0)$ is stably causal as soon as $(M, g_0)$ is stably
causal is obvious. Let us assume that $(M^\ast, g_0)$ is stably causal:
let $g'$ be smooth chronological Lorentzian metric on $M^\ast$ dominating
$g_0$. On the other hand, using the local models around singular lines,
it is easy to construct an chronological Lorentzian metric $g''$ on a tubular neighborhood $U$
of the singular locus of $g_0$ (the fact that $g'$ is chronological implies that
the singular lines are not periodic). Actually, by reducing the tubular neighborhood $U$
and modyfing $g''$ therein, one can assume that
$g'$ dominates $g''$ on $U$. Let $U' $ be a smaller tubular neighborhood of
the singular locus such that $\overline{U}' \subset U$, and let $a$, $b$ be
a partition of unity subordonate to $U$, $M \setminus U'$. Then
$g_1 = ag'' + bg'$ is a smooth Lorentzian metric dominating $g_0$.
Moreover, we also have $g_1 \prec g'$ on $M^\ast$. Hence any time-like curve
for $g_1$ can be slightly perturbed to a time-like curve for $g'$
avoiding the singular lines. It follows that $(M, g_0)$ is stably causal.
\end {proof}

\subsection{Cauchy surfaces}

\begin{defi}
A space-like surface $S$ is a Cauchy surface if it is acausal and intersects every inextendible causal curve in $M$.
\end{defi}

Since a Cauchy surface is acausal, its future $I^+(S)$ and its past $I^-(S)$ are disjoint.

\begin{remark}
\label{rk.ghnongh}
The regular part of a Cauchy surface in $M$ is not a Cauchy surface in
the regular part $M^\ast$, since causal curves can exit the regular region through a
time-like singularity.
\end{remark}

\begin{defi}
A singular AdS spacetime is globally hyperbolic if it admits a Cauchy surface.
\end{defi}

\begin{remark}
We defined Cauchy surfaces as smooth objects for further requirements in this paper,
but this definition can be generalized for non-smooth locally achronal closed subsets.
This more general definition leads to the same notion of globally hyperbolic spacetimes,
i.e. singular spacetimes admitting a non-smooth Cauchy surface also admits a smooth one.
\end{remark}

\begin{prop}
\label{pro.ghtime}
Let $M$ be a singular AdS spacetime without interaction and with singular
set reduced to massive particles. Assume that $M$ is globally hyperbolic.
Then $M$ admits a time function $T: M \to \RR$ such that every level
$T^{-1}(t)$ is a Cauchy surface.
\end{prop}

\begin{proof}
This is a well-known theorem by Geroch in the regular case, even for
general globally hyperbolic spacetimes without compact Cauchy surfaces (\cite{gerochdependence}).
But, the singular version does not follow immediately by applying this
regular version to $M^\ast$ (see Remark~\ref{rk.ghnongh}).

Let $l$ be an inextendible causal curve in $M$. It intersects the Cauchy surface $S$,
and since $S$ is achronal, $l$ cannot be periodic. Therefore, $M$ admits no periodic
causal curve, i.e. is \textit{acausal.}


Let $U$ be a small tubular neighborhood of $S$ in $M$, such that the boundary $\partial U$
is the union of two space-like hypersurfaces $S_-$, $S_+$ with $S_- \subset I^-(S)$,
$S_+ \subset I^+(S)$,
and such that every inextendible future oriented causal curve in $U$ starts from
$S_-$, intersects $S$ and then hits $S^+$.
Any causal curve starting from $S_-$ leaves immediatly $S_-$,
crosses $S$ at some point $x'$, and then cannot cross $S$ anymore. In particular,
it cannot go back in the past of $S$ since $S$ is acausal, and thus, does not reach anymore $S_-$. Therefore,
$S_-$ is acausal. Similarly, $S_+$ is acausal. It follows that
$S_\pm$ are both Cauchy surfaces for $(M, g_0)$.

For every $x$ in $I^+(S_-)$ and every past oriented $g_0$-causal tangent vector $v$,
the past oriented geodesic tangent to $(x,v)$ intersects $S$. The same property holds for
tangent vector $(x, v')$ nearby. It follows that there exists on $I^+(S_-)$ a smooth Lorentzian metric
$g'_1$ such that $g_0 \prec g'_1$ and such that every inextendible past oriented $g'_1$-causal
curve attains $S$. Furthermore, we can select $g'_1$ such that $S$ is $g'_1$-space-like,
and such that every future oriented $g'_1$-causal
vector tangent at a point of $S$ points in the $g_0$-future of $S$. It follows that future oriented $g'_1$-causal curves
crossing $S$ cannot come back to $S$: $S$ is acausal, not only for $g_0$, but
also for $g'_1$.

We can also define $g'_2$ in the past of $S_+$ so that
$g_0 \prec g'_2$, every inextendible future oriented $g'_2$-causal
curve attains $S$, and such that $S$ is $g'_2$-acausal. We can now interpolate in the common region
$I^+(S_-) \cap I^-(S_+)$, getting a Lorentzian metric $g'$ on the entire $M$
such that $g_0 \prec g' \prec g'_1$ on $I^+(S_-)$, and
$g_0 \prec g' \prec g'_2$ on $I^-(S_+)$. Observe that even if it is not totally
obvious that the metrics $g'_i$ can be selected continuous, we have enough room to
pick such a metric $g'$ in a continuous way.

Let $l$ be a future oriented $g'$-causal curve starting from a point in $S$. Since $g' \prec g'_1$, this
curve is also $g'_1$-causal as long as it remains inside $I^+(S_-)$. But since $S$ is acausal for
$g'_1$, it implies that $l$ cannot cross $S$ anymore: hence $l$ lies entirely in $I^+(S)$.
It follows that $S$ is acausal for $g'$.

By construction of $g'_1$, every past-oriented $g'_1$-causal curve starting from a point inside $I^+(S)$
must intersect $S$. Since $g' \prec g'_1$ the same property holds for $g'$-causal curves.
Using $g'_2$ for points in $I^+(S_-)$, we get that every inextendible $g'$-causal curve intersects
$S$. Hence, $(M, g')$ is globally hyperbolic. According to Geroch's Theorem in the regular case,
there is a time function $T: M \to \RR$ whose levels are Cauchy surfaces. The proposition
follows, since $g_0$-causal curves are $g'$-causal curves, implying that $g'$-Cauchy surfaces
are $g_0$-Cauchy surfaces and that $g'$-time functions are $g_0$-time functions.
\end{proof}

\begin{cor}
\label{cor.splitting}
If $(M, g_0)$ is globally hyperbolic, there is a decomposition $M \approx S \times \RR$
where every level $S \times \{ \ast \}$ is a Cauchy surface, and very vertical line
$\{ \ast \} \times \RR$ is a singular line or a time-like line.
\end{cor}

\begin{proof}
Let $T: M \to \RR$ be the time function provided by Proposition~\ref{pro.ghtime}.
Let $X$ be minus the gradient  (for $g_0$) of
$T$: it is a future oriented time-like vector field on $M^\ast$. Consider also a future oriented
time-like vector field $Y$ on a tubular neighborhood $U$ of the singular locus: using a partition
of unity as in the proof of Lemma~\ref{le.singstable}, we can construct a smooth time-like vector field
$Z = aY + bX$ on $M$ tangent to the singular lines. The orbits of the flow generated by $Z$
are time-like curves. The global hyperbolicity of $(M, g_0)$ ensures that each of these orbits
intersect every Cauchy surface, in particular, the level sets of $T$. In other words, for every $x$ in $M$
the $Z$-orbit of $x$ intersects $S$ at a point $p(x)$. Then the map
$F: M \to S \times \RR$ defined by $F(x) = (p(x), T(x))$ is the desired diffeomorphism
between $M$ and $S \times \RR$.
\end{proof}

\subsection{Maximal globally hyperbolic extensions}

From now we assume that $M$ is globally hyperbolic, admitting a compact Cauchy surface $S$.
In this section, we prove the following facts, well-known in the case of regular globally hyperbolic solutions
to the Einstein equation (\cite{gerochdependence}): \textit{there exists a maximal extension, which is unique up to isometry.}

\begin{defi} An isometric embedding $i: (M, S) \to (M', S')$ is a Cauchy embedding if
$S'=i(S)$ is a Cauchy surface of $M'$.
\end{defi}

\begin{remark}
If $i: M \to M'$ is a Cauchy embedding then the image $i(S')$ of any Cauchy surface $S'$
of $M$ is also a Cauchy surface in $M'$. Indeed, for every inextendible causal curve $l$
in $M'$, every connected component of the preimage $i^{-1}(l)$ is an inextendible causal
curve in $M$, and thus intersects $S$. Since $l$ intersects $i(S)$ in exactly one point,
$i^{-1}(l)$ is connected. It follows that the intersection $l \cap i(S')$ is non-empty and reduced to a single point: $i(S')$ is a Cauchy surface.

Therefore, we can define Cauchy embeddings without reference to the selected Cauchy surface
$S$. However, the natural category is the category of \textit{marked} globally hyperbolic
spacetimes, i.e. pairs $(M, S)$.
\end{remark}

\begin{lemma}
\label{le.coincide}
Let $i_1: (M, S) \to (M', S')$, $i_2: (M, S) \to (M', S')$ two Cauchy embeddings into the same marked globally hyperbolic singular AdS spacetime $(M', S')$. Assume that
$i_1$ and $i_2$ coincide on $S$. Then, they coincide on the entire $M$.
\end{lemma}

\begin{proof}
If $x'$, $y'$ are points in $M'$ sufficiently near to $S'$, say, in the future of $S'$,
then they are equal if and only if the intersections $I^-(x') \cap S'$ and
$I^-(y') \cap S'$ are equal. Apply this observation to $i_1(x)$, $i_2(x)$ for $x$
near $S$: we obtain that $i_1$, $i_2$ coincide in a neighborhood of $S$.

Let now $x$ be any point in $M$.
Since there is only a finite number of singular lines
in $M$, there is a time-like geodesic segment $[y, x]$, where $y$ lies in $S$,
and such that $[y, x[$ is contained in $M^\ast$ ($x$ may be singular).
Then $x$ is the image by the exponential map of some $\xi$ in $T_yM$.
Then $i_1(x)$, $i_2(x)$ are the image by the exponential map of respectively
$d_yi_1(\xi)$, $d_yi_2(\xi)$. But these tangent vectors are equal, since $i_1 = i_2$
near $S$.
\end{proof}

\begin{lemma}
\label{le.causalconvex}
Let $i: M \to M'$ be a Cauchy embedding into a singular AdS spacetime. Then, the image of $i$ is
causally convex, i.e. any causal curve in $M'$ admitting extremities
in $i(M)$ lies inside $i(M)$.
\end{lemma}

\begin{proof}
Let $l$ be a causal segment in $M'$ with extremities in $i(M)$. We extend it as an inextendible causal curve $\hat{l}$.
Let $l'$ be a connected component of $\hat{l} \cap i(M)$: it is an inextendible causal curve
inside $i(M)$. Thus, its intersection with $i(S)$ is non-empty. But $\hat{l} \cap i(S)$ contains at most one point: it follows that $\hat{l} \cap i(M)$ admits only one connected component, which contains $l$.
\end {proof}

\begin{cor}
\label{cor.bordgh}
The boundary of the image of a Cauchy embedding $i: M \to M'$
is the union of two closed edgeless achronal subsets $S^+$, $S^-$ of $M'$,
and $i(M)$ is the intersection between the past of $S^+$ and the future of $S^-$.
\end{cor}

Each of $S^+$, $S^-$ might be empty, and is not necessarily connected.

\begin{proof}
This is a general property of causally convex open subsets: $S^+$ (resp. $S^-$) is the set of elements in the boundary of $i(M)$ whose past (resp. future) intersects $i(M)$. The proof is
straightforward and left to the reader.
\end{proof}

\begin{defi}
$(M, S)$ is maximal if every Cauchy embedding $i: M \to M'$ into a singular AdS spacetime is onto, i.e. an isometric homeomorphism.
\end{defi}

\begin{prop} \label{pr:extension1}
$(M, S)$ admits a maximal singular AdS extension, i.e. a Cauchy embedding into
a maximal globally hyperbolic singular AdS spacetime $(\widehat{M}, \hat{S})$ without interaction.
\end{prop}

\begin{proof}
Let $\mathcal M$ be set of Cauchy embeddings $i: (M, S) \to (M', S')$.
We define on $\mathcal M$ the relation $(i_1, M_1, S_1) \preceq (i_2, M_2, S_2)$ if
there is a Cauchy embedding $i: (M_1, S_1) \to (M_2, S_2)$ such that $i_2 = i \circ i_1$.
It defines a preorder on $\mathcal M$. Let $\overline{\mathcal M}$ be the space
of Cauchy embeddings up to isometry, i.e. the quotient space of the equivalence relation
identifying $(i_1, M_1, S_1)$ and $(i_2, M_2, S_2)$ if
there is an isometric homeomorphism $i: (M_1, S_1) \to (M_2, S_2)$ such that
$i_2 = i \circ i_1$. Then $\preceq$ induces on $\overline{\mathcal M}$ a preorder relation, that we still denote by $\preceq$. Lemma~\ref{le.coincide} ensures
that $\preceq$ is a partial order (if $(i_1, M_1, S_1) \preceq (i_2, M_2, S_2)$
and $(i_2, M_2, S_2) \preceq (i_1, M_1, S_1)$, then $M_1$ and $M_2$ are isometric
and represent the same element of $\overline{\mathcal M})$.
Now, any totally ordered subset
$A$ of $\overline{\mathcal M}$ admits an upper bound in $A$: the inverse limit
of (representants of) the elements of $A$. By Zorn Lemma, we obtain that $\overline{M}$ contains a maximal
element. Any representant in $\overline{\mathcal M})$ of this maximal element is a maximal extension
of $(M, S)$.
\end{proof}

\begin{remark}
The proof above is sketchy: for example, we did not justify the fact that the inverse
limit is naturally a singular AdS spacetime. This is however a straightforward verification,
the same as in the classical situation, and is left to the reader.
\end{remark}

\begin{prop} \label{pr:extension2}
The maximal extension of $(M, S)$ is unique up to isometry.
\end{prop}

\begin{proof}
Let $(\widehat{M}_1, S_1)$, $(\widehat{M}_2, S_2)$ be two maximal extensions of $(M,S)$.
Consider the set of globally hyperbolic singular AdS spacetimes $(M', S')$ for which there
is a commutative diagram as below, where arrows are Cauchy embeddings.

\begin{center}
$$ \xymatrix{
       &  & (\widehat{M}_1, S_1) \\
    (M, S) \ar[r] \ar[rru] \ar[rrd] & (M', S') \ar[ru] \ar[rd] & \\
       &  & (\widehat{M}_2, S_2)
  } $$
\end{center}

Reasoning as in the previous proposition, we get that this set admits a maximal element:
there is a marked extension $(M', S')$ of $(M,S)$, and Cauchy embeddings $\varphi_i: M' \to \widehat{M}_i$ which cannot be simultaneously extended.

Define $\widehat{M}$ as the union of $(\widehat{M}_1, S_1)$ and $(\widehat{M}_2, S_2)$,
identified along their respective embedded copies of $(M', S')$, through $\varphi:=\varphi_2 \circ \varphi_1^{-1}$, equipped with the quotient topology. The key point is to prove that
$\widehat{M}$ is Hausdorff. Assume not: there is a point $x_1$ in $\widehat{M}_1$,
a point $x_2$ in $\widehat{M}_2$, and a sequence $y_n$ in $M'$ such that $\varphi_i(y_n)$
converges to $x_i$, but such that $x_1$ and $x_2$ do not represent the same
element of $\widehat{M}$. It means that $y_n$ does not converge in $M'$, and that
$x_i$ is not in the image of $\varphi_i$. Let $U_i$ be small neighborhoods in $\widehat{M}_i$ of $x_i$.

Denote by $S^+_i$, $S^-_i$ the upper and lower boundaries of $\varphi_i(M')$ in $\widehat{M}_i$ (cf.
Corollary~\ref{cor.bordgh}). Up to time reversal, we can assume that $x_1$ lies
in $S^+_1$: it implies that all the $\varphi_1(y_n)$ lies in $I^-(S^+_1)$, and that,
if $U_1$ is small enough, $U_1 \cap I^-(x_1)$ is contained in $\varphi_1(M')$. It is an open subset, hence $\varphi$ extends to some AdS isometry $\overline{\varphi}$ between $U_1$ and $U_2$ (reducing the $U_i$ if necessary). Therefore, every $\varphi_i$ can be extended
to isometric embeddings $\overline{\varphi}_i$ of a spacetime $M''$ containing $M'$, so that
$$\overline{\varphi}_2 = \overline{\varphi} \circ \overline{\varphi}_1$$

We intend to prove that $x_i$ and $U_i$ can be chosen such that $S_i$ is a Cauchy surface in $\overline{\varphi}_i(M'') \cup U_i$. Consider past oriented causal curves, starting from $x_1$, and contained in $S^+_1$. They are partially ordered by the inclusion. According to Zorn lemma,
there is a maximal causal curve $l_1$ satisfying all these properties.
Since $S^+_1$ is disjoint from $S_1$, and since every inextendible causal curve crosses
$S$, the curve $l_1$ is not inextendible: it has a final endpoint $y_1$ belonging
to $S^+_1$ (since $S^+_1$ is closed).
Therefore, any past oriented causal curve starting from $y_1$ is disjoint from $S^+_1$
(except at the starting point $y_1$).

We have seen that $\varphi$ can be extended over in a neighborhood of $x_1$: this extension
maps the initial part of $l_1$ onto a causal curve in $\widehat{M}_2$ starting from $x_2$ and contained in $S^+_2$. By compactness of $l_1$, this extension can be performed along the entire $l_1$, and the image is a causal curve admitting a final point $y_2$ in $S^+_2$. The points
$y_1$ and $y_2$ are not separated one from the other by the topology of $\widehat{M}$.
Replacing $x_i$ by $y_i$, we can thus assume that \textit{every past oriented causal curve
starting from $x_i$ is contained in $I^-(S^+_i)$.} It follows that, once more reducing
$U_i$ if necessary, inextendible past oriented causal curves starting from points in $U_i$
and in the future of $S^+_i$ intersects $S^+_i$ before escaping from $U_i$. In other words,
inextendible past oriented causal curves in $U_i \cup I^-(S^+_i)$ are also
inextendible causal curves in $\widehat{M}_i$, and therefore, intersect $S_i$.
As required, $S_i$ is a Cauchy surface in $U_i \cup \overline{\varphi_i}(M')$.

Hence, there is a Cauchy embedding of $(M, S)$ into some globally hyperbolic spacetime
$(M'', S'')$, and Cauchy embeddings
$\overline{\varphi}_i: (M'', S'') \to \varphi_i(M') \cup U_i$, which are related by
some isometry $\overline{\varphi}: \varphi_1(M') \cup U_1 \to \varphi_2(M') \cup U_2$:
$$\overline{\varphi}_2 = \overline{\varphi} \circ \overline{\varphi}_1$$

It is a contradiction with the maximality of $(M', S')$. Hence, we have proved that
$\widehat{M}$ is Hausdorff. It is a manifold, and the singular AdS metrics on $\widehat{M}_1$, $\widehat{M}_2$ induce a singular AdS metric on $\widehat{M}$.
Observe that $S_1$ and $S_2$ projects in $\widehat{M}$ onto the same space-like surface
$\widehat{S}$. Let $l$ be any inextendible curve in $\widehat{M}$. Without loss of generality, we can assume that $l$ intersects the projection $W_1$ of $\widehat{M}_1$ in $\widehat{M}$.
Then every connected component of $l \cap W_1$ is an inextendible causal curve in
$W_1 \approx \widehat{M}_1$. It follows that $l$ intersects $\widehat{S}$. Finally,
if some causal curve links two points in $\widehat{S}$, then it must be contained
in $W_1$ since globally hyperbolic open subsets are causally convex. It would contradict
the acausality of $S_1$ inside $\widehat{M}_1$.

The conclusion is that $\widehat{M}$ is globally hyperbolic, and that $\widehat{S}$
is a Cauchy surface in $\widehat{M}$. In other words, the projection of $\widehat{M}_i$
into $\widehat{M}$ is a Cauchy embedding. Since $\widehat{M}_i$ is a maximal extension,
these projections are onto. Hence $\widehat{M}_1$ and $\widehat{M}_2$ are isometric.
\end{proof}

\begin{remark}
The uniqueness of the maximal globally hyperbolic AdS extension is no longer true if
we allow interactions. Indeed, in the next section we will see how,
given some singular AdS spacetime without interaction, to define a surgery near a point in a singular line,
introducing some collision or interaction at this point. The place where such a surgery
can be performed is arbitrary.

However, the uniqueness of the maximal globally hyperbolic extension holds in the case of interactions,
if one stipulates than no new interactions can be introduced.
The point is to consider the maximal extension in the future of a Cauchy surface
in the future of all interactions, and the maximal extension in the
past of a Cauchy surface contained in the past of all interactions.
This point, along with other aspects of the global geometry of moduli spaces of
AdS manifolds with interacting particles, is further studied in \cite{colII}.
\end{remark} 

\section{Global examples}
\label{sc:examples}

The main goal of this section is to construct examples of globally hyperbolic
AdS manifolds with interacting particles, so we go beyond the local examples 
constructed in Section 2.

\subsection{An explicit example} \label{ssc:explicit}

Let $S$ be a hyperbolic surface with one cone point $p$ of angle
$\theta$.   Denote by $\mu$ the corresponding
singular hyperbolic metric on $S$.

Let us consider the Lorentzian metric on $S\times (-\pi/2,\pi/2)$ given by
\begin{equation}\label{ads:eq}
   h= -{\rm d}t^2\ + \cos^2t \ \mu
 \end{equation}
 where $t$ is the real parameter of the interval $(-\pi/2,\pi/2)$.

 We denote by $M(S)$ the singular spacetime $(S\times(-\pi/2,\pi/2),h)$.
 \begin{lemma}
 $M(S)$ is an $AdS$ spacetime with a particle corresponding to the
   singular line $\{p\}\times(-\pi/2,\pi/2)$.  The corresponding cone
   angle is $\theta$. Level surfaces $S\times\{t\}$ are orthogonal to
   the singular locus.
 \end{lemma}
\begin{proof}
First we show that $h$ is an $AdS$ metric. The computation is local, so
we can assume $S=\mathbb H^2$. Thus we can identify $S$ to a geodesic
plane in $AdS_3$. We consider $AdS_3$ as embedded in $\R^{2,2}$, as mentioned in the introduction.
Let $n$ be the normal direction to $S$ then we can
consider the normal evolution
\[
      F:S\times (-\pi/2,\pi/2)\ni(x,t)\mapsto \cos t x +\sin t n\in AdS_3\,.
\]
The map $F$ is a diffeomorphism onto an open domain of $AdS_3$ and the
pull-back of the $AdS_3$-metric takes the form (\ref{ads:eq}).

To prove that $\{p\}\times (-\pi/2,\pi/2)$ is a conical singularity of
angle $\theta$, take a geodesic plane $P$ in $\mathcal P_\theta$
orthogonal to the singular locus. Notice that $P$ has exactly one cone
point $p_0$ corresponding to the intersection of $P$ with the singular
line of $\mathcal P_\theta$ (here $\mathcal P_\theta$ is the singular model space
defined in Subsection \ref{sub:desgeom}). 
Since the statement is local, it is
sufficient to prove it for $P$. Notice that the normal evolution of
$P\setminus\{p_0\}$ is well-defined for any $t\in (-\pi/2,\pi/2)$. Moreover, such
evolution can be extended to a map on the whole
$P\times(-\pi/2,\pi/2)$ sending $\{p_0\}\times (-\pi/2,\pi/2)$ onto
the singular line. This map is a diffeomorphism of $P\times
(-\pi/2,\pi/2)$ with an open domain of $\mathcal P_\theta$. Since the
pull-back of the $AdS$-metric of $\mathcal P_\theta$ on
$(P\setminus\{p_0\})\times(-\pi/2,\pi/2)$ takes the form
(\ref{ads:eq}) the statement follows.
\end{proof}

Let $T$ be a triangle in $HS^2$, with one vertex in the future hyperbolic region and
two vertices in the past hyperbolic region. Doubling $T$, we obtain
 a causally regular HS-sphere $\Sigma$ with an elliptic future
singularity at $p$  and two elliptic past
singularities, $q_1,q_2$.

Let $r$ be the future singular ray in $e(\Sigma)$. For a given
$\epsilon>0$ let $p_\epsilon$ be the point at distance $\epsilon$ from
the interaction point. Consider the geodesic disk $D_\epsilon$ in
$e(\Sigma)$ centered at $p_\epsilon$, orthogonal to $r$ and with
radius $\epsilon$.

The past normal evolution $n_t:D_\epsilon\rightarrow e(\Sigma)$ is
well-defined for $t\leq\epsilon$.  In fact, if we restrict to the
annulus $A_\epsilon=D_\epsilon\setminus D_{\epsilon/2}$, the evolution
can be extended for $t\leq\epsilon'$ for some $\epsilon'>\epsilon$.

\begin{figure}
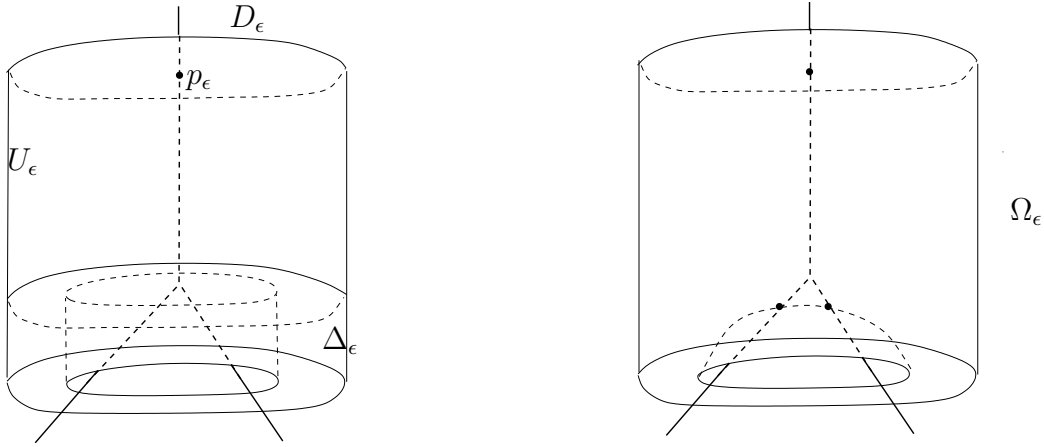

\begin{center}
\input exa.pstex_t
\end{center}
\caption{Construction of a singular tube with an interaction of two particles.}
\end{figure}

Let us set
\[
\begin{array}{l}
   U_\epsilon=\{n_t(p)|p\in D_\epsilon,
   t\in(0,\epsilon)\}\,,\\ 
   \Delta_\epsilon=\{n_t(p)~|~p\in
   D_\epsilon\setminus D_{\epsilon/2}, t\in (0,\epsilon')\}\,.
   \end{array}
\]
Notice that the interaction point is in the closure of $U_\epsilon$.
It is possible to contruct a neighborhood $\Omega_\epsilon$ of the
interaction point $p_0$ such that
\begin{itemize}
\item
$U_\epsilon\cup\Delta_\epsilon\subset\Omega_\epsilon\subset U_\epsilon\cup\Delta_\epsilon\cup B(p)$ where $B(p_0)$ is a small ball around $p_0$;
\item
$\Omega_\epsilon$ admits a foliation in achronal disks
  $(D(t))_{t\in(0,\epsilon')}$ such that
\begin{enumerate}
\item
 $D(t)=n_t(D_\epsilon)$ for $t\leq\epsilon$
\item
  $D(t)\cap\Delta_t=n_t(D_\epsilon\setminus D_{\epsilon/2})$ for
  $t\in(0,\epsilon')$
 \item
  $D(t)$ is orthogonal to the singular locus.
  \end{enumerate}
  \end{itemize}

Consider now the space $M(S)$ as in the previous lemma.
For small $\epsilon$ the disk $D_\epsilon$ embeds in $M(S)$, sending
$p_\epsilon$ to $(p,0)$.

Let us identify $D_\epsilon$ with its image in $M(S)$. The normal
evolution on $D_\epsilon$ in $M(S)$ is well-defined for $0<t<\pi/2$ and
in fact coincides with the map
\[
   n_t(x,0)=(x,t)\,.
\]
It follows that the map
 \[
    F:(D_\epsilon\setminus D_{\epsilon/2})\times
    (0,\epsilon')\rightarrow \Delta_\epsilon
 \]
 defined by $F(x,t)=n_t(x)$ is an isometry.

\begin{figure}[ht]
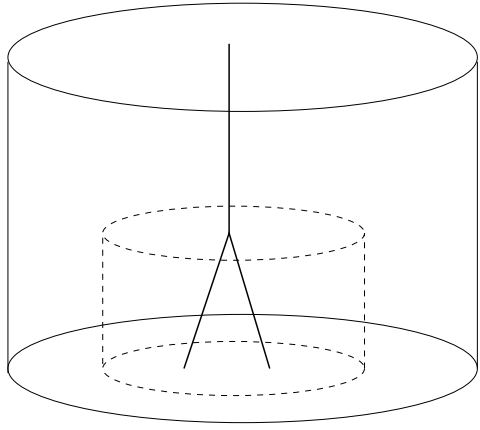

\input interaction.pstex_t
\caption{Surgery to add a collision.}
\label{fig.collision}
\end{figure}

 Thus if we glue $(S\setminus D_{\epsilon/2})\times
 (0,\epsilon')$ to $\Omega_\epsilon$ by identifying
 $D_{\epsilon}\setminus D_{\epsilon/2}$ to $\Delta_\epsilon$ via $F$
 we get a spacetime
 \[
   \hat M=(S\setminus D_{\epsilon/2})\times(0,\epsilon')\cup_F
   \Omega_\epsilon
 \]
  such that
  \begin{enumerate}
  \item
  topologically, $\hat M$ is homeomorphic to $S\times\mathbb R$,
  \item
  in $\hat M$, two particles collide producing one particle only,
  \item
  $\hat M$ admits a foliation by spacelike surfaces orthogonal to the
    singular locus.
  \end{enumerate}

  We say that $\hat M$ is obtained by a surgery on $M'=S\times(0,\epsilon')$.

\subsection{Surgery}\label{surg:sec}

In this section we get a generalization of the
construction explained in the previous section.  In particular we
show how to do a surgery on a spacetime with conical singularity
in order to obtain a spacetime with collision more complicated
than that described in the previous section.

\begin{lemma}\label{surgery-key}
Let $\Sigma$ be a causally regular HS-sphere containing only elliptic
singularities.
Suppose that the circle of photons $C_+$ bounding the future hyperbolic part
of $\Sigma$ carries an elliptic structure of angle $\theta$.
Then $e(\Sigma)\setminus (I^+(p_0)\cup I^-(p_0))$ embeds in $\mathcal P_\theta$
($p_0$ denotes the interaction point of $e(\Sigma)$).
\end{lemma}
\begin{proof}

Let $D$ be the de Sitter part of $\Sigma$,
Notice that
 \[e(D)=e(\Sigma)\setminus (I^+(p_0)\cup I^-(p_0))\,.
 \]
To prove that $e(D)$ embeds in $\mathcal P_\theta$ it is
sufficient to prove that $D$ is isometric to the de Sitter part of the
HS sphere $\Sigma_\theta$ that is the link of a singular point of
$\mathcal P_\theta$. Such de Sitter surface is the quotient of
$\tilde{dS}_2$ under an elliptic transformation of $\tilde{SO}(2,1)$
of angle $\theta$.

So the statement is equivalent to proving that  the developing map
\[
  d:\tilde D\rightarrow \tilde{dS_2}
\]
is a diffeomorphism. Since $\tilde{dS_2}$ is simply connected and $d$ is a local
diffeomorphism, it is sufficient to prove that $d$ is proper.

As in Section~\ref{sec.classificationHS}, $\tilde {dS}_2$ can be
completed by two lines of photons, say $R_+, R_-$ that are
projectively isomorphic to $\tilde{\mathbb R\mathbb P^1}$.

Consider the left isotropic foliation of $\tilde {dS}_2$. Each leaf 
 has an $\alpha$-limit in $R_-$ and an $\omega$-limit on
$R_+$. Moreover every point of $R_-$ (resp. $R_+$) is an
$\alpha$-limit (resp. $\omega$-limit) of exactly one leaf of each
foliation.  Thus we have a continuous projection 
$\iota_L:\tilde{dS_2}\cup R_-\cup R_+\rightarrow R_+$, obtained by sending a point $x$
to the $\omega$-limit of the leaf of the left foliation trough it. The map
$\iota_L$ is a proper submersion.
Since $D$ does not contain singularities, we have an analogous  proper submersion
\[
  \iota'_L:\tilde D\cup\tilde C_-\cup\tilde C_+\rightarrow\tilde C_+\,,  
\]
where $\tilde C_+$, $\tilde C_-$ are the universal covering of the circle of photons of $\Sigma$.

By the naturality of the construction, the following diagram commutes
\[
\begin{CD}
\tilde D\cup\tilde C_-\cup\tilde C@>d>>\tilde{dS_2}\cup R_-\cup R_+\\
@V\iota'_L VV    @V\iota_L VV\\
\tilde C_+ @>d>>\tilde R_+\,.
\end{CD}
\]
The map $d|_{\tilde C_+}$ is the developing map for the projective structure
of $C_+$. By the hypothesis, we have that $d|_{\tilde C_+}$ is a homeomorphism,
so it is proper.

Since the diagram is commutative and the fact that $\iota_L$ and $\iota'_L$ are both proper
one easily proves that $d$ is proper.
\end{proof}

\begin{remark}
If $\Sigma$ is a causally regular HS-sphere containing only elliptic singularities,
the map $\iota'_L:\tilde C_-\rightarrow\tilde C_+$ induces a projective
isomorphism $\bar\iota:C_-\rightarrow C_+$.
\end{remark}

\begin{defi}
Let $M$ be a singular spacetime homeomorphic to $S\times\mathbb R$
and let $p\in M$. A neighborhood $U$ of $p$ is said to be {\bf cylindrical} if
\begin{itemize}
\item $U$ is topologically a ball;
\item $\partial_\pm C:=\partial U\cap I^\pm(p)$ is a spacelike disk;
\item there are two disjoint closed spacelike slices $S_-, S_+$ homeomorphic to $S$
  such that $S_-\subset I^-(S_+)$ and $I^\pm(p)\cap S_\pm=
  \partial_\pm C$.
\end{itemize}
  \end{defi}

\begin{remark}
\-
\begin{itemize}
\item
If a spacelike slice through $p$ exists then
cylindrical neighborhoods form a fundamental family of neighborhoods.
\item
There is an open retract $M'$ of $M$ whose boundary is $S_-\cup S_+$.
\end{itemize}
\end{remark}

\begin{cor}
Let $\Sigma$ be a HS-sphere as in Lemma \ref{surgery-key}.  Given an
$AdS$ spacetime $M$ homeomorphic to $S\times\mathbb R$ containing a
particle of angle $\theta$, let us fix a point $p$ on it and suppose
that a spacelike slice through $p$ exists. There is a cylindrical
neighborhood $C$ of $p$ and a cylindrical neighborhood $C_0$ of the
interaction point $p_0$ in $e(\Sigma)$ such that $C\setminus (
I^+(p)\cup I^-(p))$ is isometric to $C_0\setminus(I^+(p_0)\cup
I^-(p_0))$.
\end{cor}

Take an open deformation retract $M'\subset M$ with spacelike boundary
such that $\partial_\pm C\subset\partial M'$.  Thus let us glue
$M'\setminus ( I^+(p)\cup I^-(p))$ and $C_0$ by identifying
$C\setminus (I^+(p)\cup I^-(p))$ to $C_0\cap e(D)$. In this way we get
a spacetime $\hat M$ homeomorphic to $\Sigma\times\mathbb R$ with an
interaction point modelled on $e(\Sigma)$. We say that $\hat M$ is
obtained by a surgery on $M'$.

The following proposition is a kind of converse to the previous construction.

\begin{prop} \label{pr:example}
Let $\hat M$ be a spacetime with conical singularities homeomorphic to
$S\times\mathbb R$ containing only one interaction between
particles. Suppose moreover that a neighborhood of the interaction point is isometric to
an open subset in 
$e(\Sigma)$, where $\Sigma$ is a HS-surface as in
Lemma~\ref{surgery-key}.  Then a subset of $M$ is obtained by
a surgery on a spacetime without interaction.
\end{prop}

\begin{proof}
Let $p_0$ be the interaction point. There is an HS-sphere $\Sigma$ as
in Lemma~\ref{surgery-key} such that a neighborhood of $p_0$ is
isometric to a neighborhood of the vertex of $e(\Sigma)$. In
particular there is a small cylindrical neighborhood $C_0$ around
$p_0$. According to Lemma~\ref{surgery-key}, for a suitable cylindrical
neighborhood $C$ of a singular point $p$ in $\mathcal P_\theta$ we
have
\[
    C\setminus (I^+(p)\cup I^-(p))\cong C_0\setminus(I^+(p_0)\cup I^-(p_0))
 \]
 Taking the retract $M'$ of $\hat M$ such that $\partial_\pm C_0$ is in the
 boundary of $M'$, the spacetime $M'\setminus(I^+(p_0)\cup I^-(p_0))$
 can be glued to $C$ via the above identification. We get a spacetime
 $M^0$ with only one singular line. Clearly the surgery on $M^0$ of
 $C_0$ produces $M'$.
  \end{proof}

  \subsection{Spacetimes containing BTZ-type singularities}
  \label{sub.bhfrancesco}

  In this section we describe a class of spacetimes containing
BTZ-type singularities.

We use the projective model of $AdS$ geometry, that is the $AdS_{3,+}$.
From Subsection \ref{sub:backgroundads}, $AdS_{3,+}$ is a domain in $\mathbb R\mathbb P^3$ bounded
by the quadric $Q$ of signature $(1,1)$.  Using the double family of lines $\mathcal L_l,\mathcal L_r$
we identify $Q$ to $\mathbb R\mathbb P^1\times\mathbb R\mathbb P^1$ so that the isometric
action of $\isomz=PSL(2,\mathbb R)\times PSL(2,\mathbb R)$ on $AdS_3$ extends to the product action
on the boundary.

%

We have seen in Section \ref{sub:backgroundads} that gedesics of $AdS_{3,+}$ 
are projective segments  whereas geodesics planes are the intersection of
$AdS_{3,+}$ with projective planes.The scalar product of
$\mathbb R^{2,2}$ induces a duality between points and projective planes
and between projective lines. In particular points in $AdS_3$ are dual to spacelike planes
and the dual of a  spacelike geodesic is still a spacelike geodesics. Geometrically,
every timelike geodesic starting from a point $p\in AdS_3$ orthogonally meets
the dual plane at time $\pi/2$, and points on the dual plane can be characterized
by the poperty to be connected to $p$ be a timelike geodesic of length $\pi/2$.
Analogously, the dual line of a line $l$ is the set of points that be can be connected to every point
of $l$ by a timelike geodesic of length $\pi/2$.

Now, consider two hyperbolic transformations $\gamma_1,\gamma_2\in
PSL(2,\mathbb R)$ with the same translation length. There are
exactly $2$ spacelike geodesics $l_1,l_2$ in $AdS_3$ that are 
invariant under the action of $(\gamma_1,\gamma_2)\in PSL(2,\mathbb
R)\times PSL(2,\mathbb R)=\isomz$.
Namely, if $x^+(c)$ denotes the attractive fixed point of a hyperbolic transformation $c\in PSL(2,\mathbb R)$,   
$l_2$ is the line in $AdS_3$ joining the boundary points
$(x^+(\gamma_1), x^+(\gamma_2))$ and $(x^+(\gamma_1^{-1}), x^+(\gamma_2^{-1}))$.
On the other hand $l_1$ is the geodesic dual to $l_2$, the endpoints of $l_1$ 
are $(x^+(\gamma_1),x^+(\gamma_2^{-1}))$ and $(x^+(\gamma_1^{-1}),x^+(\gamma_2))$.

  Points of $l_1$ are fixed by
  $(\gamma_1,\gamma_2)$ whereas it acts by pure translation on $l_2$. 
The union of the timelike segments with past end-point on $l_2$ and
  future end-point on $l_1$ is a domain $\Omega_0$ in $AdS_{3,+}$
  invariant under $(\gamma_1,\gamma_2)$.  The action of
  $(\gamma_1,\gamma_2)$ on $\Omega_0$ is proper and free and the
  quotient $M_0(\gamma_1,\gamma_2)=\Omega_0/(\gamma_1,\gamma_2)$ is a
  spacetime homeomorphic to $S^1\times\mathbb R^2$.


  There exists a spacetime with singularities 
  $\hat M_0(\gamma_1,\gamma_2)$ such that $M_0(\gamma_1,\gamma_2)$ is
  isometric to the regular part of $\hat M_0(\gamma_1,\gamma_2)$ and
  it contains a future BTZ-type singularity.
 Define
 \[
 \hat M_0(\gamma_1,\gamma_2)= (\Omega_0\cup l_1)/(\gamma_1,\gamma_2)
 \]

 To show that $l_1$ is a future BTZ-type singularity, let us
 consider an alternative description of $\hat M_0(\gamma_1,\gamma_2)$.
 Notice that a fundamental domain in $\Omega_0\cup l_1$ for the action
 of $(\gamma_1,\gamma_2)$ can be constructed as follows. Take on $l_2$
 a point $z_0$ and put $z_1=(\gamma_1,\gamma_2)z_0$.  Then consider
 the domain $P$ that is the union of timelike geodesic joining a point
 on the segment $[z_0,z_1]\subset l_2$ to a point on $l_1$. $P$ is
 clearly a fundamental domain for the action with two timelike
 faces. $\hat M_0(\gamma_1,\gamma_2)$ is obtained by gluing the faces
 of $P$. 

\begin{figure}
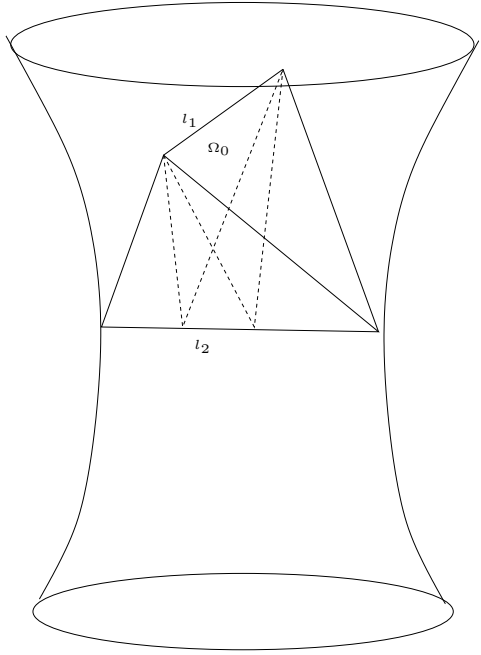

\begin{center}
\input btz.pstex_t
\caption{The region P is bounded by the dotted triangles, whereas $M_0(\gamma_1,\gamma_2)$ is obtained
by gluing the faces of P}
\end{center}
\end{figure}

We now generalize the above constructions as follows.
Let us fix a surface $S$ with some boundary component and
negative Euler characteristic. Consider on $S$ two
hyperbolic metrics $\mu_l$ and $\mu_r$ with geodesic boundary such
that each boundary component has the same length with respect to those
metrics.

Let $h_l,h_r:\pi_1(S)\rightarrow PSL(2,\mathbb R)$ be the
corresponding holonomy representations.  The pair
$(h_l,h_r):\pi_1(S)\rightarrow PSL(2,\mathbb R)\times
PSL(2,\mathbb R)$ induces an isometric action of $\pi_1(S)$ on
$AdS_3$.

In \cite{barbtz1,barbtz2, bks} it is proved that there exists a convex
domain $\Omega$ in $\AdS_{3,+}$ invariant under the action of
$\pi_1(S)$ and the quotient $M=\Omega/\pi_1(\Sigma)$ is a
strongly causal manifold homeomorphic to $S\times\mathbb R$. 
For the convenience of the reader we sketch the construction of $\Omega$
referring to \cite{barbtz1, barbtz2} for details.

The domain $\Omega$ can be defined as follows.
 First consider the \emph{limit set} $\Lambda$ defined as the closure of the set
 of pairs $(x^+(h_l(\gamma)), x^+(h_r(\gamma))$ for $\gamma\in\pi_1(S)$.
$\Lambda$ is a $\pi_1(S)$-invariant 
subset of $\partial AdS_{3,+}$ and it turns out that there exists a spacelike
plane $P$ disjoint from $\Lambda$. So we can consider the convex hull $K$ of $\Lambda$ in 
the affine chart $\mathbb R\mathbb P^3\setminus P$.

$K$ is a convex subset contained in $AdS_{3,+}$. 
For any peripheral loop $\gamma$, the spacelike geodesic $c_\gamma$ joining $(x^+(h_l(\gamma^{-1})), x^+(h_r(\gamma^{-1})))$ to
 $(x^+(h_l(\gamma)), x^+(h_r(\gamma)))$ is contained in $\partial K$ and
 $\Lambda\cup \bigcup c_\gamma$ disconnects $\partial K$ into components called the future boundary,  $\partial_+ K$, and the past boundary, $\partial_-K$.
 
 One then defines $\Omega$ as the set of points whose dual plane is disjoint from $K$.
 We have
 \begin{enumerate}
 \item the interior of $K$ is contained in $\Omega$.
 \item $\partial\Omega$ is the set of points whose dual plane is a support plane for $K$.
 \item $\partial\Omega$ has two components: the past and the future boundary.
Points dual to support planes of $\partial_-K$ are contained in the future boundary of $\Omega$,
whereas points dual to support planes of $\partial_+K$ are contained in the past boundary 
of $\Omega$.
\item
Let $\mathcal A$ be the set of triples $(x,v,t)$, where $t\in[0,\pi/2]$, $x\in\partial_-K$ and
$v\in\partial_+\Omega$ is a point dual to  some support plane of $K$ at $x$. We consider
the normal evolution map $\Phi:\mathcal A\rightarrow AdS_{3,+}$, where $\Phi(x,v,t)$ is the point 
on the geodesic segment joining $x$ to $v$ at distance $t$ from $x$. In \cite{bb_wick} the map $\Phi$
is shown to be injective.
\end{enumerate}

\begin{figure}
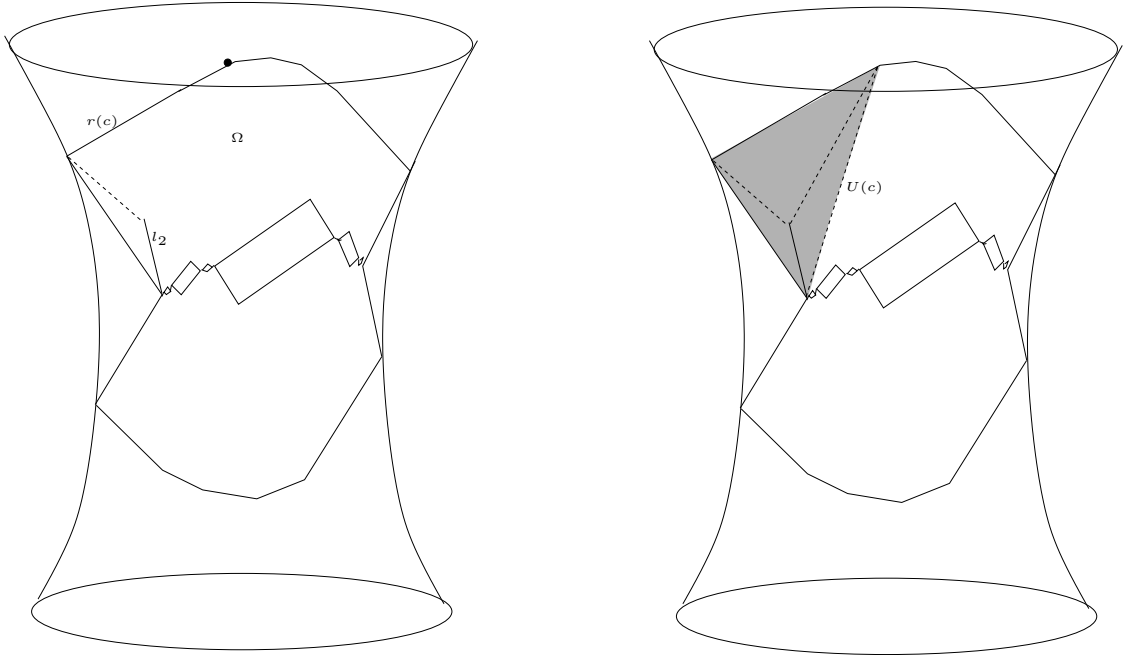

\begin{center}
\input multi.pstex_t
\end{center}
\caption{The segment $r(c)$ projects to a BTZ-type singularity for $M$.}
\end{figure}

\begin{prop}\label{pp}
There exists a manifold with singularities $\hat M$ such that
\begin{enumerate}
\item The regular part of $\hat M$ is $M$.
\item There is a future BTZ-type singularity and a past BTZ-type
  singularity for each boundary component of $M$.
\end{enumerate}
 \end{prop}
 \begin{proof}

Let $c\in\pi_1(S)$ be a loop representing a
   boundary component of $S$ and let $\gamma_1=h_l(c)$,
   $\gamma_2=h_r(c)$.

 By hypothesis, the translation lengths of $\gamma_1$ and $\gamma_2$
 are equal, so, as in the previous example, there are two invariant 
 geodesics $l_1$ and $l_2$. Moreover
 the geodesic $l_2$ is contained in $\Omega$ and is in the boundary of
 the convex core $K$ of $\Omega$. By~\cite{bks,bb_hand}, there exists
 a face $F$ of the past boundary of $K$ that contains $l_2$. The dual
 point of such a face, say $p$, lies in $l_1$. Moreover a component of
 $l_1\setminus\{p\}$ contains points dual to some support planes of
 the convex core containing $l_2$. Thus there is a ray $r=r(c)$ in
 $l_1$ with vertex at $p$ contained in $\partial_+\Omega$ (and
 similarly there is a ray $r_-=r_-(c)$ contained in
 $l_1\cap\partial_-\Omega$).

 Now let $U(c)$ be the union of timelike segments in $\Omega$ with past
 end-point in $l_2$ and future end-point in $r(c)$. Clearly
 $U(c)\subset\Omega(\gamma_1,\gamma_2)$.
The stabilizer of $U(c)$ in $\pi_1(S)$ is the group
 generated by $(\gamma_1,\gamma_2)$. Moreover we have
 \begin{itemize}
 \item
 for some $a\in\pi_1(S)$ we have $a\cdot U(c)=U(aca^{-1})$,
 \item
 if $d$ is another peripheral loop, $U(c)\cap U(d)=\emptyset$.
 \end{itemize}
(The last property is a consequence of the fact that  the normal evolution
of $\partial_-K$ is injective -- see property (4) before Proposition \ref{pp}.)

 So if we  put
 \[
 \hat M=(\Omega\cup\bigcup r(c)\cup\bigcup r_-(c))/\pi_1(S)
 \]
 then a neighborhood of $r(c)$ in $\hat M$ is isometric to a neighborhood
 of $l_1$ in $M(\gamma_1,\gamma_2)$, and is thus a BTZ-type
 singularity (and analogously $r_-(c)$ is a white hole singularity).
 \end{proof}


 \subsection{Surgery on spacetimes containing BTZ-type singularities}

 Now we illustrate how to get spacetimes $\cong S\times\mathbb R$
 containing two particles that collide producing a BTZ-type
 singularity.  Such examples are obtained by a surgery operation
 similar to that implemented in Section~\ref{surg:sec}. The main
 difference with that case is that the boundary of these spacetimes is
 not spacelike.

 Let $M$ be a spacetime $\cong S\times\mathbb R$ containing a
 BTZ-type singularity $l$ of mass $m$ and fix a point $p\in l$. Let
 us consider a HS-surface $\Sigma$ containing a BTZ-type singularity
 $p_0$ of mass $m$ and two elliptic singularities $q_1,q_2$. A small
 disk $\Delta_0$ around $p_0$ is isomorphic to a small disk $\Delta$
 in the link of the point $p\in l$. (As in the previous section, one can construct
 such a surface by doubling a traingle in $HS^2$ with one vertex in the de Sitter region and
 two vertices in the past hyperbolic region.)

 Let $B$ be a ball around $p$ and $B_\Delta$ be the intersection of
 $B$ with the union of segments starting from $p$ with velocity in
 $\Delta$.  Clearly $B_\Delta$ embeds in $e(\Sigma)$, moreover there
 exists a small disk $\Delta_0$ around the vertex of $e(\Sigma)$ such that
 $e(\Delta_0)\cap B_0$ is isometric to the image of $B_\Delta$ in $B_0$.

 Now $\Delta'=\partial B\setminus B_\Delta$ is a disk in
 $M$. So there exists a topological surface $S_0$ in $M$ such that
 \begin{itemize}
 \item $S_0$ contains $\Delta'$;
 \item $S_0\cap B=\varnothing$;
 \item $M\setminus S_0$ is the union of two copies of $S\times\mathbb R$.
 \end{itemize}
 Notice that we do not require $S_0$ to be spacelike.

 Let $M_1$ be the component of $M\setminus S_0$ that contains
 $B$.  Consider the spacetime $\hat M$ obtained by gluing
 $M_1\setminus(B\setminus B_\Delta)$ to $B_0$ identifying $B_\Delta$
 to its image in $B_0$.
Clearly $\hat M$ contains two particles that collide giving a BH
singularity and topologically $\hat M\cong S\times\mathbb R$.

\bibliography{bibliocollision}

\newcommand{\etalchar}[1]{$^{#1}$}
\begin{thebibliography}{ABB{\etalchar{+}}07}

\bibitem[ABB{\etalchar{+}}07]{mess-notes}
T.~Andersson, T.~Barbot, R.~Benedetti, F.~Bonsante, W.M. Goldman, F.~Labourie,
  K.P. Scannell, and J.M. Schlenker.
\newblock {Notes on a paper of Mess}.
\newblock {\em Geom. Dedicata}, 126:47--70, 2007.

\bibitem[Ale05]{alex}
A.~D. Alexandrov.
\newblock {\em Convex polyhedra}.
\newblock Springer Monographs in Mathematics. Springer-Verlag, Berlin, 2005.
\newblock Translated from the 1950 Russian edition by N. S. Dairbekov, S. S.
  Kutateladze and A. B. Sossinsky, With comments and bibliography by V. A.
  Zalgaller and appendices by L. A. Shor and Yu. A. Volkov.

\bibitem[Bar08a]{barbtz1}
Thierry Barbot.
\newblock Causal properties of {A}d{S}-isometry groups. {I}. {C}ausal actions
  and limit sets.
\newblock {\em Adv. Theor. Math. Phys.}, 12(1):1--66, 2008.

\bibitem[Bar08b]{barbtz2}
Thierry Barbot.
\newblock Causal properties of {A}d{S}-isometry groups. {II}. {BTZ}
  multi-black-holes.
\newblock {\em Adv. Theor. Math. Phys.}, 12(6):1209--1257, 2008.

\bibitem[BB09a]{bb_hand}
Riccardo Benedetti and Francesco Bonsante.
\newblock {$(2+1)$} {E}instein spacetimes of finite type.
\newblock In {\em Handbook of {T}eichm\"uller theory. {V}ol. {II}}, volume~13
  of {\em IRMA Lect. Math. Theor. Phys.}, pages 533--609. Eur. Math. Soc.,
  Z\"urich, 2009.

\bibitem[BB09b]{bb_wick}
Riccardo Benedetti and Francesco Bonsante.
\newblock Canonical {Wick} rotations in 3-dimensional gravity.
\newblock {\em Memoirs of the American Mathematical Society}, 198:164pp, 2009.
\newblock math.DG/0508485.

\bibitem[BBES03]{brock-bromberg-evans-souto}
Jeffrey Brock, Kenneth Bromberg, Richard Evans, and Juan Souto.
\newblock Tameness on the boundary and {A}hlfors' measure conjecture.
\newblock {\em Publ. Math. Inst. Hautes \'Etudes Sci.}, (98):145--166, 2003.

\bibitem[BBS09]{collision}
Thierry Barbot, Francesco Bonsante, and Jean-Marc Schlenker.
\newblock Collisions of particles in locally {AdS} spacetimes.
\newblock arXiv:0905.1823., 2009.

\bibitem[BBS10]{colII}
Thierry Barbot, Francesco Bonsante, and Jean-Marc Schlenker.
\newblock Collisions of particles in locally ads spacetimes ii. moduli of
  globally hyperbolic spaces.
\newblock Work in progress, 2010.

\bibitem[BEE96]{beem}
J.K. Beem, P.E. Ehrlich, and K.L. Easley.
\newblock {\em {Global Lorentzian Geometry}}.
\newblock Marcel Dekker, 1996.

\bibitem[BKS06]{bks}
F.~Bonsante, K.~Krasnov, and J.-M. Schlenker.
\newblock Multi black holes and earthquakes on {Riemann} surfaces with
  boundaries.
\newblock {arXiv:math/0610429}, 2006.

\bibitem[Bro04]{bromberg1}
K.~Bromberg.
\newblock Hyperbolic cone-manifolds, short geodesics, and {S}chwarzian
  derivatives.
\newblock {\em J. Amer. Math. Soc.}, 17(4):783--826 (electronic), 2004.

\bibitem[BS09a]{cone}
Francesco Bonsante and Jean-Marc Schlenker.
\newblock Ad{S} manifolds with particles and earthquakes on singular surfaces.
\newblock {\em Geom. Funct. Anal.}, 19(1):41--82, 2009.

\bibitem[BS09b]{earthquakes}
Francesco Bonsante and Jean-Marc Schlenker.
\newblock Fixed points of compositions of earthquakes.
\newblock arXiv:0812.3471, 2009.

\bibitem[BS09c]{maximal}
Francesco Bonsante and Jean-Marc Schlenker.
\newblock Maximal surfaces and the universal {Teichm{\"u}ller} space.
\newblock arXiv:0911.4124. To appear, {\it Inventiones Mathematicae}, 2009.

\bibitem[BTZ92]{BTZ}
M.~Ba{\~n}ados, C.~Teitelboim, and J.~Zanelli.
\newblock {Black hole in three-dimensional spacetime}.
\newblock {\em Physical Review Letters}, 69(13):1849--1851, 1992.

\bibitem[Car03]{carlip}
S.~Carlip.
\newblock {\em {Quantum Gravity in 2+ 1 Dimensions}}.
\newblock Cambridge University Press, 2003.

\bibitem[CFGO94]{carroll}
S.~M. Carroll, E.~Farhi, A.~H. Guth, and K.~D. Olum.
\newblock Energy-momentum restrictions on the creation of gott time machines.
\newblock {\em Phys. Rev. D}, 50:6190--6206, 1994.

\bibitem[DS93]{deser}
S.~{Deser} and A.~R. {Steif}.
\newblock {No Time Machines from Lightlike Sources in 2 + 1 Gravity}.
\newblock In {B.~L.~Hu, M.~P.~Ryan Jr., \& C.~V.~Vishveshwara}, editor, {\em
  Directions in General Relativity: Papers in Honor of Charles Misner, Volume
  1}, pages 78--+, 1993.

\bibitem[Ger70]{gerochdependence}
R.~Geroch.
\newblock Domain of dependence.
\newblock {\em J. Math. Phys.}, 11(2):437--449, 1970.

\bibitem[GL98]{gottli}
J.~R. Gott and Li-Xin Li.
\newblock Can the universe create itself?
\newblock {\em Phys. Rev. Lett.}, 58(2), 1998.

\bibitem[Gol10]{goldmanexp}
W.M. Goldman.
\newblock Locally homogeneous geometric manifolds.
\newblock Proceedings of the International Congress of Mathematicians
  Hyderabad, India, 2010; arXiv:1003.2759, 2010.

\bibitem[Got91]{gott}
J.~R. Gott.
\newblock Closed timelike curves produced by pairs of moving cosmic strings:
  exact solutions.
\newblock {\em Phys. Rev. Lett.}, 66:1126--1129, 1991.

\bibitem[Gra93]{grant}
J.~D.E. Grant.
\newblock Cosmic strings and chronology protection.
\newblock {\em Phys. Rev. D}, 47, 1993.

\bibitem[HK98]{HK}
Craig~D. Hodgson and Steven~P. Kerckhoff.
\newblock Rigidity of hyperbolic cone-manifolds and hyperbolic {Dehn} surgery.
\newblock {\em J. Differential Geom.}, 48:1--60, 1998.

\bibitem[HM99]{matschull}
S.~Holst and H.J. Matschull.
\newblock The anti-de sitter gott universe: a rotating btz wormhole.
\newblock {\em Class. Quantum Grav.}, 16(10):3095--3131, 1999.

\bibitem[Mes07]{mess}
Geoffrey Mess.
\newblock Lorentz spacetimes of constant curvature.
\newblock {\em Geom. Dedicata}, 126:3--45, 2007.

\bibitem[MM]{mazzeo-montcouquiol}
Rafe Mazzeo and Gr\'{e}goire Montcouquiol.
\newblock Infinitesimal rigidity of cone-manifolds and the {S}toker problem for
  hyperbolic and {E}uclidean polyhedra.
\newblock {\tt arXiv:0908.2981}.

\bibitem[Sch98]{shu}
Jean-Marc Schlenker.
\newblock M\'etriques sur les poly\`edres hyperboliques convexes.
\newblock {\em J. Differential Geom.}, 48(2):323--405, 1998.

\bibitem[Sch01]{cpt}
Jean-Marc Schlenker.
\newblock Convex polyhedra in {L}orentzian space-forms.
\newblock {\em Asian J. Math.}, 5(2):327--363, 2001.

\bibitem[Ste94]{steif}
A.~R. Steif.
\newblock Multiparticle solutions in 2+1 gravity and time machines.
\newblock {\em Int J Mod Phys D}, 3(1):277--280, 1994.

\bibitem[tH93]{thooft2}
G.~'t~Hooft.
\newblock The evolution of gravitating point particles in {$2+1$} dimensions.
\newblock {\em Classical Quantum Gravity}, 10(5):1023--1038, 1993.

\bibitem[tH96]{thooft1}
G.~'t~Hooft.
\newblock Quantization of point particles in {$(2+1)$}-dimensional gravity and
  spacetime discreteness.
\newblock {\em Classical Quantum Gravity}, 13(5):1023--1039, 1996.

\bibitem[Thu98]{thurstonshape}
W.P. Thurston.
\newblock {Shapes of polyhedra and triangulations of the sphere}.
\newblock {\em Geometry and Topology monographs}, 1(1):511--549, 1998.

\bibitem[Wei]{weiss:09}
Hartmut Weiss.
\newblock The deformation theory of hyperbolic cone-3-manifolds with
  cone-angles less than $2\pi$.
\newblock {\tt arXiv:0904.4568}.

\end{thebibliography}
\bibliographystyle{alpha}

\end{document}